\documentclass[10pt]{article}
\usepackage{amsfonts,amsmath,amssymb}
\usepackage[ruled]{algorithm2e}
\usepackage[english]{babel}
\usepackage[autolang=other,backend=bibtex,bibencoding=ascii,maxnames=99,style=trad-plain]{biblatex}
\usepackage{bm}
\usepackage{enumerate}
\usepackage{gensymb}
\usepackage{graphicx}
\usepackage[pdftex]{hyperref}
\usepackage[utf8]{inputenc}
\usepackage{csquotes}
\usepackage{mathrsfs}
\usepackage[version=3]{mhchem}
\usepackage{subcaption}

\hypersetup{
   colorlinks,
   linkcolor=black,
   citecolor=black
}

\providecommand{\abs}[1]{\left\lvert#1\right\rvert}

\providecommand{\norm}[1]{\lVert#1\rVert}

\providecommand{\dnorm}[1]{\left\lVert#1\right\rVert}
\renewcommand{\vec}[1]{\boldsymbol{#1}}

\DeclareMathOperator{\id}{id}

\newtheorem{thrm}{Theorem}[section]
\newtheorem{crllr}[thrm]{Corollary}
\newtheorem{lmm}[thrm]{Lemma}
\newtheorem{prpstn}[thrm]{Proposition}
\newtheorem{rmrk}[thrm]{Remark}

\newenvironment{proof}{\textsc{\noindent Proof. \ignorespaces}}{\nopagebreak\hspace*{\fill}$\square$}

\title{Convergence analysis of adaptive DIIS algorithms with application to electronic ground state calculations}
\author{Maxime Chupin\thanks{CEREMADE, UMR CNRS 7534, Universit\'e Paris-Dauphine, Universit\'e PSL, Place du Mar\'echal de Lattre de Tassigny, 75775 Paris cedex 16, France (\url{chupin@ceremade.dauphine.fr}).}, Mi-Song Dupuy\thanks{Laboratoire Jacques-Louis Lions,UMR CNRS 7598, Sorbonne Université, boîte courrier 187, 75252 Paris Cedex 05, France (\url{dupuymi@ljll.math.upmc.fr})}, Guillaume Legendre\thanks{CEREMADE, UMR CNRS 7534, Universit\'e Paris-Dauphine, Universit\'e PSL, Place du Mar\'echal de Lattre de Tassigny, 75775 Paris cedex 16, France (\url{guillaume.legendre@ceremade.dauphine.fr}).}\ and \'Eric S\'er\'e\thanks{CEREMADE, UMR CNRS 7534, Universit\'e Paris-Dauphine, Universit\'e PSL, Place du Mar\'echal de Lattre de Tassigny, 75775 Paris cedex 16, France (\url{sere@ceremade.dauphine.fr}).}}
\date{\today}

\AtEveryBibitem{\clearlist{language}} 

\apptocmd{\sloppy}{\hbadness 10000\relax}{}{}

\bibliography{diis}

\begin{document}
\maketitle

\begin{abstract}
This paper deals with a general class of algorithms for the solution of fixed-point problems that we refer to as \emph{Anderson--Pulay acceleration}. This family includes the DIIS technique and its variant sometimes called commutator-DIIS, both introduced by Pulay in the 1980s to accelerate the convergence of self-consistent field procedures in quantum chemistry, as well as the related Anderson acceleration which dates back to the 1960s, and the wealth of techniques they have inspired. Such methods aim at accelerating the convergence of any fixed-point iteration method by combining several iterates in order to generate the next one at each step. This extrapolation process is characterised by its \emph{depth}, \textit{i.e.} the number of previous iterates stored, which is a crucial parameter for the efficiency of the method. It is generally fixed to an empirical value. 

In the present work, we consider two parameter-driven mechanisms to let the depth vary along the iterations. In the first one, the depth grows until a certain nondegeneracy condition is no longer satisfied; then the stored iterates (save for the last one) are discarded and the method ``restarts''. In the second one, we adapt the depth continuously by eliminating at each step some of the oldest, less relevant, iterates. In an abstract and general setting, we prove under natural assumptions the local convergence and acceleration of these two adaptive Anderson--Pulay methods, and we show that one can theoretically achieve a superlinear convergence rate with each of them. We then investigate their behaviour in quantum chemistry calculations. These numerical experiments show that both adaptive variants exhibit a faster convergence than a standard fixed-depth scheme, and require on average less computational effort per iteration. This study is complemented by a review of known facts on the DIIS, in particular its link with the Anderson acceleration and some multisecant-type quasi-Newton methods.
\end{abstract}

\section{Introduction}
The \emph{Direct Inversion in the Iterative Subspace} (DIIS) technique, introduced by Pulay \cite{Pulay:1980} and also known as \emph{Pulay mixing}, is a locally convergent method widely used in computational quantum chemistry for accelerating self-consistent field convergence. As a complement to available globally convergent methods, like the \emph{optimal damping algorithm} (ODA) \cite{Cances:2000} or its energy-DIIS (EDIIS) variant \cite{Kudin:2002}, it remains a method of choice, in a large part due to its simplicity and nearly unparalleled performance once a convergence region has been attained. Due to this success, variants of the technique have been proposed over the years in other types of application, like the GDIIS adaptation \cite{Csaszar:1984,Eckert:1997} for geometry optimization or the \textit{residual minimisation method--direct inversion in the iterative subspace} (RMM-DIIS) for the simultaneous computation of eigenvalues and corresponding eigenvectors, attributed to Bendt and Zunger and described in \cite{Wood:1985}. It has also been combined with other schemes to improve the rate of convergence of various types of iterative calculations (see \cite{Kawata:1998} for instance).

\smallskip

From a general point of view, the DIIS can be seen as an acceleration technique based on extrapolation, applicable to any fixed-point iterative scheme for which a measure of the error at each step, in the form of a residual for instance, is (numerically) available. It was recently established that this technique is closely related to an older process known as the \emph{Anderson acceleration} \cite{Anderson:1965}. It was also shown that it amounts to a multisecant-type variant of a Broyden method \cite{Broyden:1965} and that, when applied to linear problems, it is (essentially) equivalent to the \emph{generalised minimal residual} (GMRES) method of Saad and Schultz \cite{Saad:1986}. On this basis, Rohwedder and Schneider \cite{Rohwedder:2011} (for the DIIS), and later Toth and Kelley \cite{Toth:2015} (for the Anderson acceleration), analysed the method in an abstract framework and provided convergence results.

\smallskip

In the present paper, we consider a unified family of methods, which we refer to as \emph{Anderson--Pulay acceleration}, encompassing the DIIS, its commutator-DIIS (CDIIS) variant \cite{Pulay:1982}, and the Anderson acceleration. In such methods, one keeps a ``history'' of previous iterates which are combined to generate the next one at each step with the aim of accelerating the convergence of the sequence of iterates. This extrapolation process is characterised by an integer, sometimes called the \emph{depth} (see \cite{Toth:2015,An:2017,Evans:2020}), which is the number of previous iterates stored and is an important parameter for the efficiency of the method. In most applications, the depth grows up to an empirically fixed value $m$, then remains constant. One of the main conclusions of our work is that it can be beneficial to let the depth vary adaptively along the iterations.

In order to prove this point, we propose and investigate two parameter-driven procedures to determine the depth at each step of the Anderson--Pulay acceleration method. The first one allows the method to ``restart'', based on a condition initially introduced by Gay and Schnabel for a quasi-Newton method using multiple secant equations \cite{Gay:1977} and used by Rohwedder and Schneider in the context of the DIIS \cite{Rohwedder:2011}. In the second one, we adapt the depth continuously, thanks to a very simple criterion which is new, as far as we know.

In a general framework, we mathematically analyse these two adaptive Anderson--Pulay methods, and prove local convergence and acceleration properties. In contrast with preceding works \cite{Rohwedder:2011,Toth:2015}, our results are obtained \emph{without} any assumption on the boundedness of the extrapolation coefficients (which would have to be verified \textit{a posteriori} in practice). Indeed, the built-in mechanism in each of the proposed algorithms prevents linear dependency from occurring in the least-squares problem for the coefficients, allowing us to derive a theoretical \textit{a priori} bound on these coefficients. Applications of both methods to self-consistent field calculations in quantum chemistry and comparisons with their fixed-depth counterparts, demonstrate their good performances, and even suggest that adapting continuously the depth gives the fastest convergence in practice. 

\smallskip

The paper is organized as follows. The principle of the Anderson--Pulay acceleration is presented in Section \ref{presentation} through an overview of the DIIS and its relation with the Anderson acceleration and a class of quasi-Newton methods based on multisecant updating. We also recall convergence results existing in the literature for this class of methods, in both linear and nonlinear cases. In Section \ref{CDIIS variants}, a generic abstract problem is set and two variants of the Anderson--Pulay acceleration with adaptive depth are proposed to solve it. For both variants, a local convergence result is proved as well as an acceleration property, and we show that one can theoretically achieve a superlinear convergence rate. In Section \ref{electronic ground state calculation}, the quantum chemistry problems we consider for the application of the methods are recalled and their mathematical properties are discussed in connection with our abstract framework. Finally, numerical experiments, performed on molecules in order to illustrate the convergence behaviour of both methods and to compare them with their ``classical'' fixed-depth counterpart, are reported in Section~\ref{numerical experiments}.

\section{Overview of the DIIS}\label{presentation}
The class of methods named Anderson--Pulay acceleration in the present paper comprises several methods, introduced in various applied contexts with the same goal of accelerating the convergence of fixed-point iterations by means of extrapolation, the most famous of these in the quantum chemistry community probably being the DIIS technique, introduced by Pulay in 1980 \cite{Pulay:1980}. We first describe its basic principle in an abstract context and explore its relation to similar extrapolation methods and to a family of multisecant-type quasi-Newton methods. We conclude this section by recalling a number of previously established results pertaining to the DIIS and the Anderson acceleration.

\subsection{Presentation}\label{DIIS presentation}
Consider the numerical computation of the fixed-point $x_*$ of a nonlinear function $g$ from $\mathbb{R}^n$ to $\mathbb{R}^n$ by the fixed-point iterative scheme
\begin{equation}\label{basic fixed-point iteration}
\forall k\in\mathbb{N},\ x^{(k+1)}=g(x^{(k)}),
\end{equation}
for which one can compute at each step an \textit{``error''} vector $r^{(k)}$ of the form
\[
r^{(k)}=f(x^{(k)}),
\]
with $f$ a function from $\mathbb{R}^n$ to $\mathbb{R}^p$ (the integer $p$ being possibly different from $n$) satisfying
\[
f(x_*)=0.
\]
Given a non-negative integer $m$ (the maximal depth), the DIIS paradigm assumes that a good approximation to the solution $x_*$ can be obtained at step $k+1$ by forming a combination involving the $m_k+1$ previous guess values, with $m_k=\min\,\{m,k\}$, that is
\begin{equation}\label{DIIS iteration}
\forall k\in\mathbb{N},\ x^{(k+1)}=\sum_{i=0}^{m_k}c^{(k)}_i\,g(x^{(k-m_k+i)}),
\end{equation}
while requiring that the coefficients $c^{(k)}_0,\dots,c^{(k)}_{m_k}$ of the combination be such that the norm of the associated linearised error vector, given by $\sum_{i=0}^{m_k}c^{(k)}_i\,r^{(k-m_k+i)}$, is minimal in the least-squares sense under the constraint that
\begin{equation}\label{DIIS constraint}
\sum_{i=0}^{m_k}c^{(k)}_i=1.
\end{equation}

As discussed by Pulay, this minimisation 
may be achieved through a Lagrange multiplier technique applied to the normal equations associated with the problem. More precisely, one can introduce an undetermined scalar $\lambda$ and define the Lagrangian function
\[
L(c_0,\cdots,c_{m_k},\lambda)=\frac{1}{2}\,{\dnorm{\sum_{i=0}^{m_k}c_i\,r^{(k-m_k+i)}}_2}^2-\lambda\,\left(\sum_{i=0}^{m_k}c_i-1\right)=\frac{1}{2}\,\sum_{i=0}^{m_k}\sum_{j=0}^{m_k}c_i c_jb^{(k)}_{ij}-\lambda\,\left(\sum_{i=0}^{m_k}c_i-1\right),
\]
in which the coefficients $b^{(k)}_{ij}$, $i,j=0,\dots,m_k$, are those of the Gramian matrix associated with the set of error vectors stored at the end of step $k$, \textit{i.e.} $b^{(k)}_{ij}=\left<r^{(k-m_k+j)},r^{(k-m_k+i)}\right>_2$. One then finds the saddle point $(c^{(k)}_0,\dots,c^{(k)}_{m_k},\lambda^{(k)})$ of this Lagrangian by solving the associated Euler-Lagrange equations. 
This amounts to solving the following system of $m_k+2$ linear equations
\begin{equation}\label{lagdiis}
\begin{pmatrix}
b^{(k)}_{00}&\dots&b^{(k)}_{0m_k}&-1\\
\vdots&&\vdots&\vdots\\
b^{(k)}_{m_k0}&\dots&b^{(k)}_{m_km_k}&-1\\
-1&\dots&-1&0
\end{pmatrix}
\begin{pmatrix}
c^{(k)}_0\\\vdots\\c^{(k)}_{m_k}\\\lambda^{(k)}
\end{pmatrix}
=
\begin{pmatrix}
0\\\vdots\\0\\-1
\end{pmatrix}.
\end{equation}

The DIIS iterations are generally ended once an acceptable accuracy has been reached, for instance when the value of the norm of the current error vector lies below a prescribed tolerance. This procedure is summarised in Algorithm \ref{alg:diis}.

\smallskip

\begin{algorithm}[H]
\KwData{$x^{(0)}$, $tol$, $m$}
$r^{(0)}=f(x^{(0)})$\\
$k=0$\\
$m_k=0$\\
\While{$\dnorm{r^{(k)}}_2>\text{tol}$}{
    solve the constrained least-squares problem for the coefficients $\{c^{(k)}_i\}_{i=0,\dots,m_k}$ using linear system \eqref{lagdiis}\\
    $x^{(k+1)}=\sum_{i=0}^{m_k}c^{(k)}_ig(x^{(k-m_k+i)})$\\
    $r^{(k+1)}=f(x^{(k+1)})$\\
    $m_{k+1}=\min(m_k+1,m)$\\
    $k=k+1$
}
\caption{Pulay's DIIS \cite{Pulay:1980} for the acceleration of a fixed-point method based on a function $g$.}\label{alg:diis}
\end{algorithm}

\smallskip

In practice and for a large number of applications, the integer $m$ is chosen small, of the order of a few units. Since system \eqref{lagdiis} is related to normal equations, it may nevertheless happen that it is ill-conditioned if some error vectors are (almost) linearly dependent. In such a situation, a possible cure is to drop the oldest stored vectors one by one, 
until the condition number of the resulting system becomes acceptable. One should note that the use of an unconstrained equivalent formulation of the least-squares problem, like those derived in the next subsection, is usually advocated, as it is generally observed that it results in a better condition number for the matrix of the underlying linear system. We refer the reader to \cite{Shepard:2007,Walker:2011} for more details on this topic.

In \cite{Pulay:1980}, Pulay initially suggested using the quantities
\begin{equation}\label{choice of error vectors, Pulay '80}
\forall k\in\mathbb{N},\ r^{(k)}=x^{(k+1)}-x^{(k)}=g(x^{(k)})-x^{(k)},
\end{equation}
as  error vectors. Note that with this choice, $p=n$ and there is a direct relation between 
the error function $f$ and the fixed-point function $g\,:$
\begin{equation}\label{assumed relation between f and g}
f=g-\id.
\end{equation}

Pulay later proposed in \cite{Pulay:1982} a variant form of the procedure, sometimes known as the commutator-DIIS or simply CDIIS (see \cite{Garza:2012} for instance), which is more appropriate for self-consistent field (SCF) iterative schemes in quantum chemistry (see Section \ref{electronic ground state calculation}). In such applications, the fixed-point function takes its values in a submanifold $\Sigma$ (consisting of idempotent matrices), but an extrapolation like \eqref{DIIS iteration} would not lie on $\Sigma$. As a remedy, the definition of the new iterate is modified in CDIIS, so that it belongs to the submanifold. 
Moreover, in this variant, \emph{no} relation of the form \eqref{assumed relation between f and g}  is assumed between the fixed-point and error functions. For applications in quantum chemistry, Pulay introduced an error function that turns out to be the commutator of two matrices, hence the name given to the method.

\subsection{Relation with the Anderson and nonlinear Krylov accelerations} 

\noindent
To solve a nonlinear equation of the form
\begin{equation}\label{nonlinear equation}
h(x)=0,
\end{equation}
a classical idea is to perform a fixed-point iteration based on the function
\begin{equation}\label{fixed-point iteration function}
g(x)=x+\beta\,h(x),
\end{equation}
where $\beta$ is a sufficiently regular, homogeneous operator\footnote{This operator can be seen as a preconditioner of some sort, the simplest case being the multiplication by a (nonzero) constant~$\beta$ -- a relaxation (or damping) parameter -- or, in the context of the fixed-point iteration \eqref{basic fixed-point iteration}, a sequence $(\beta^{(k)})_{k\in\mathbb{N}}$ of such constants (see \cite{Anderson:1965}). Hereafter, we restrict ourselves to the case where $\beta$ is a fixed operator, so that the function $f$ remains unchanged along the iteration.}. One sees that the choice \eqref{choice of error vectors, Pulay '80} for the error vectors corresponds to setting
\[
f(x)=\beta\,h(x).
\]
In this particular case, the error vectors are (possibly preconditioned) \emph{residuals}, and it is possible to relate the DIIS to a structurally similar extrapolation method, the Anderson acceleration \cite{Anderson:1965} (sometimes called the \emph{Anderson mixing}), introduced for the numerical solution of discretised nonlinear integral equations and originally formulated as follows. At the $(k+1)$th step, given the $m_k+1$ most recent iterates $x^{(k-m_k)},\dots,x^{(k)}$ and the corresponding residuals $r^{(k-m_k)},\dots,r^{(k)}$, define
\begin{equation}\label{original AA iteration}
\forall k\in\mathbb{N},\ x^{(k+1)}=g(x^{(k)})+\sum_{i=1}^{m_k}\theta^{(k)}_i\,\left(g(x^{(k-m_k-1+i)})-g(x^{(k)})\right),
\end{equation}
the scalars $\theta^{(k)}_i$, $i=1,\dots,m_k$ being chosen so as to minimise the norm of the associated linearised residual, \textit{i.e.}
\begin{equation}\label{original AA minimization}
\forall k\in\mathbb{N},\ \vec{\theta}^{(k)}=\underset{(\theta_1,\dots,\theta_{m_k})\in\mathbb{R}^{m_k}}{\arg\min}\,\left\|r^{(k)}+\sum_{i=1}^{m_k}\theta_i\,\left(r^{(k-m_k-1+i)}-r^{(k)}\right)\right\|_2.
\end{equation}
Setting
\[
\forall k\in\mathbb{N},\ c^{(k)}_i=\theta^{(k)}_{i+1},\ i=0,\dots,m_k-1,\text{ and }c^{(k)}_{m_k}=1-\sum_{j=1}^{m_k}\theta^{(k)}_j,
\]
one observes that relation \eqref{original AA iteration} can be put into the form of relation \eqref{DIIS iteration}, so that the DIIS applied to the solution of \eqref{nonlinear equation} by fixed-point iteration is actually a reformulation of the Anderson acceleration.

This instance of the DIIS is not the only reinvention of the Anderson acceleration. In 1997, Washio and Oosterlee \cite{Washio:1997} introduced a closely related process dubbed \emph{Krylov subspace acceleration} for the solution of nonlinear partial differential equation problems. The extrapolation within it relies on a nonlinear extension of the GMRES method (see subsection \ref{equivalence with GMRES in the linear case}) and is in all points identical to \eqref{original AA iteration} and \eqref{original AA minimization}, save for the definition of the fixed-point function $g$ which corresponds in their setting to the application of a cycle of a nonlinear multigrid scheme. Another method, also based on a nonlinear generalisation of the GMRES method, is the so-called \emph{nonlinear Krylov acceleration} (NKA), introduced by Carlson and Miller around 1990 (see \cite{Carlson:1998}) and used for accelerating modified Newton iterations in a finite element solver with moving mesh. It leads to yet another equivalent formulation of the Anderson acceleration procedure by setting
\begin{equation}\label{NKA iteration}
\forall k\in\mathbb{N},\ x^{(k+1)}=g(x^{(k)})-\sum_{i=1}^{m_k}\alpha^{(k)}_i\,\left(g(x^{(k-m_k+i)})-g(x^{(k-m_k+i-1)})\right),
\end{equation}
with
\begin{equation}\label{NKA minimization}
\forall k\in\mathbb{N},\ \vec{\alpha}^{(k)}=\underset{(\alpha_1,\dots,\alpha_{m_k})\in\mathbb{R}^{m_k}}{\arg\min}\,\left\|r^{(k)}-\sum_{i=1}^{m_k}\alpha_i\,\left(r^{(k-m_k+i)}-r^{(k-m_k+i-1)}\right)\right\|_2,
\end{equation}
which amounts to take $\alpha^{(k)}_i=\sum_{j=1}^{i}\theta^{(k)}_j$,\ $i=1,\dots,m_k$. Another form of the minimisation problem for the norm of the linearised residual is found in \cite{DeSterck:2012}, where it is combined with a line-search to yield a so-called \emph{nonlinear GMRES} (N-GMRES) optimisation algorithm for computing the sum of $R$ rank-one tensors that has minimal distance to a given tensor in the Frobenius norm.

\medskip

The fact that the above acceleration schemes are different forms of the same method has been recognised on several occasions in the literature, see \cite{Fattebert:2010,Walker:2011,Calef:2013} for instance or Anderson's own account in \cite{Anderson:2019}. In \cite{Brezinski:2018}, it is shown that the Anderson acceleration is actually part of a general framework for the so-called Shanks transformations, used for accelerating the convergence of sequences.

Like the DIIS, the Anderson acceleration has been used to improve the convergence of numerical schemes in a number of contexts, and the recent years have seen a wealth of publications dealing with a large range of applications (see \cite{Lott:2012,Ganine:2013,Willert:2014,Higham:2016,An:2017,Pavlov:2018,ZhangYao,Henderson:2019} for instance). Variants with restart conditions to prevent stagnation \cite{Fang:2009}, periodic restarts \cite{Pratapa:2015} or a periodic use of the extrapolation \cite{Banerjee:2016} have also been proposed.

\subsection{Interpretation as a multisecant-type quasi-Newton method}
An insight on the behaviour of the DIIS may be gained by seeing it as a quasi-Newton method using multiple secant equations. It was indeed observed by Eyert \cite{Eyert:1996} that the Anderson acceleration procedure amounts to a modification of Broyden's \emph{second\footnote{It is sometimes also called the \emph{bad} Broyden method (see \cite{Griewank:2012}).}} method \cite{Broyden:1965}, in which a given number of secant equations are satisfied at each step. This relation was further clarified by Fang and Saad \cite{Fang:2009} as follows (see also the paper by Walker and Ni \cite{Walker:2011}).

Keeping with the previously introduced notations\footnote{We continue to assume that the relaxation parameter $\beta$ does not vary from one iteration to the next.} and considering vectors as column matrices, we introduce the matrices respectively containing the differences of successive iterates and the differences of associated successive error vectors stored at step $k$,
\[
\mathscr{Y}^{(k)}=\left[x^{(k-m_k+1)}-x^{(k-m_k)},\dots,x^{(k)}-x^{(k-1)}\right]\text{ and }
\mathscr{S}^{(k)}=\left[r^{(k-m_k+1)}-r^{(k-m_k)},\dots,r^{(k)}-r^{(k-1)}\right],
\]
in order to rewrite the recursive relation \eqref{NKA iteration} as
\[
\forall k\in\mathbb{N},\ x^{(k+1)}=x^{(k)}+r^{(k)}-\left(\mathscr{Y}^{(k)}+\mathscr{S}^{(k)}\right)\,\vec{\alpha}^{(k)},
\]
the minimization problem \eqref{NKA minimization} becoming a simple least-squares problem,
\[
\forall k\in\mathbb{N},\ \vec{\alpha}^{(k)}=\underset{\vec{\alpha}\in\mathcal{M}_{m_k,1}(\mathbb{R})}{\arg\min}\,\dnorm{r^{(k)}-\mathscr{S}^{(k)}\vec{\alpha}}_2.
\]
Assuming that $\mathscr{S}^{(k)}$ is a full-rank matrix, which means that the error vectors are affinely independent, and characterising $\vec{\alpha}^{(k)}$ as the solution of the associated normal equations, with closed form
\begin{equation}\label{type-I AA normal equations}
\vec{\alpha}^{(k)}=\left({\mathscr{S}^{(k)}}^{\top}\mathscr{S}^{(k)}\right)^{-1}{\mathscr{S}^{(k)}}^{\top}r^{(k)},
\end{equation}
one obtains the relation
\[
\forall k\in\mathbb{N},\ x^{(k+1)}=x^{(k)}-\left(-I_n+\left(\mathscr{Y}^{(k)}+\mathscr{S}^{(k)}\right)\left({\mathscr{S}^{(k)}}^{\top}\mathscr{S}^{(k)}\right)^{-1}{\mathscr{S}^{(k)}}^{\top}\right)r^{(k)},
\]
which can be identified with the update formula of a quasi-Newton method of multisecant type,
\begin{equation}\label{quasi-Newton iteration}
\forall k\in\mathbb{N},\ x^{(k+1)}=x^{(k)}-G^{(k)}r^{(k)},
\end{equation}
with
\begin{equation}\label{formula for G}
G^{(k)}=-I_n+\left(\mathscr{Y}^{(k)}+\mathscr{S}^{(k)}\right)\left({\mathscr{S}^{(k)}}^{\top}\mathscr{S}^{(k)}\right)^{-1}{\mathscr{S}^{(k)}}^{\top}.
\end{equation}
Here, the matrix $G^{(k)}$ is regarded as an approximate inverse of the Jacobian of the function $f$ at point $x^{(k)}$, satisfying the inverse multiple secant condition
\[
G^{(k)}\mathscr{S}^{(k)}=\mathscr{Y}^{(k)}.
\]
It moreover minimises the Frobenius norm $\norm{G^{(k)}+I_n}_2$ among all the matrices satisfying this condition. Thus, formula \eqref{formula for G} can be viewed as a rank-$m_k$ update of $-I_n$ generalising Broyden's second method, which effectively links the Anderson acceleration to a particular class of quasi-Newton methods. Apparently not aware of this connection, Rohwedder and Schneider \cite{Rohwedder:2011} derived a similar conclusion in their analysis of the DIIS considering a full history of iterates, showing that it corresponds to a projected variant of Broyden's second method proposed by Gay and Schnabel in 1977 (see Algorithm II' in \cite{Gay:1977}).

\smallskip

In contrast, the generalisation of Broyden's first method aims at directly approximating the Jacobian of the function $f$ at point $x^{(k)}$ by a matrix $B^{(k)}$ which minimises the norm $\norm{B^{(k)}+I_n}_2$ subject to the multiple secant condition $B^{(k)}\mathscr{Y}^{(k)}=\mathscr{S}^{(k)}$. Under the assumption that the matrix $\mathscr{Y}^{(k)}$ is full-rank, this yields 
\[
B^{(k)}=-\,I_n+\left(\mathscr{Y}^{(k)}+\mathscr{S}^{(k)}\right)\left({\mathscr{Y}^{(k)}}^{\top}\mathscr{Y}^{(k)}\right)^{-1}{\mathscr{Y}^{(k)}}^{\top}.
\]
Using this fact, Fang and Saad \cite{Fang:2009} defined a ``type-I'' Anderson acceleration (the ``type-II'' Anderson acceleration corresponding to the original one, related to Broyden's second method as seen above). Indeed, assuming that the matrix ${\mathscr{Y}^{(k)}}^{\top}\mathscr{S}^{(k)}$ is invertible and applying the Sherman--Morrison--Woodbury formula, it follows that
\[
(B^{(k)})^{-1}=-I_n+\left(\mathscr{Y}^{(k)}+\mathscr{S}^{(k)}\right)\left({\mathscr{Y}^{(k)}}^{\top}\mathscr{S}^{(k)}\right)^{-1}{\mathscr{Y}^{(k)}}^{\top},
\]
and substituting $(B^{(k)})^{-1}$ to $G^{(k)}$ in \eqref{quasi-Newton iteration}, one is led to a variant of the original method, the coefficients of which satisfy
\[
\vec{\tilde\alpha}^{(k)}=\left({\mathscr{Y}^{(k)}}^{\top}\mathscr{S}^{(k)}\right)^{-1}{\mathscr{Y}^{(k)}}^{\top}r^{(k)},
\]
instead of \eqref{type-I AA normal equations}.

\smallskip

To end this subsection, let us mention that a relation between the DIIS applied to ground state calculations and quasi-Newton methods was also suspected, albeit in heuristic ways, in articles originating from the computational chemistry community (see \cite{Thogersen:2005,Kudin:2007} for instance).

\subsection{Equivalence with the GMRES method for linear problems}\label{equivalence with GMRES in the linear case}
Consider the application of the DIIS to the solution of a system of $n$ linear equations in $n$ unknowns, with solution $x_*$, by assuming that the function $h$ in \eqref{nonlinear equation} is of the form $h(x)=b-Ax$, where $A$ is a nonsingular matrix of order $n$ and $b$ a vector of $\mathbb{R}^n$. In such a case, relation \eqref{fixed-point iteration function} amounts to $g(x)=(I_n-\beta\,A)\,x+\beta\,b$, and the fixed-point iteration method \eqref{basic fixed-point iteration} reduces to the stationary Richardson method. In what follows, we assume that all the previous error vectors are kept at each iteration, that is $m=+\infty$, so that $m_k=k$ for each natural integer $k$.

A well-known iterative method for the solution of the above linear system is the GMRES method \cite{Saad:1986}, in which the approximate solution $x^{(k+1)}$ at the $(k+1)$th step is the unique vector minimising the Euclidean norm of the residual
\[
r_{\mathrm{GMRES}}^{(k+1)}=b-A\,x^{(k+1)}
\]
in the affine space
\[
\mathcal{W}_{k+1}=\left\{v=x^{(0)}+z\,|\,z\in\mathcal{K}_{k+1}(A,r_{\mathrm{GMRES}}^{(0)})\right\},
\]
where, for any integer $j$ greater than or equal to $1$, $\mathcal{K}_j(A,r_{\mathrm{GMRES}}^{(0)})=\operatorname{span}\left\{r_{\mathrm{GMRES}}^{(0)},Ar_{\mathrm{GMRES}}^{(0)},\dots,A^{j-1}r_{\mathrm{GMRES}}^{(0)}\right\}$ is the \emph{order-$j$ Krylov (linear) subspace} generated by the matrix $A$ and the residual vector $r_{\mathrm{GMRES}}^{(0)}=b-Ax^{(0)}$, the starting vector $x^{(0)}$ being given. Since each of these Krylov subspaces is contained in the following one, the norm of the residual decreases monotonically with the iterations and the exact solution is obtained in at most $n$ iterations (the idea being that an already good approximation to the exact solution is reached after a number of iterations much smaller than $n$). More precisely, the following result holds (for a proof, see for instance \cite{Potra:2013}).

\begin{prpstn}
The GMRES method converges in exactly $\nu(A,r_{\mathrm{GMRES}}^{(0)})$ steps, where the integer $\nu(A,r_{\mathrm{GMRES}}^{(0)})$ is the grade\footnote{The grade of a non-zero vector $x$ with respect to a matrix $A$ is the smallest integer $\ell$ for which there is a non-zero polynomial $p$ of degree $\ell$ such that $p(A)x=0$.} of $r_{\mathrm{GMRES}}^{(0)}$ with respect to $A$.
\end{prpstn}


In \cite{Haelterman:2010}, it was shown that the GMRES method could be interpreted as a quasi-Newton method. Similarly, the fact that the DIIS (or the Anderson acceleration) used to solve a linear system is equivalent to the GMRES method in the following sense was proved in \cite{Walker:2011} and \cite{Rohwedder:2011} (see also \cite{Potra:2013} for more refined results).

\begin{thrm}\label{GMRES=DIIS}
Suppose that $x_{\mathrm{DIIS}}^{(0)}=x_{\mathrm{GMRES}}^{(0)}=x^{(0)}$ and that  the sequence of residual norms does not stagnate\footnote{As shown in \cite{Greenbaum:1996}, any nonincreasing sequence of residual norms can be produced by the GMRES method. It is then possible for the residual norm to stagnate at some point, say at the $k$th step, $k$ being a nonzero natural integer, with $r_{\mathrm{GMRES}}^{(k)}=r_{\mathrm{GMRES}}^{(k-1)}\neq0$. In such a case, the equivalence between the methods implies that $x_{\mathrm{DIIS}}^{(k+1)}=x_{\mathrm{DIIS}}^{(k)}$, making the least-square problem associated with the minimisation ill-posed due to the family $\{f(x_{\mathrm{DIIS}}^{(0)}),\dots,f(x_{\mathrm{DIIS}}^{(k+1)})\}$ not having full rank. The DIIS method then breaks down upon stagnation before the solution has been found, whereas the GMRES method does not. As pointed out in \cite{Walker:2011}, this results in the conditioning of the least-square problem being of utmost importance in the numerical implementation of the method.} before step $k$, \it{i.e.} $r_{\mathrm{GMRES}}^{(k-1)}\neq0$ and $\norm{r_{\mathrm{GMRES}}^{(j-1)}}_2>\norm{r_{\mathrm{GMRES}}^{(j)}}_2$ for each integer $j$ such that $0<j<k$. Then
\begin{equation*}
\ x_{\mathrm{GMRES}}^{(k)}=\sum_{i=0}^{k}c^{(k)}_ix_{\mathrm{DIIS}}^{(i)},
\end{equation*}
and
\begin{equation*}
\ x_{\mathrm{DIIS}}^{(k+1)}=g(x_{\mathrm{GMRES}}^{(k)}).
\end{equation*}
\end{thrm}

\medskip

Assuming that a full history of iterates is kept and that no stagnation occurs, this equivalence directly provides convergence results in the linear case for the DIIS in view of the well-known theory for the GMRES method. In particular, if $g$ is a contraction, then stagnation is forbidden and the DIIS converges to the exact solution in a finite number of steps. Theorem \ref{GMRES=DIIS} also justifies the inclusion of Krylov's name in some of the rediscoveries of the Anderson acceleration we previously mentioned. Of course, if the number of stored iterates is fixed or if ``restarts'' are allowed over the course of the computation, in the spirit of the periodically restarted GMRES method, GMRES($m$), these results are no longer valid. Nevertheless, the behaviour of GMRES($m$) has been studied in a number of particular cases, and various restart or truncation strategies (see \cite{DeSturler:1999} and the references therein for instance) have been proposed over the years in order to compensate for the loss of information by selectively retaining some of it from earlier iterations.

Walker and Ni \cite{Walker:2011} showed in a similar way that the type-I Anderson acceleration, recalled in the preceding subsection, is essentially equivalent to the full orthogonalisation method (FOM) based on the Arnoldi process (see section 2 of \cite{Saad:1986}) in the linear case.

\subsection{Existing convergence theories in the nonlinear case}
In light of the previous subsection, it appears that the convergence theory for the DIIS in the linear case is well developed and intimately linked to that of the GMRES method. On the contrary, this theory is far from being established when the method is applied to nonlinear problems. The interpretation of the DIIS as a particular instance of quasi-Newton methods does not provide an answer, as there is no general convergence theory for these methods. Nevertheless, results have emerged in the literature in the recent years.

\smallskip

In \cite{Rohwedder:2011}, Rohwedder and Schneider proved a local q-convergence result\footnote{The notions of q-linear convergence and r-linear convergence are both recalled in Subsection \ref{linear convergence}.} for the DIIS applied to a general nonlinear problem of the form \eqref{nonlinear equation}, with the choice $\beta=-1$ in \eqref{fixed-point iteration function}. They assumed that the underlying mapping is locally contractive and that error vectors associated with the stored iterates satisfy an affine independence condition. Their analysis is based on the equivalence between the DIIS and a projected variant of Broyden's second method, the convergence properties and an improvement of estimates appearing in the proof of convergence by Gay and Schnabel in \cite{Gay:1977}. By interpreting the DIIS as an inexact Newton method, they also obtained a refined estimate for the error vector at a given step that allowed them to discuss the possibility of obtaining a superlinear convergence rate for the method.

\smallskip

More recently, Kelley and Toth \cite{Toth:2015} studied the use of the Anderson acceleration on a fixed-point iteration in $\mathbb{R}^n$, also based on the function \eqref{fixed-point iteration function} in which one has set $\beta=1$, the history of error vectors being of fixed depth $m$. Assuming that the fixed-point mapping is Lipschitz continuously differentiable and is a contraction in a neighbourhood of the solution, and with the requirement that the extrapolation coefficients remain uniformly bounded in the $\ell^1$-norm, they proved that the method converges locally r-linearly. When $m=1$, they showed that the sequence of residuals converges q-linearly if the fixed-point mapping is sufficiently contractive and the initial guess is close enough to the solution. This analysis was later extended to the case where the evaluation of the fixed point map is corrupted with errors in \cite{Toth:2017}. Chen and Kelley \cite{Chen:2019} also obtained global and local convergence results for a variant of the method in which a non-negativity constraint on the coefficients is imposed (an idea previously used in the EDIIS method \cite{Kudin:2002}), generalising the results in \cite{Toth:2015}. Recently, Evans \textit{et al.} \cite{Evans:2020} studied the convergence improvement given by the Anderson method.

\smallskip

Note, however, that in all these works, for $m>1$ an assumption is made on the boundedness of the extrapolation coefficients, or on the affine independence of the error vectors. Such an assumption can only be checked \textit{a posteriori} in the standard DIIS or Anderson algorithm.
To address this issue, a stabilised version of the type-I Anderson acceleration is introduced by Zhang \textit{et al.} in \cite{Zhang:2020}. It includes a regularisation ensuring the non-singularity of the approximated Jacobian, a restart strategy guaranteeing a certain linear independence of the differences of successive stored iterates, and a safeguard mechanism guaranteeing a decrease of the residual norm. A global convergence result for the corresponding acceleration applied to a Krasnoselskii--Mann iteration in $\mathbb{R}^n$ is then proved.

\section{Restarted and adaptive-depth Anderson--Pulay acceleration}\label{CDIIS variants}
Two peculiarities of the CDIIS variant proposed by Pulay for electronic ground state calculations \cite{Pulay:1982} are that, when interpreted in terms of an abstract fixed-point function $g$ and an abstract error function $f$, the target set of the function $g$ is possibly a submanifold and the functions $g$ and $f$ do not necessarily satisfy a relation of the form \eqref{assumed relation between f and g} (see Section \ref{electronic ground state calculation} for more details). As a consequence, the convergence theories found in \cite{Rohwedder:2011} or in \cite{Toth:2015} do not apply.

This observation leads us to introduce an abstract framework in which the class of Anderson--Pulay acceleration algorithms is defined. This class encompasses and generalises the DIIS, the Anderson acceleration, the CDIIS and other methods reviewed in the previous section (see Algorithm~\ref{alg:diisori}). In addition, two adaptive modifications of this procedure are proposed: the first one, Algorithm~\ref{alg:cdiis-restart-variant}, includes a restart condition in the spirit of Gay and Schnabel \cite{Gay:1977}; in the second one, Algorithm~\ref{alg:cdiis-adaptive-depth-variant}, the depth is continuously adapted at each step according to a new criterion.

\subsection{Description of the algorithms in an abstract framework} 
\medskip

We consider an error function $f$, a fixed-point function $g$, and a point $x_*$ in $\mathbb{R}^n$ such that $f(x_*)=0$ and $g(x_*)=x_*$.
We do \emph{not} impose that $f$ and $g$ be linked by a relation of the form \eqref{assumed relation between f and g}. In Subsection \ref{linear convergence}, two assumptions will be made on these functions in a neighborhood of $x_*$, ensuring that the norm $\Vert f(x^{(k)})\Vert_2$ of the error diminishes along the basic fixed-point iteration~\eqref{basic fixed-point iteration}, and that this norm controls the distance $\Vert x^{(k)} - x_*\Vert_2$.\medskip

Our precise definition of Anderson--Pulay acceleration is given below in Algorithm~\ref{alg:diisori}. \smallskip

\begin{algorithm}[H]
\KwData{$x^{(0)}$, $tol$, $m$}
$r^{(0)}=f(x^{(0)})$\\
$k=0$\\
$m_k=0$\\
\While{$\dnorm{r^{(k)}}_2>\text{tol}$}{
   solve the constrained least-squares problem for the coefficients $\{c^{(k)}_i\}_{i=0,\dots,m_k}=\underset{\substack{(c_0,\dots,c_{m_k})\in\mathbb{R}^{m_k+1}\\\sum_{i=0}^{m_k}c_i=1}}{\arg\min}\,\dnorm{\sum_{i=0}^{m_k}c_i\,r^{(k-m_k+i)}}_2$\\
   \mbox{ }\\
   $x^{(k+1)}=\sum_{i=0}^{m_k}c^{(k)}_ig(x^{(k-m_k+i)})$ (\textbf{version A})\\
   \textbf{or}\\
   $x^{(k+1)}=g(\widetilde{x}^{(k+1)})\text{ with }\widetilde{x}^{(k+1)}=\sum_{i=0}^{m_k}c^{(k)}_ix^{(k-m_k+i)}$ (\textbf{version P})\\
   $r^{(k+1)}=f(x^{(k+1)})$\\
   $m_{k+1}=\min(m_k+1,m)$\\
   $k=k+1$
}
\caption{Fixed-depth Anderson--Pulay acceleration for a fixed-point method based on a function~$g$.}\label{alg:diisori}
\end{algorithm}

\smallskip

One can observe that the Anderson--Pulay acceleration comes in two versions, which only differ in the definition of the new iterate. Comparing with Algorithm~\ref{alg:diis}, one sees that version A corresponds to the DIIS \cite{Pulay:1980}, and is equivalent to the Anderson acceleration \cite{Anderson:1965} if relation \eqref{assumed relation between f and g} holds. The formula giving the new iterate in version A can be seen as a linearised form of the one in version P, the latter being directly inspired by the CDIIS \cite{Pulay:1982}. One may thus note that the two versions coincide for a linear fixed-point function~$g$.

\smallskip

As already discussed in Subsection \ref{DIIS presentation}, the numerical solution of the constrained linear least-squares problem for the extrapolation coefficients requires some care. Indeed, the matrix of linear system \eqref{lagdiis} resulting from the use of a Lagrange multiplier accounting for the constraint has to be well-conditioned for the method to be applicable in floating-point arithmetic. Unconstrained formulations of the problem exist, as employed in the Anderson acceleration (see \eqref{original AA minimization}) or its numerous reinventions (see \eqref{NKA minimization} for instance), each leading to an equivalent Anderson--Pulay acceleration algorithm, but with a possibly different condition number for the matrix of the linear system associated with the least-squares problem.

Note that a bound on the extrapolation coefficients is also needed when investigating the convergence of the method. In \cite{Toth:2015}, this bound is presented as a requirement, following, for instance, from the uniform well-conditioning of the least-squares problem, which cannot be \textit{a priori} guaranteed. Nevertheless, such a condition can be enforced \textit{a posteriori} by diminishing the value of the integer $m_k$. In practice, this may be achieved by monitoring the condition number of the least-squares coefficient matrix and by dropping as many of its left-most columns as needed to have this condition number below a prescribed threshold, as proposed in \cite{Walker:2011}. Another possibility consists in ensuring that the error vectors associated with the stored iterates fulfill a kind of linear independence condition like the one found in \cite{Gay:1977,Rohwedder:2011}. To enforce this condition in practice, one can simply choose to ``restart'' the method, by resetting the value of $m_k$ to zero and discarding the iterates previously stored, as soon as an ``almost'' linear dependence is detected, as done in \cite{Gay:1977,Zhang:2020}. This approach is adopted for the first adaptive modification of Algorithm~\ref{alg:diisori} proposed and analysed in the present work.

More precisely, assuming that $m_k\geq1$ at step $k$, set 
\[
s^{(k-m_k+i)}=r^{(k-m_k+i)}-r^{(k-m_k)},\ i=1,\dots,m_k,
\]
and let $\Pi_k$ denote the orthogonal projector onto $\text{span}\{s^{(k-m_k+1)},\dots,s^{(k)}\}$. The algorithm will ``restart'' at step $k+1$ (meaning that the number $m_{k+1}$ will be set to zero) if the norm of $s^{(k+1)}-\Pi_ks^{(k+1)}$ is ``small''  compared to the norm of $s^{(k+1)}$, that is, if the following inequality holds:
\begin{equation}\label{restart condition, abstract problem}
\tau\,\norm{r^{(k+1)}-r^{(k-m_k)}}_2>\norm{(\id-\Pi_k)(r^{(k+1)}-r^{(k-m_k)})}_2
\end{equation}
where $\tau$ is a real parameter chosen between $0$ and $1$. Otherwise, the integer $m_k$ is incremented by one unit with the iteration number $k$. One can view condition \eqref{restart condition, abstract problem} as a near-linear dependence criterion for the set of differences of stored error vectors. Indeed, as explained in \cite{Gay:1977}, it \emph{``recognizes that, in general, the projection of $s^{(i)}$ orthogonal to the subspace spanned by $s^{(1)},\dots,s^{(i-1)}$ must be the zero vector for some $i\leq n$''}.

Such a restart condition is particularly adapted to a specific unconstrained formulation of the linear least-squares problem for the extrapolation coefficients
\begin{equation}\label{DIIS minimisation problem}
\vec{c}^{(k)}=\underset{\substack{(c_0,\dots,c_{m_k})\in\mathbb{R}^{m_k+1}\\\sum_{i=0}^{m_k}c_i=1}}{\arg\min}\,\dnorm{\sum_{i=0}^{m_k}c_i\,r^{(k-m_k+i)}}_2,
\end{equation}
to be effectively employed in practice. In this formulation, instead of using the constrained vector of coefficients $\vec{c}=(c_0,\dots,c_{m_k})$, one works with the unconstrained vector $\vec{\gamma}=(\gamma_1,\cdots,\gamma_{m_k})$ with $\gamma_i=c_i$, $1\leq i \leq m_k$. One then has
\[
\sum_{i=0}^{m_k}c_i\,r^{(k-m_k+i)}=r^{(k-m_k)}+\sum_{i=1}^{m_k}\gamma_i\,s^{(k-m_k+i)},
\]
resulting in \eqref{DIIS minimisation problem} being replaced by
\[
\vec{\gamma}^{(k)}=\underset{(\gamma_1,\dots,\gamma_{m_k})\in\mathbb{R}^{m_k}}{\arg\min}\,\dnorm{r^{(k-m_k)}+\sum_{i=1}^{m_k}\gamma_i\,s^{(k-m_k+i)}}_2.
\]

The corresponding  modification of the Anderson--Pulay acceleration is given in Algorithm \ref{alg:cdiis-restart-variant}. The question of the numerical computation of the orthogonal projector $\Pi_k$ will be addressed in Subsection \ref{implementation}.

Note that, with this instance of the method, all the stored iterates, except for the last one, are discarded when a restart occurs. In some practical computations, a temporary slowdown of the convergence is observed after the restart (see Section \ref{numerical experiments}). To try to remedy such an inconvenience, we introduce a smoother adaptive modification of Algorithm~\ref{alg:diisori}, based on an update of the set of stored iterates at step $k+1$. The idea is to eliminate the old iterates that are too far from the fixed point, compared to the most recent one. More precisely, one chooses the largest possible depth $m_{k+1}\leq m_k +1$ such that the following inequality is satisfied for all $k+1-m_{k+1}\leq i\leq k$:
\[
\delta\norm{r^{(i)}}_2<\norm{r^{(k+1)}}_2.
\]
Here, the real number $\delta$ is a small positive parameter. To the best of our knowledge, this criterion is new, and the corresponding variant of the Anderson--Pulay acceleration is given below in Algorithm \ref{alg:cdiis-adaptive-depth-variant}.

In this last algorithm, one may observe that the unconstrained form of the least-squares problem for the coefficients uses the differences of error vectors associated with \emph{successive} iterates, that is, $s^{(i)}=r^{(i)}-r^{(i-1)}$, $i=k-m_k+1,\dots,k$, which is a computationally convenient choice when storing a set of iterates whose depth is adapted at each step. One then has $\alpha_i^{(k)}=\sum_{j=0}^{i-1}c_j^{(k)}$, $i=1,\dots,m_k$.\newpage

\begin{algorithm}[H]
\KwData{$x^{(0)}$, $tol$, $\tau$}
$r^{(0)}=f(x^{(0)})$\\
$k=0$\\
$m_k=0$\\
\While{$\dnorm{r^{(k)}}_2>\text{tol}$}{
    \eIf{$m_k=0$}{
        $x^{(k+1)}=g(x^{(k)})$\\
        }{
        solve the unconstrained least-squares problem for the coefficients $\{\gamma^{(k)}_i\}_{i=1,\dots,m_k}=\underset{(\gamma_1,\dots,\gamma_{m_k})\in\mathbb{R}^{m_k}}{\arg\min}\,\dnorm{r^{(k-m_k)}+\sum_{i=1}^{m_k}\gamma_i\,s^{(k-m_k+i)}}_2$\\
        $x^{(k+1)}=g(x^{(k-m_k)})+\sum_{i=1}^{m_k}\gamma^{(k)}_i(g(x^{(k-m_k+i)})-g(x^{(k-m_k)}))$ (\textbf{version A})\\
        \textbf{or}\\
        $x^{(k+1)}=g(\widetilde{x}^{(k+1)})\text{ with }\widetilde{x}^{(k+1)}=x^{(k-m_k)}+\sum_{i=1}^{m_k}\gamma^{(k)}_i(x^{(k-m_k+i)}-x^{(k-m_k)})$ (\textbf{version P})
    }
   $r^{(k+1)}=f(x^{(k+1)})$\\
   $s^{(k+1)}=r^{(k+1)}-r^{(k-m_k)}$\\
   compute $\Pi_k s^{(k+1)}$ (the orthogonal projection of $s^{(k+1)}$ onto $span\{s^{(k-m_k+1)},\dots,s^{(k)}\}$)\\
   \eIf{$\tau\,\norm{s^{(k+1)}}_2>\norm{(\id-\Pi_k)s^{(k+1)}}_2$}{
      $m_{k+1}=0$\\
      }{
      $m_{k+1}=m_k+1$\\
    }
   $k=k+1$\\
}
\caption{Restarted Anderson-Pulay acceleration.}\label{alg:cdiis-restart-variant}
\end{algorithm}

\smallskip

\begin{algorithm}[H]
\KwData{$x^{(0)}$, $tol$, $\delta$}
$r^{(0)}=f(x^{(0)})$\\
$k=0$\\
$m_k=0$\\
\While{$\dnorm{r^{(k)}}_2>\text{tol}$}{
   \eIf{$m_k=0$}{
        $x^{(k+1)}=g(x^{(k)})$\\
        }{
        solve the unconstrained least-squares problem for the coefficients $\{\alpha^{(k)}_i\}_{i=1,\dots,m_k}=\underset{(\alpha_1,\dots,\alpha_{m_k})\in\mathbb{R}^{m_k}}{\arg\min}\dnorm{r^{(k)}-\sum_{i=1}^{m_k}\alpha_i\,s^{(k-m_k+i)}}_2$\\
        $x^{(k+1)}=g(x^{(k)})-\sum_{i=1}^{m_k}\alpha^{(k)}_i(g(x^{(k-m_k+i)})-g(x^{(k-m_k+i-1)}))$ (\textbf{version A})\\
        \textbf{or}\\
        $x^{(k+1)}=g(\widetilde{x}^{(k+1)})\text{ with }\widetilde{x}^{(k+1)}=x^{(k)}-\sum_{i=1}^{m_k}\alpha^{(k)}_i(x^{(k-m_k+i)}-x^{(k-m_k+i-1)})$ (\textbf{version P})
    }
   $r^{(k+1)}=f(x^{(k+1)})$\\
   $s^{(k+1)}=r^{(k+1)}-r^{(k)}$\\
   set $m_{k+1}$ the largest integer $m\leq m_k+1$ such that for $k+1-m\leq i\leq k$, $\delta\,\norm{r^{(i)}}_2<\norm{r^{(k+1)}}_2$\\
   $k=k+1$
}
\caption{Adaptive-depth Anderson--Pulay acceleration.}\label{alg:cdiis-adaptive-depth-variant}
\end{algorithm}

The rest of this section is devoted to theoretical convergence results for Algorithms \ref{alg:cdiis-restart-variant} and \ref{alg:cdiis-adaptive-depth-variant}. Since we will analyse the behaviour of infinite sequences, the stopping parameter $tol$ will be set to zero. If there exists some natural integer $k_\text{stop}$ for which $r^{(k_\text{stop})}=0$, we will also adopt the convention that ${x}^{(k)}={x}^{(k_\text{stop})}$, ${r}^{(k)}=0$ and $m_k=0$ for all natural integers $k$ greater or equal to $k_\text{stop}$. Moreover, whenever $m_k=0$, we will take $c^{(k)}_0=1$.

\subsection{Linear convergence}\label{linear convergence}
We shall now study under natural assumptions the convergence of the modified Anderson--Pulay acceleration methods devised in the previous subsection. To this end, we consider the following functional setting.

Let $n$ and $p$ be two nonzero natural integers, $\Sigma$ be a smooth submanifold in $\mathbb{R}^n$, $V$ be an open subset of $\mathbb{R}^n$, $f$ be a function in $\mathscr{C}^2(V,\mathbb{R}^p)$, $g$ be a function in $\mathscr{C}^2(V,\Sigma)$, and $x_*$ in $V\cap\Sigma$ be a fixed point of $g$ satisfying $f(x_*)=0$. We next make two assumptions.

\begin{enumerate}[\bf {Assumption} 1.]
\item\label{assumption 1} There exists a constant $K$ in $(0,1)$ such that
\[ 
\forall x\in V\cap g^{-1}(V) \cap \Sigma,\ \norm{(f\circ g)(x)}_2\leq K\,\norm{f(x)}_2.
\] 
\item\label{assumption 2} There exists a positive constant $\sigma$ such that
\[
\forall x\in V\cap\Sigma,\ \sigma \norm{x-x_*}_2\leq \norm{f(x)}_2.
\]
\end{enumerate}
Note that these assumptions ensure that $x_*$ is the unique fixed point of $g$ in $V$ and the unique solution of the equation $f(x)=0$ in $V\cap\Sigma$. There might be other zeroes of $f$ lying outside of $\Sigma$, but they would not be fixed points of $g$.

\smallskip

Since $f$ and $g$ are both of class $\mathscr{C}^2$, we may assume, taking $U\subset V\cap g^{-1}(V)$ a smaller neighbourhood of $x_*$ if necessary, that
\begin{multline}\label{mean value inequalities}
\forall(x,y)\in U^2,\ \norm{f(x)-f(y)}_2\leq 2\,\norm{\mathrm{D}f(x_*)}_2\norm{x-y}_2\\\text{ and }\norm{(f\circ g)(x)-(f\circ g)(y)}_2\leq 2\,\norm{\mathrm{D}f(x_*)\circ\mathrm{D}g(x_*)}_2\norm{x-y}_2,
\end{multline}
where $\mathrm{D}f(x_*)$ and $\mathrm{D}g(x_*)$ are the respective differentials of $f$ and $g$ at point $x_*$, and that one also has, for some positive constants $L$ and $L'$,
\begin{equation}\label{eq:f_C2_estimate}
\forall x\in U,\ \norm{f(x)-\mathrm{D}f(x_*)(x-x_*)}_2\leq\frac{L}{2}\,{\norm{x-x_*}_2}^2,
\end{equation}
\begin{equation}\label{eq:g_C2_estimate}
\forall x\in U,\ \norm{g(x)-x_*-\mathrm{D}g(x_*)(x-x_*)}_2\leq\frac{L'}{2}\,{\norm{x-x_*}_2}^2.
\end{equation}

We may also impose that $U$ be a tubular neighbourhood of $\Sigma \cap U$ containing $x_*$ and such that the formula
\[
P_\Sigma(x)= \underset{y\in\Sigma}{\arg\min}\,\norm{x-y}_2
\]
defines a smooth nonlinear projection operator $P_\Sigma$ from $U$ to $\Sigma\cap U$ (see {\it e.g.} \cite{Spivak:1999}). Let $P_{T_{x_*}\Sigma}$ be the orthogonal projector onto the tangent space $T_{x_*}\Sigma$ to $\Sigma$ at $x_*$. For $U$ small enough, there exists a positive constant $M$ such that it holds:
\begin{equation}
\label{eq:projector-linear-projector-difference}
\forall x \in U, \ \norm{P_\Sigma(x)-(x_*+P_{T_{x_*}\Sigma}(x-x_*))}_2 \leq \frac{M}{2} {\norm{x-x_*}_2}^2.
\end{equation}

Note that the introduction of the submanifold $\Sigma$ is needed in view of applications to quantum chemistry, as presented in Section \ref{electronic ground state calculation}; the above assumptions will be discussed in such a context.

\medskip

When the submanifold $\Sigma$ is not an affine subspace of $\mathbb{R}^n$, one may observe that the iterate $x^{(k+1)}$ generated by version A of the algorithms may in general lie outside of $\Sigma$, which does not seem appropriate. This is the main reason for the introduction of version P, starting with the CDIIS in quantum chemistry (see \cite{Pulay:1982}). On the contrary, when $\Sigma$ is an affine subspace of $\mathbb{R}^n$, it is isometric to $\mathbb{R}^{n'}$ for some integer $n'$ such that $n'\leq n$, and all the iterates of both versions lie in $\Sigma$, so there is no difference with the case $\Sigma=\mathbb{R}^n$. The theoretical results of the present paper are thus stated for version P of the algorithms when $\Sigma$ is arbitrary, and for version A assuming that $\Sigma=\mathbb{R}^n$. 

As seen in Section \ref{presentation}, the usual abstract framework for the DIIS and/or the Anderson acceleration is a special case of ours: there is no submanifold (that is, one takes $\Sigma=\mathbb{R}^n$), the integers $p$ and $n$ are the same, and the functions $f$ and $g$ are related by \eqref{assumed relation between f and g}. In this setting, Assumptions \ref{assumption 1} and \ref{assumption 2}  are satisfied as soon as the norm of $\mathrm{D}g(x_*)$ is strictly lower than one. The theoretical results of the present paper are, of course, still valid and, as far as we know, new in this more restrictive setting.\medskip

Before stating the convergence results, let us first recall the terminology associated with the different types of linear convergence (see also Appendix A of \cite{Nocedal:2006}). Let $(x^{(k)})_{k\in\mathbb{N}}$ be a sequence in a normed vector space $E$, equipped with the norm $\norm{\cdot}$, that converges to some $x_*$ in $E$. The convergence is said to be (at least) \emph{q-linear} with rate $\mu$ in $(0,1)$ if
\[
\frac{\norm{x^{(k+1)}-x_*}}{\norm{x^{(k)}-x_*}}\leq\mu,\text{ for all $k$ sufficiently large}.
\]
In addition, the convergence of the sequence is said to be (at least) \emph{r-linear} with rate $\mu$ in $(0,1)$ if there exists a positive constant $C$ such that
\[
\norm{x^{(k)}-x_*}\leq C\,\mu^k,\text{ for all $k$ sufficiently large}.
\]

In dealing with the convergence of the acceleration techniques considered in the present work, we shall say that a method is \emph{locally} q-linearly (resp. r-linearly) convergent if, for $\norm{r^{(0)}}_2$ small enough, the sequence of error vectors $(r^{(k)})_{k\in\mathbb{N}}$ converges at least q-linearly (resp. r-linearly) to the zero vector in $\mathbb{R}^p$. Note that, from Assumption \ref{assumption 2}, this implies that the sequence of iterates $(x^{(k)})_{k\in\mathbb{N}}$ converges at least r-linearly to $x_*$ in $\mathbb{R}^n$ with the same rate.

\medskip

Our first result deals with the convergence of the restarted Anderson--Pulay acceleration method.

\begin{thrm}\label{thm:linear-convergence-DIIS-with-restarts}
Let Assumptions \ref{assumption 1} and \ref{assumption 2} hold and let $\mu$ be a real number such that $K<\mu<1$. 
There exists a positive constant $R_{\mu}$ such that, for any choice of the restart parameter $\tau$ in the interval $(0,1)$, if the initial point $x^{(0)}$ in $U\cap\Sigma$ satisfies $0<\norm{r^{(0)}}_2\leq R_\mu\tau^{2p}$ and one runs version P of Algorithm \ref{alg:cdiis-restart-variant} (or version A in the case $\Sigma=\mathbb{R}^n$), the sequences $(x^{(k)})_{k\in\mathbb{N}}$ and $(r^{(k)})_{k\in\mathbb{N}}$ are well-defined and, as long as $\norm{r^{(k)}}_2$ is nonzero, one has
\begin{equation}\label{bound on m_k}
m_k\leq\min(k,p),
\end{equation}
\begin{equation}\label{bound on the extrapolation coefficients}
\norm{c^{(k)}}_\infty\leq C_{m_k}\left(1+\frac{1}{(\tau(1-\mu))^{m_k}}\right),
\end{equation}
where $C_{m_k}$ is a positive constant depending only on $m_k$, and
\begin{equation}\label{r-linear convergence with respect to restarts}
\norm{r^{(k+1)}}_2\leq\mu^{m_k+1}\norm{r^{(k-m_k)}}_2.
\end{equation}
As a consequence, one has
\begin{equation}\label{r-linear convergence global}
\forall k\in\mathbb{N},\ \norm{r^{(k)}}_2\leq\mu^k\,\norm{r^{(0)}}_2,
\end{equation}
meaning that the method is locally r-linearly convergent.

\noindent
Moreover, there exists a positive constant $\Gamma$, independent of the parameter $\tau$, such that, if a restart occurs at step $k+1$ (that is, if $m_{k+1}=0$), one has
\begin{equation}\label{estimate at restart}
\left(1-K(1+\tau)\right)\norm{r^{(k+1)}}_2\leq \left( K\tau+\frac{\Gamma}{\tau^{2m_k}}\norm{r^{(k-m_k)}}_2 \right) \norm{r^{(k-m_k)}}_2.
\end{equation}
\end{thrm}

\medskip

Our second theorem deals with the convergence of the adaptive-depth variant of the Anderson--Pulay acceleration method.

\begin{thrm}\label{thm:linear-convergence-adaptive-depth}
Let Assumptions \ref{assumption 1} and \ref{assumption 2} hold and let $\mu$ be a real number such that $K<\mu<1$. There exists a positive constant $c_\mu$ such that, for any choice of the parameter $\delta$ in the interval $(0,K)$, if the initial point $x^{(0)}$ in $U\cap\Sigma$ satisfies $0<\norm{r^{(0)}}_{2}\leq c_\mu\delta^2$ and one runs version P of Algorithm \ref{alg:cdiis-adaptive-depth-variant} (or version A in the case $\Sigma=\mathbb{R}^n$), the sequences $(x^{(k)})_{k\in\mathbb{N}}$ and $(r^{(k)})_{k\in\mathbb{N}}$ are well-defined and, as long as $\norm{r^{(k)}}_2>0$, one has
\begin{equation}\label{thm:sliding:enum:mk}
m_{k}\leq\min(k,p),
\end{equation}
\begin{equation}\label{thm:sliding:enum:ck}
\norm{c^{(k)}}_{\infty}\leq C_{m_{k}}\left(1+\left(\dfrac{\mu}{1-\mu}\right)^{m_{k}}\right),
\end{equation}
where $C_{m_k}$ is a positive constant depending only on $m_k$, and
\begin{equation}\label{thm:sliding:enum:rkmu}
\norm{r^{(k+1)}}_{2}\leq\mu\norm{r^{(k)}}_{2}.
\end{equation}
The method is thus locally q-linearly convergent.

\noindent
Moreover, 
if $k-m_k\geq 1$, then one has
\begin{equation}\label{thm:sliding:enum:rkdelta}
\norm{r^{(k)}}_{2}\leq\delta\norm{r^{(k-m_k-1)}}_{2}.
\end{equation}
\end{thrm}

In both theorems, we have given an upper bound on the extrapolation coefficients $c^{(k)}$ that depends on $m_k$ (see \eqref{bound on the extrapolation coefficients} and \eqref{thm:sliding:enum:ck}, respectively). Using the fact that $m_k$ is less than or equal to $p$, one immediately infers from these inequalities that the coefficients $c^{(k)}$ are bounded independently of $k$, which is sufficient to prove convergence. Nevertheless, this uniform estimate is rough, as it is observed in practical calculations that the depth is generally much smaller than the dimension $p$.

\smallskip

In addition to the local linear convergence estimates \eqref{r-linear convergence with respect to restarts} and \eqref{thm:sliding:enum:rkmu}, we have also provided some technical estimates on the residual -- inequalities \eqref{estimate at restart} and \eqref{thm:sliding:enum:rkdelta}, respectively -- which are crucial for the proofs of the acceleration properties in the next subsection.

\smallskip

Theorems \ref{thm:linear-convergence-DIIS-with-restarts} and \ref{thm:linear-convergence-adaptive-depth} are proved in Subsection \ref{proofs of the results}. Note that for version P, inequality \eqref{eq:g_C2_estimate} is not used in the proofs and one just needs the function $g$ to be of class $\mathscr{C}^1$. For version A, on the contrary, \eqref{eq:g_C2_estimate} is needed, but, of course, \eqref{eq:projector-linear-projector-difference} is not, since it is assumed that $\Sigma=\mathbb{R}^n$.


\subsection{Acceleration}\label{acceleration}
In this subsection, we study from a mathematical viewpoint the local acceleration properties of the proposed variants of the Anderson--Pulay acceleration. We prove that, for any choice of the real number $\lambda$ in $(0,K)$, r-linear convergence with rate $\lambda$ can be achieved, meaning that $\norm{r^{(k)}}_2=O(\lambda^k)$, with Algorithm \ref{alg:cdiis-restart-variant} (for $\tau$ and $\norm{r^{(0)}}_2$ small enough) or with Algorithm \ref{alg:cdiis-adaptive-depth-variant} (for $\delta$ and $\norm{r^{(0)}}_2$ small enough). These results are direct consequences of Theorems \ref{thm:linear-convergence-DIIS-with-restarts} and \ref{thm:linear-convergence-adaptive-depth}, respectively.

\begin{crllr} \label{cor:lambda-restart}
Let Assumptions \ref{assumption 1} and \ref{assumption 2} hold and let $\lambda$ and $\mu$ be two real numbers such that $0<\lambda<K<\mu<1$. Let $\tau$ in $(0,1)$ satisfy the smallness constraint $\tau\leq\frac{1-K}{3K}\lambda^{p+1}$. Assume that the initial point $x^{(0)}$ in $U\cap\Sigma$ satisfies $0<\norm{r^{(0)}}_2\leq\min\{R_\mu,\frac{(1-K)\lambda^{p+1}}{3\Gamma}\}\tau^{2p}$, where $R_\mu$ and $\Gamma$ are the positive constants of Theorem \ref{thm:linear-convergence-DIIS-with-restarts}. If one runs version P of Algorithm \ref{alg:cdiis-restart-variant} (or version A, assuming that $\Sigma=\mathbb{R}^n$), then the sequences $(x^{(k)})_{k\in\mathbb{N}}$ and $(r^{(k)})_{k\in\mathbb{N}}$ are well-defined, with $m_k\leq \min(k,p)$ and
\[
\forall k\in\mathbb{N},\ \norm{r^{(k)}}_2\,\leq\,\mu^{m_k}\lambda^{k-m_k}\norm{r^{(0)}}_2\,\leq\,\mu^{\min(k,p)}\lambda^{\max(0,k-p)}\norm{r^{(0)}}_2.
\]
As a consequence, the sequence $(x^{(k)})_{k\in\mathbb{N}}$ converges at least r-linearly to $x_*$ with rate $\lambda$.
\end{crllr}

\begin{proof}
First of all, since $0<\norm{r^{(0)}}_2\leq R_\mu\tau^{2p}$, 
the conclusions of Theorem \ref{thm:linear-convergence-DIIS-with-restarts} hold. In particular, the sequences $(x^{(k)})_{k\in\mathbb{N}}$ and $(r^{(k)})_{k\in\mathbb{N}}$ are well-defined, with $m_k\leq\min\{k,p\}$ and
$\norm{r^{(k-m_k)}}_2\leq \mu^k\norm{r^{(0)}}_2$. From the other smallness condition on $\norm{r^{(0)}}_2$, we also have $\norm{r^{(k-m_k)}}_2\leq \frac{(1-K)\lambda^{p+1}}{3\Gamma}\tau^{2p}\leq \frac{(1-K)\lambda^{p+1}}{3\Gamma}\tau^{2 m_k}$. Finally, the parameter $\tau$ is such that $1-K(1+\tau) \geq \frac{2}{3}(1-K)$ and $K\tau \leq \frac{(1-K)\lambda^{p+1}}{3}$. Using these facts, we infer the following property from estimate \eqref{estimate at restart} of Theorem \ref{thm:linear-convergence-DIIS-with-restarts}:
\begin{equation}\label{accel1}
\text{if }m_{k+1}=0,\text{ then }\norm{r^{(k+1)}}_2\leq\lambda^{p+1}\,\norm{r^{(k-m_k)}}_2.
\end{equation}

Using \eqref{r-linear convergence with respect to restarts} and \eqref{accel1}, we are going to prove by complete induction that, for any natural integer $k$, there holds
\[
\norm{r^{(k)}}_2\leq\mu^{m_k}\lambda^{k-m_k}\norm{r^{(0)}}_2.
\]
Let us denote by ${\mathcal P(k)}$ the above property at rank $k$. Obviously, ${\mathcal P(0)}$ is true. Let us assume that the property is satisfied up to rank $k$, with $k$ a natural integer. We will establish that ${\mathcal P}(k+1)$ holds by distinguishing two cases.

First, if $m_{k+1}\geq 1$, then, using inequality \eqref{r-linear convergence with respect to restarts}, we may write that
\[
\norm{r^{(k+1)}}_2\leq\mu^{m_{k+1}}\norm{r^{(k+1-m_{k+1})}}_2.
\]
The natural integer $k+1-m_{k+1}$ being less than or equal to $k$, property ${\mathcal P(k+1-m_{k+1})}$ holds and, since $m_{k+1-m_{k+1}}=0$, it takes the form
\[
\norm{r^{(k+1-m_{k+1})}}_2\leq\lambda^{k+1-m_{k+1}}\norm{r^{(0)}}_2.
\]
We then conclude by combining the two above estimates.

\smallskip

Second, if $m_{k+1}=0$, remembering that $m_k\leq p$ and $\lambda<1$, we infer from  inequality \eqref{accel1} that
\[
\norm{r^{(k+1)}}_2\leq\lambda^{p+1}\norm{r^{(k-m_{k})}}_2\leq\lambda^{m_k+1}\norm{r^{(k-m_{k})}}_2.
\]
The integer $k-m_{k}$ being less than or equal to $k$, property ${\mathcal P(k-m_{k})}$ holds and, since $m_{k-m_{k}}=0$, it follows that
\[
\norm{r^{(k-m_{k})}}_2\leq\lambda^{k-m_{k}}\,\norm{r^{(0)}}_2,
\]
so that one reaches
\[
\norm{r^{(k+1)}}_2\leq\lambda^{m_k+1}\lambda^{k-m_{k}}\,\norm{r^{(0)}}_2=\lambda^{k+1}\,\norm{r^{(0)}}_2.
\]

\smallskip

Estimate ${\mathcal P(k)}$ thus holds for any natural integer $k$. Moreover, since $m_k\leq \min\{k,p\}$ and $\lambda<\mu$, we immediately see that $\lambda^{k-m_k}\mu^{m_k}<\lambda^{\max(0,k-p)}\mu^{\min(k,p)}$, which ends the proof.
\end{proof}

\medskip

Note that our theoretical smallness requirements on $\tau$ and $\norm{r^{(0)}}_2$ are unreasonably restrictive and certainly not representative of what is observed in applications. We will indeed see in Section \ref{numerical experiments} that acceleration is commonly achieved in practice.

\medskip

A similar result for acceleration with Algorithm~\ref{alg:cdiis-adaptive-depth-variant} follows from Theorem~\ref{thm:linear-convergence-adaptive-depth}.

\begin{crllr}\label{cor:lambda-sliding}
Let Assumptions \ref{assumption 1} and \ref{assumption 2} hold and let $\lambda$ and $\mu$ be two real numbers such that $0<\lambda<K<\mu<1$. Choose $\delta$ so that $0<\delta\leq \lambda^{p+1}$. Suppose that the initial point $x^{(0)}\in U\cap\Sigma$ satisfies $0<\norm{r^{(0)}}_2\leq c_\mu \delta^{2}$ where $c_\mu$ is the same as in Theorem \ref{thm:linear-convergence-adaptive-depth}, and run version P of Algorithm \ref{alg:cdiis-adaptive-depth-variant} (or version A in the case $\Sigma=\mathbb{R}^n$). Denote $k_{\max}=\max\{k\in\mathbb{N}\,|\,m_k=k\}$. Then one has $k_{\max}\leq p$ and, for all $k$ in $\mathbb{N}$,
\[
\norm{r^{(k)}}_2\leq \mu^{\min(k,k_{\max})}\lambda^{\max(0,k-k_{\max})}\norm{r^{(0)}}_2\leq \mu^{\min(k,p)}\lambda^{\max(0,k-p)}\norm{r^{(0)}}_2.
\]
As a consequence, $(x^{(k)})_{k\in\mathbb{N}}$ converges at least r-linearly to $x_*$ with rate $\lambda$.
\end{crllr}
\begin{proof}
Since $\delta<\lambda<K$ and $0<\norm{r^{(0)}}_2\leq c_\mu \delta^{2}$, the conclusions of Theorem \ref{thm:linear-convergence-adaptive-depth} hold. In particular, from bound \eqref{thm:sliding:enum:mk}, one has $k_{\max}\leq p$, hence $\lambda^{\max(0,k-k_{\max})}\leq \mu^{\min(k,p)}\lambda^{\max(0,k-p)}$ for any natural integer $k$. As a result, we just need to prove that the following estimate
\begin{equation}\label{accel-adaptive-depth}
\norm{r^{(k)}}_2\leq \mu^{\min(k,k_{\max})}\lambda^{\max(0,k-k_{\max})}\norm{r^{(0)}}_2
\end{equation}
holds for any natural integer $k$. To do this, we argue by complete induction.

\smallskip

On the one hand, if $k\leq k_{\max}$, inequality \eqref{accel-adaptive-depth} reduces to an estimate of r-linear convergence with rate $\mu$, which follows immediately from the q-linear convergence property \eqref{thm:sliding:enum:rkmu} of Theorem \ref{thm:linear-convergence-adaptive-depth}.

On the other hand, let $k\geq k_{\max}$ be such that $\eqref{accel-adaptive-depth}$ holds for all the natural integers less than or equal to $k$. From the definition of $k_{\max}$, one has $k+1-m_{k+1}\geq 1$, and, since $\delta\leq\lambda^{p+1}$, $m_{k+1}\leq p$ and $\lambda<1$, estimate \eqref{thm:sliding:enum:rkdelta} implies that
\[
\norm{r^{(k+1)}}_{2}\leq\lambda^{m_{k+1}+1}\norm{r^{(k-m_{k+1})}}_{2}.
\]
The natural integer  $k-m_{k+1}$ being less than or equal to $k$, one also has
\[
\norm{r^{(k-m_{k+1})}}_2\leq\mu^{\min(k-m_{k+1},k_{\max})}\lambda^{\max(0,k-m_{k+1}-k_{\max})}\norm{r^{(0)}}_2,
\]
and, combining the last two estimates, one gets
\[ 
\norm{r^{(k+1)}}_{2}\leq\mu^{\min(k-m_{k+1},k_{\max})}\lambda^{\max(m_{k+1}+1,k+1-k_{\max})}\norm{r^{(0)}}_2.
\]
In order to study the right-hand side of this inequality, we distinguish two cases.

If $k-m_{k+1}> k_{\max}$, then it holds $\mu^{\min(k-m_{k+1},k_{\max})}\lambda^{\max(m_{k+1}+1,k+1-k_{\max})}=\mu^{k_{\max}}\lambda^{k+1-k_{\max}}$, so that estimate \eqref{accel-adaptive-depth} holds for $k+1$. Otherwise, if $k-m_{k+1}\leq k_{\max}$, then it holds $\mu^{\min(k-m_{k+1},k_{\max})}\lambda^{\max(m_{k+1}+1,k+1-k_{\max})}=\mu^{k-m_{k+1}}\lambda^{m_{k+1}+1}$. Since $\lambda<\mu$, one has $\mu^{k-m_{k+1}}\lambda^{m_{k+1}+1}\leq \mu^{k_{\max}}\lambda^{k+1-k_{\max}}$ and estimate \eqref{accel-adaptive-depth} holds for $k+1$.

\smallskip

We have thus shown that estimate $\eqref{accel-adaptive-depth}$ holds for any natural integer $k$, ending the proof.
\end{proof}

\subsection{Reaching superlinear convergence}\label{superlinear}
Let us recall some standard notions of superlinear convergence (see {\it e.g.} \cite{Nocedal:2006,Potra:1989}). We shall say that a sequence $(x^{(k)})_{k\in\mathbb{N}}$ in a normed space $E$, equipped with the norm $\norm{\cdot}$, converges \emph{q-superlinearly} to $x_*$ in $E$ if
\[
\lim_{k\to+\infty}\frac{\norm{x^{(k+1)}-x_*}}{\norm{x^{(k)}-x_*}}=0.
\]
Let $\theta$ denote a real number greater than one. It is said that the superlinear convergence occurs with \emph{q-order at least $\theta$} if there exists a positive constant $a$ such that
\[
\frac{\norm{x^{(k+1)}-x_*}}{\norm{x^{(k)}-x_*}^{\theta}}\leq a,\text{ for all $k$ sufficiently large}.
\]
More generally, the sequence is said to converge \emph{r-superlinearly} to $x_*$ in $E$ if there exists a sequence of positive numbers $(\epsilon_k)_{k\in\mathbb{N}}$ converging q-superlinearly to zero and such that
\[
\norm{x^{(k)}-x_*}\leq \epsilon_k,\text{ for all $k$ sufficiently large}.
\]
If one has additionally 
\[
\forall k\in\mathbb{N},\ \epsilon_k=b\,\eta^{\theta^k}
\]
for some positive constant $b$ and $\eta$ in $(0,1)$, the sequence is said to converge superlinearly with \emph{r-order at least~$\theta$}. 

\smallskip

In some references, one can find mentions of the DIIS exhibiting a local superlinear convergence behaviour. Rohwedder and Schneider discussed in \cite{Rohwedder:2011} the circumstances under which this may occur for the DIIS method. We now propose a modification of the restarted (resp. adaptive-depth) Anderson--Pulay acceleration, in which the value of the parameter $\tau$ (resp. $\delta$) is slowly decreased along the iteration. We will show that the resulting sequence of approximations locally converges r-superlinearly to the fixed point.

\medskip

Let us first describe the modification made to Algorithm \ref{alg:cdiis-restart-variant}. One may observe that estimate \eqref{estimate at restart} in Theorem \ref{thm:linear-convergence-DIIS-with-restarts} allows a local r-superlinear convergence result, by carefully changing the value of the parameter $\tau$ after each restart. As a consequence, we modify the algorithm as follows: first, in the list of data, the positive constant $\tau$ is replaced by two positive constants $T_0$ and $\zeta$~; then, we replace the boolean test ``$\tau\,\dnorm{s^{(k+1)}}_2>\dnorm{(\id-\Pi_k)s^{(k+1)}}_2$'' by ``$T_0{\dnorm{r^{(k-m_k)}}_2}^\zeta \,\dnorm{s^{(k+1)}}_2>\dnorm{(\id-\Pi_k)s^{(k+1)}}_2$''. In other words, the fixed parameter $\tau$ in the test is replaced by an adaptive one, $\tau^{(k-m_k)}=T_0{\dnorm{r^{(k-m_k)}}_2}^\zeta$, which becomes smaller and smaller along the iteration. Naming this variant Algorithm \ref{alg:cdiis-restart-variant}', we have the following result.

\begin{thrm}\label{thm:superlinear restart}
Let Assumptions \ref{assumption 1} and \ref{assumption 2} hold and let $\mu$ be a real number such that $K<\mu<1$. Choose a positive $T_0$ and $\zeta$ in $(0,\frac{1}{2p})$. Then, there exists a positive constant $\varepsilon(T_0,\zeta,\mu)$ such that, if we run version P of Algorithm \ref{alg:cdiis-restart-variant}' (or version A in the case $\Sigma=\mathbb{R}^n$) with an initial point $x^{(0)}$ in $U\cap\Sigma$ satisfying $0<\norm{r^{(0)}}_2<\varepsilon(T_0,\zeta,\mu)$, then the sequence $(x^{(k)})_{k\in\mathbb{N}}$ is well-defined, satisfies estimates \eqref{bound on m_k} and \eqref{r-linear convergence global}, and converges superlinearly to $x_*$ with r-order at least $\left(1+\min\{\zeta,1-2p\zeta\}\right)^{\frac{1}{p}}$.
\end{thrm}

For Algorithm~\ref{alg:cdiis-adaptive-depth-variant}, the changes are the following: first, in the list of data, we replace the positive constant $\delta$ by two positive constants $D_0$ and $\xi$~; then, we replace the boolean test ``$\delta\,\norm{r^{(i)}}_2<\norm{r^{(k+1)}}_2$'' by ``$D_0\dnorm{r^{(i)}}_2^{1+\xi}<\norm{r^{(k+1)}}_2$''. That is to say, the fixed parameter $\delta$ in the test is replaced by the adaptive one, $\delta^{(i)}=D_0\dnorm{r^{(i)}}_2^\xi$. We name this variant Algorithm~\ref{alg:cdiis-adaptive-depth-variant}' and state the following result.

\begin{thrm}\label{thm:superlinear adaptive}
Let Assumptions \ref{assumption 1} and \ref{assumption 2} hold and let $\mu$ be a real number such that $K<\mu<1$. Choose a positive constant $D_0$ and $\xi$ in $(0,\sqrt{2}-1]$. In the limit case $\xi=\sqrt{2}-1$, assume in addition that $D_0\geq \Delta_\mu$, where $\Delta_\mu$ is a suitable positive constant. 
Then, there exists a positive constant $\varepsilon(D_0,\xi,\mu)$ such that, if we run version P of Algorithm~\ref{alg:cdiis-adaptive-depth-variant}' (or version A in the case $\Sigma=\mathbb{R}^n$) with an initial point $x^{(0)}$ in $U\cap\Sigma$ satisfying $0<\norm{r^{(0)}}_2<\varepsilon(D_0,\xi,\mu)$, the sequence $(x^{(k)})_{k\in\mathbb{N}}$ is well-defined, satisfies estimates \eqref{thm:sliding:enum:mk}, \eqref{thm:sliding:enum:ck}, and \eqref{thm:sliding:enum:rkmu}, and converges superlinearly to $x_*$ with r-order at least $\left(1+\xi\right)^{\frac{1}{p+1}}$.
\end{thrm}

Both theorems \ref{thm:superlinear restart} and \ref{thm:superlinear adaptive} are proved at the end of Subsection \ref{proofs of the results}.

\medskip

Note that the best theoretical value of the r-order given by Theorem \ref{thm:superlinear restart} is $\left(1+\frac{1}{2p+1}\right)^{\frac{1}{p}}$. It corresponds to the choice $\zeta=\frac{1}{2p+1}$. For Theorem \ref{thm:superlinear adaptive}, the best theoretical value is $2^{\frac{1}{2p+2}}$, corresponding to $\xi=\sqrt{2}-1$. In both cases, it is very close to one when $p$ is large.

\smallskip

In the numerical experiments reported in the last part of the paper, we only ran Algorithms \ref{alg:cdiis-restart-variant} and  \ref{alg:cdiis-adaptive-depth-variant} with fixed values of $\tau$ and $\delta$ and obtained excellent acceleration results for self-consistent field iterations. Nevertheless, Theorems \ref{thm:superlinear restart} and \ref{thm:superlinear adaptive} suggest that it may be beneficial in practical calculations to decrease the value of the parameters $\tau$ and $\delta$ in an adaptive way along the iteration. We plan to investigate this in the future.

\subsection{Proofs of the theorems}\label{proofs of the results}
For 
linear problems in finite dimension, if $g$ is a contraction, then it is well-known that the DIIS (or the Anderson acceleration) converges to the exact solution in a finite number of steps (see Subsection \ref{equivalence with GMRES in the linear case}).
In the nonlinear case, this is no longer true due to the presence of quadratic terms. Despite these additional terms, convergence at any given linear rate should be allowed, on the condition that the depth is large enough and one starts close enough to the solution. Then, if one allows the depth to grow slowly along the process, superlinear convergence should occur.

To prove these properties, one has to control the quadratic errors and a bound on the size of the extrapolation coefficients at each step is needed. This amounts to quantitatively measuring the affine independence of the error vectors, since the optimal coefficients are solution to a least-squares problem whose solubility is directly related to this independence.

In the existing literature, it is generally\footnote{An exception is made in the paper by Zhang \textit{et al.} \cite{Zhang:2020}, where a bound on the coefficients is shown for a restarted type-I Anderson acceleration method (the DIIS rather corresponds to a type-II Anderson acceleration). The authors then prove a global linear convergence result when $f$ is the gradient of a convex functional, but the linear rate they obtain is no better than the rate of the basic iteration process.} assumed that the coefficients remain bounded throughout the iteration. While we do not know how to prove \textit{a priori} such an estimate for the ``classical'' fixed-depth Anderson--Pulay acceleration, the mechanisms employed in the restarted and adaptive-depth variants we study allow to derive one. A key theoretical ingredient is that the problem for the extrapolation coefficients will become poorly conditioned when the norm of the last stored error vector is much smaller than that of the oldest one. This phenomenon is rigorously described in Lemma \ref{lem:induction step} below, which constitutes the main technical tool in our proofs of both convergence and acceleration of the method.

In the restarted Anderson--Pulay acceleration, the role of the parameter $\tau$ is to control the affine independence of the error vectors. Lemma \ref{lem:induction step} shows that, as long as a restart does not occur, the least-squares problem remains well-conditioned and the size of the coefficients is bounded, whereas the norm of the last stored error vector is necessarily much smaller than that of the oldest one when it does. Since there must be a restart after at most $p+1$ consecutive iterations, convergence with acceleration can be established.

In the adaptive-depth version of the Anderson--Pulay acceleration, the parameter $\delta$ is used to eliminate the stored iterates which are not ``relevant'' enough when compared with the most recent one. This criterion does not directly quantify the independence of the error vectors, and it is certainly not good when the initial guess is chosen far from the solution, since a large number of iterates will be kept in that case. However, starting close enough to the solution, Lemma \ref{lem:induction step} allows to inductively prove a bound on the extrapolation coefficients, for reasons similar to those invoked with the restarted variant. Indeed, at each step, either the stored error vectors are affinely independent, and the extrapolation coefficients are bounded, or the norm of the last stored error vector is smaller than $\delta$ times that of some of the oldest ones. In the latter case, the criterion discards the corresponding iterates, which results in a restoration of the condition number associated with the least-squares problem at the next step. In addition, observing that a stored iterate is dismissed at least once in every $p+1$ consecutive iterations, it is inferred that accelerated convergence is possible for a parameter $\delta$ chosen small enough.

\subsubsection{A preliminary lemma}
We first introduce some notations. We recall that the set $U$, introduced in Subsection \ref{linear convergence}, is a small neighborhood of $x_*$ such that the estimates \eqref{mean value inequalities}, \eqref{eq:f_C2_estimate}, \eqref{eq:g_C2_estimate}, and \eqref{eq:projector-linear-projector-difference}
hold. For any natural integers $k$ and $m_k$ such that $k\geq m_k\geq 1$, consider a family $(x^{(k-m_k)},\dots,x^{(k)})$ of vectors in $U\cap\Sigma\setminus\{x_*\}$. For any integer $i$ in $\{0,\dots,m_k\}$, set $r^{(k-m_k+i)}=f(x^{(k-m_k+i)})$, and, for any integer $\ell$ in $\{0,\dots,m_{k}\}$, let us denote by $\mathcal{A}_{\ell}^{(k)}$ the affine span of $\{r^{(k-m_k)},\dots,r^{(k-m_k+\ell)}\}$, that is
\[\mathcal{A}_{\ell}^{(k)}=\operatorname{Aff}\left\{r^{(k-m_{k})},r^{(k-m_{k}+1)},\dots,r^{(k-m_k+\ell)}\right\}=\left\{r=\sum_{i=0}^{\ell}c_i\,r^{(k-m_k+i)}\,|\,\sum_{i=0}^{\ell}c_i=1\right\},\]
and, for any integer $\ell$ in $\{1,\dots,m_k\}$, by $d_{\ell}^{(k)}$ the distance between $r^{(k-m_k+\ell)}$ and $\mathcal{A}_{\ell-1}^{(k)}$, that is
\[d_{\ell}^{(k)}=\min_{r\in\mathcal{A}_{\ell-1}^{(k)}}\norm{r^{(k-m_k+\ell)}-r}_2.\]

\begin{lmm}\label{lem:induction step}
Let $k$ and $m_k$ be two integers such that $k\geq m_k\geq 1$, let $(x^{(k-m_k)},\dots,x^{(k)})$ be a family of vectors in $U\cap\Sigma\setminus\{x_*\}$ and $t$ be a positive real number. For any integer $i$ in $\{0,\dots,m_k\}$, set $r^{(k-m_k+i)}=f(x^{(k-m_k+i)})$. Assume that
\begin{equation}\label{lower bound on distance}
\forall i\in\{1,\dots,m_k\},\ d^{(k)}_i\geq t\max_{j\in\{i,\dots,m_k\}}\norm{r^{(k-m_k+j)}}_2.
\end{equation}
Then, the vectors $r^{(k-m_k)},\dots,r^{(k)}$ are affinely independent, so that one has $m_k\leq p$ and there is a unique $\tilde{c}$ in $\mathbb{R}^{m_k+1}$ such that $\,\sum_{i=0}^{m_k}\tilde{c}_i=1$ and $\norm{\sum_{i=0}^{m_k}\tilde{c}_ir^{(k-m_k+i)}}_2=\mathrm{dist}_2(0,\mathcal{A}^{(k)}_{m_k})$. Moreover, there exists a positive constant $C_{m_k}$ such that
\begin{equation}\label{bound on the extrapolation coefficients, general setting}
\norm{\tilde{c}}_\infty\leq C_{m_k}\,(1+t^{-m_k}).
\end{equation}

Under the additional assumption that $\tilde{x}^{(k+1)}=\sum_{i=0}^{m_k}\tilde{c}_i\,x^{(k-m_k+i)}$ belongs to $U$, as well as $g(\tilde{x}^{(k+1)})$ and $\sum_{i=0}^{m_k}\tilde{c}_i\,g(x^{(k-m_k+i)})$, define 
\[
x^{(k+1)}=g(\tilde{x}^{(k+1)})\ \text{ (version P) }\ \text{ or  }\ x^{(k+1)}=\sum_{i=0}^{m_k}\tilde{c}_i\,g(x^{(k-m_k+i)})\ \text{ (version A)}
\]
and set $r^{(k+1)}=f(x^{(k+1)})$, $d^{(k)}_{m_k+1}=\mathrm{dist}_2(r^{(k+1)},\mathcal{A}^{(k)}_{m_k})$. Then, there exists a positive constant $\kappa$
such that
\begin{equation}\label{lem:inductionRlast}
\norm{r^{(k+1)}}_2\leq K\,\mathrm{dist}_2(0,\mathcal{A}^{(k)}_{m_k})+\kappa(1+t^{-2m_k})\,\max_{i\in\{0,\dots,m_k\} }{\norm{r^{(k-m_k+i)}}_2}^2
\end{equation}
and
\begin{equation}\label{lem:induction1KRlast}
(1-K)\,\norm{r^{(k+1)}}_2\leq K\,d^{(k)}_{m_k+1}+\kappa(1+t^{-2m_k})\,\max_{i\in\{0,\dots,m_k\} }{\norm{r^{(k-m_k+i)}}_2}^2.
\end{equation}
\end{lmm}
\begin{proof}
Set $s^{(i)}=r^{(k-m_k+i)}-r^{(k-m_k+i-1)}$, for any integer $i$ in $\{1,\dots,m_k\}$, $q^{(1)}=s^{(1)}$ and $q^{(i)}=(\id-\Pi_{\vec{\mathcal{A}}^{(k)}_{i-1}})s^{(i)}$, for any integer $i$ in $\{2,\dots,m_k\}$, where $\Pi_{\vec{\mathcal{A}}^{(k)}_{i-1}}$ denotes the orthogonal projector onto $\vec{\mathcal{A}}^{(k)}_{i-1}$, the underlying vector space of $\mathcal{A}^{(k)}_{i-1}$. Then, the vectors $q^{(1)},\dots,q^{(m_k)}$ are mutually orthogonal and, for any integer $i$ in $\{1,\dots,m_k\}$, one has $\norm{q^{(i)}}_2=d^{(k)}_i$.\\
We may write
\[
\sum_{i=0}^{m_k}\tilde{c}_i\,r^{(k-m_k+i)}=r^{(k)}+\sum_{i=1}^{m_k}\tilde{\zeta}_is^{(i)}=r^{(k)}+\sum_{i=1}^{m_k}\tilde{\lambda}_iq^{(i)},
\]
with $\tilde{\lambda}_i=-\frac{(q^{(i)})^\top r^{(k)}}{(d^{(k)}_i)^2}$, so that $\abs{\tilde{\lambda}_i}\leq\frac{\norm{r^{(k)}}_2}{d^{(k)}_i}\leq\frac{1}{t}$, due to lower bound \eqref{lower bound on distance}.
\medskip

\noindent
On the other hand, $\tilde{\zeta}=P\tilde{\lambda}$, where $P$ is the change-of-basis matrix from $\{s^{(i)}\}_{i=1,\dots,m_k}$ to $\{q^{(i)}\}_{i=1,\dots,m_k}$. We need a bound on $\norm{P}$. For that purpose, we introduce the matrix factorisation $P=P^{(2)}P^{(3)}\dots P^{(m_k)}$ where, for any integer $j$ in $\{2,\dots,m_k\}$, $P^{(j)}$ is the change-of-basis matrix from $\{q^{(1)},\dots,q^{(j-1)},s^{(j)},\dots,s^{(m_k)}\}$ to $\{q^{(1)},\dots,q^{(j)},s^{(j+1)},\dots,s^{(m_k)}\}$. Since $q^{(1)}=s^{(1)}$, and $q^{(j)}=s^{(j)}-\sum_{i=1}^{j-1}\frac{(s^{(j)})^\top q^{(i)}}{(q^{(i)})^\top q^{(i)}}\,q^{(i)}$, we have
\[
\forall i\in\{1,\dots,m_k\},\ (P^{(j)})_{ii}=1\text{ and},\ \forall i\in\{1,\dots,j-1\},\ (P^{(j)})_{ij}=-\frac{(s^{(j)})^\top q^{(i)}}{(q^{(i)})^\top q^{(i)}},
\]
all the other coefficients of this matrix being zero. It follows that
\[
\forall i\in\{1,\dots,j-1\},\ \abs{(P^{(j)})_{ij}}\leq\frac{\norm{s^{(j)}}_2}{d^{(k)}_i}\leq \frac{\norm{r^{(k-m_k+j)}}_2+\norm{r^{(k-m_k+j-1)}}_2}{d^{(k)}_i}\leq\frac{2}{t}.
\]
As a consequence, there holds an estimate of the form
\[
\norm{P^{(j)}}_\infty\leq C(1+t^{-1})
\]
for some positive constant $C$ depending only on $m_k$, which thus yields
\[
\norm{P}_\infty\leq \norm{P^{(2)}}_\infty\norm{P^{(3)}}_\infty\dots \norm{P^{(m_k)}}_\infty\leq C^{m_k-1}(1+t^{-1})^{m_k-1}.
\]
Hence, it follows that
\[
\norm{\tilde{\zeta}}_\infty\leq C^{m_k-1}(1+t^{-1})^{m_k-1}\norm{\tilde{\lambda}}_\infty\leq C^{m_k-1}(1+t^{-1})^{m_k-1}t^{-1}.
\]
Finally, one has $\tilde{c}_0=-\tilde{\zeta}_1$, $\tilde{c}_i=\tilde{\zeta}_{i}-\tilde{\zeta}_{i+1}$, $1\leq i\leq m_k-1$, and $\tilde{c}_{m_k}=1+\tilde{\zeta}_{m_k}$, so that $\norm{\tilde{c}}_\infty\leq C_{m_k}(1+t^{-m_k})$, for some positive constant $C_{m_k}$ depending only on $m_k$, thus proving bound \eqref{bound on the extrapolation coefficients, general setting} on the coefficients.

\smallskip

Next, we remark that estimate \eqref{lem:induction1KRlast} is an immediate consequence of \eqref{lem:inductionRlast}. Indeed, by the triangle inequality, one has
\[
\mathrm{dist}_2(0,\mathcal{A}^{(k)}_{m_k})\leq\norm{r^{(k+1)}}_2+d^{(k)}_{m_k+1}.
\]

It only remains to prove the bound \eqref{lem:inductionRlast} on $\norm{r^{(k+1)}}_2$. As a first step, we are going to estimate the norm of $\tilde{r}^{(k+1)}=f(\tilde{x}^{(k+1)})$. One has
\[
\norm{\tilde{r}^{(k+1)}-\sum_{i=0}^{m_k}\tilde{c}_ir^{(k-m_k+i)}}_2\leq\norm{f(\tilde{x}^{(k+1)})-\mathrm{D}f(x_*)(\tilde{x}^{(k+1)}-x_*)}_2+\norm{\sum_{i=0}^{m_k}\tilde{c}_if(x^{(k-m_k+i)})-\mathrm{D}f(x_*)(\tilde{x}^{(k+1)}-x_*)}_2 
\]
so that, using inequality~\eqref{eq:f_C2_estimate} and the fact that the coefficients $(\tilde{c}_i)_{0\leq i\leq m_k}$ sum to $1$, 
\begin{align*}
\norm{\tilde{r}^{(k+1)}-\sum_{i=0}^{m_k}\tilde{c}_i r^{(k-m_k+i)}}_2&\leq\frac{L}{2}\norm{\sum_{i=0}^{m_k}\tilde{c}_i(x^{(k-m_k+i)}-x_*)}_2^2+\norm{\sum_{i=0}^{m_k}\tilde{c}_i(f(x^{(k-m_k+i)})-\mathrm{D}f(x_*)(x^{(k-m_k+i)}-x_*))}_2\nonumber\\
&\leq\frac{L}{2}(m_k+1)\norm{\tilde{c}}_\infty((m_k+1)\norm{\tilde{c}}_\infty+1)\underset{i\in\{0,\dots,m_k\}}{\max}\,{\norm{x^{(k-m_k+i)}-x_*}_2}^2. 
\end{align*}
It thus follows from the definition of the coefficients $\tilde{c}_i$ that
\begin{align*}
\norm{\tilde{r}^{(k+1)}}_2&\leq\norm{\tilde{r}^{(k+1)}-\sum_{i=0}^{m_k}\tilde{c}_i\,r^{(k-m_k+i)}}_2+\mathrm{dist}_2(0,\mathcal{A}^{(k)}_{m_k})\\
&\leq\frac{L}{2}(m_k+1) \norm{\tilde{c}}_\infty((m_k+1)\norm{\tilde{c}}_\infty+1)\underset{i\in\{0,\dots,m_k\}}{\max}\,{\norm{x^{(k-m_k+i)}-x_*}_2}^2+\mathrm{dist}_2(0,\mathcal{A}^{(k)}_{m_k}).
\end{align*}
Using bound \eqref{bound on the extrapolation coefficients, general setting}, we find a constant $C'_{m_k}>0$ independent of $t$, such that
\begin{equation}\label{C'}
(m_k+1) \norm{\tilde{c}}_\infty((m_k+1)\norm{\tilde{c}}_\infty+1)\leq C'_{m_k} (1+t^{-2m_k}).
\end{equation}
So, using Assumption~\ref{assumption 2}, we finally obtain
\begin{equation}\label{common estimate}
\norm{\tilde{r}^{(k+1)}}_2\leq\frac{L}{2\sigma^2} C'_{m_k}(1+t^{-2m_k})\underset{i\in\{0,\dots,m_k\}}{\max}\,{\norm{r^{(k-m_k+i)}}_2}^2+\mathrm{dist}_2(0,\mathcal{A}^{(k)}_{m_k}).
\end{equation}

In the remaining part of the proof, we treat separately version P and version A.

\begin{itemize}
\item \textbf{Version P.} Recall that, in this version, the set $\Sigma$ is an arbitrary smooth submanifold and one takes $x^{(k+1)}=g(\tilde{x}^{(k+1)})$. An estimate on the distance between $\tilde{x}^{(k+1)}$ and $\Sigma$ is needed. By the triangle inequality, one has
\[
\norm{\tilde{x}^{(k+1)}-P_\Sigma(\tilde{x}^{(k+1)})}_2\leq\norm{\tilde{x}^{(k+1)}-x_*-P_{T_{x_*}\Sigma}(\tilde{x}^{(k+1)}-x_*)}_2+\norm{x_*+P_{T_{x_*}\Sigma}(\tilde{x}^{(k+1)}-x_*)-P_\Sigma(\tilde{x}^{(k+1)})}_2,
\]
and we will thus get bounds for both contributions in the right-hand side of this inequality. For the first one, remembering that $x^{(k-m_k+i)}=P_{\Sigma}(x^{(k-m_k+i)})$ for any integer $i$ in $\{0,\dots,m_k\}$ and using \eqref{eq:projector-linear-projector-difference}, we reach
\begin{align*}
\norm{\tilde{x}^{(k+1)}-x_*-P_{T_{x_*}\Sigma}(\tilde{x}^{(k+1)}-x_*)}_2&=\norm{\sum_{i=0}^{m_k}\tilde{c}_i\,(\id-P_{T_{x_*}\Sigma})(x^{(k-m_k+i)}-x_*)}_2\\
&\leq \sum_{i=0}^{m_k}\abs{\tilde{c}_i}\norm{(\id-P_{T_{x_*}\Sigma})(x^{(k-m_k+i)}-x_*)}_2\\
&=\sum_{i=0}^{m_k}\abs{\tilde{c}_i}\norm{P_\Sigma(x^{(k-m_k+i)})-x_*-P_{T_{x_*}\Sigma}(x^{(k-m_k+i)}-x_*)}_2\\
&\leq\frac{M}{2}\sum_{i=0}^{m_k}\abs{\tilde{c}_i}{\norm{x^{(k-m_k+i)}-x_*}_2}^2\\
&\leq\frac{M}{2}(m_k+1)\norm{\tilde{c}}_\infty\underset{i\in\{0,\dots,m_k\}}{\max}\,{\norm{x^{(k-m_k+i)}-x_*}_2}^2.
\end{align*}
For the second term, we use again \eqref{eq:projector-linear-projector-difference} to obtain
\begin{align*}
\norm{x_*+P_{T_{x_*}\Sigma}(\tilde{x}^{(k+1)}-x_*)-P_\Sigma(\tilde{x}^{(k+1)})}_2&\leq\frac{M}{2}{\norm{\tilde{x}^{(k+1)}-x_*}_2}^2\\
&\leq\frac{M}{2}(m_k+1)^2\norm{\tilde{c}}_\infty^2\underset{i\in\{0,\dots,m_k\}}{\max}\,{\norm{x^{(k-m_k+i)}-x_*}_2}^2.
\end{align*}
Adding these two estimates, using Assumption \ref{assumption 2} and bound \eqref{C'}, we find
\begin{align*}
\norm{\tilde{x}^{(k+1)}-P_\Sigma(\tilde{x}^{(k+1)})}_2 &\leq\frac{M}{2}C'_{m_k}(1+t^{-2m_k}) \underset{i\in\{0,\dots,m_k\}}{\max}\,{\norm{x^{(k-m_k+i)}-x_*}_2}^2\\
&\leq\frac{M}{2\sigma^2} C'_{m_k}(1+t^{-2m_k})\underset{i\in\{0,\dots,m_k\}}{\max}\,{\norm{r^{(k-m_k+i)}}_2}^2.
\end{align*}
Finally, using Assumption \ref{assumption 1} and combining the above estimate with inequalities \eqref{mean value inequalities} and \eqref{common estimate}, we get
\begin{align*}
\norm{r^{(k+1)}}_2&\leq\norm{f(g(P_\Sigma(\tilde{x}^{(k+1)})))}_2+2\,\norm{\mathrm{D}f(x_*)\circ\mathrm{D}g(x_*)}_2\norm{\tilde{x}^{(k+1)}-P_\Sigma(\tilde{x}^{(k+1)})}_2\\
&\leq K\,\norm{f(P_\Sigma(\tilde{x}^{(k+1)}))}_2+2\,\norm{\mathrm{D}f(x_*)\circ\mathrm{D}g(x_*)}_2\norm{\tilde{x}^{(k+1)}-P_\Sigma(\tilde{x}^{(k+1)})}_2\\
&\leq K\,\norm{\tilde{r}^{(k+1)}}_2+2\left(K\,\norm{\mathrm{D}f(x_*)}_2+\norm{\mathrm{D}f(x_*)\circ \mathrm{D}g(x_*)}_2\right)\norm{\tilde{x}^{(k+1)}-P_\Sigma(\tilde{x}^{(m+1)})}_2\\
&\leq K\,\mathrm{dist}_2(0,\mathcal{A}^{(k)}_{m_k})+\kappa(1+t^{-2m_k})\underset{i\in\{0,\dots,m_k\}}{\max}\,{\norm{r^{(k-m_k+i)}}_2}^2,
\end{align*}
where
$\kappa=\frac{1}{\sigma^2}\left(\frac{KL}{2} + KM\,\norm{\mathrm{D}f(x_*)}_2+M\norm{\mathrm{D}f(x_*)\circ \mathrm{D}g(x_*)}_2\right)\underset{m\in\{1,\dots,p\}}\max C'_m$.
 
This concludes the proof of the Lemma for version P.

\item \textbf{Version A.} In this version, the set $\Sigma$ coincides with $\mathbb{R}^n$ and one takes $x^{(k+1)}=\sum_{i=0}^{m_k}\tilde{c}_i g(x^{(k-m_k+i)})$. A bound on the norm of $x^{(k+1)}-g(\tilde{x}^{(k+1)})$ must be obtained, justifying the assumption that $g$ is of class $\mathscr{C}^2$.

We proceed as for the estimate obtained in the first step, the only difference being the use of \eqref{eq:g_C2_estimate} instead of \eqref{eq:f_C2_estimate}, ending up getting something very similar, namely
\begin{align*}
\norm{x^{(k+1)}-g(\tilde{x}^{(k+1)})}_2 &\leq\frac{L'}{2}(m_k+1)\norm{\tilde{c}}_\infty((m_k+1)\norm{\tilde{c}}_\infty+1)\underset{i\in\{0,\dots,m_k\}}{\max}\,{\norm{x^{(k-m_k+i)}-x_*}_2}^2\\
&\leq\frac{L'}{2\sigma^2}C'_{m_k}(1+t^{-2m_k})\underset{i\in\{0,\dots,m_k\}}{\max}\,{\norm{r^{(k-m_k+i)}}_2}^2.
\end{align*}
Then, using Assumption \ref{assumption 1} and combining the above estimate with inequalities \eqref{mean value inequalities} and \eqref{common estimate} yields
\begin{align*}
\norm{r^{(k+1)}}_2&\leq\norm{f(g(\tilde{x}^{(k+1)}))}_2+2\,\norm{\mathrm{D}f(x_*)}_2\norm{x^{(k+1)}-g(\tilde{x}^{(k+1)})}_2\\
&\leq K\,\norm{\tilde{r}^{(k+1)}}_2+2\,\norm{\mathrm{D}f(x_*)}_2\norm{{x}^{(k+1)}-g(\tilde{x}^{(k+1)})}_2\\
&\leq K\,\mathrm{dist}_2(0,\mathcal{A}^{(k)}_{m_k})+\kappa(1+t^{-2m_k})\underset{i\in\{0,\dots,m_k\}}{\max}\,{\norm{r^{(k-m_k+i)}}_2}^2,
\end{align*}
where
$\kappa=\frac{1}{\sigma^2}\left(\frac{KL}{2} + L'\,\norm{\mathrm{D}f(x_*)}_2\right)\underset{m\in\{1,\dots,p\}}\max C'_m$.

The bound being obtained for version A, the proof is ended.
\end{itemize}
\end{proof}

\subsubsection{Proof of Theorem~\ref{thm:linear-convergence-DIIS-with-restarts}}
For any natural integer $k$, we consider the following properties
\begin{equation}\label{eq:restartRecurA}
\tag{$a_{k}$}
\text{For any $i$ in }\{0,\dots,m_k\},\ x^{(k-m_k+i)}\text{ exists and lies in }U\cap\Sigma\setminus\{x_*\},\text{ and }\norm{r^{(k-m_{k})}}_{2}\leq R_\mu\tau^{2p}.
\end{equation}
\begin{equation}\label{eq:restartRecurB}
\tag{$b_{k}$}
\text{If }m_{k}\geq 1,\text{ then, }\forall i\in\{1,\dots,m_{k}\},\ \norm{r^{(k-m_{k}+i)}}_{2}\leq\mu^i\norm{r^{(k-m_{k})}}_{2}.
\end{equation}
\begin{equation}\label{eq:restartRecurC}
\tag{$c_{k}$}
\text{If }m_{k}\geq 1,\text{ then, }\forall\ell\in\{1,\dots,m_{k}\},\ d_{\ell}^{(k)}\geq\tau(1-\mu)\norm{r^{(k-m_k)}}_{2}.
\end{equation}
Let us denote by $(\mathcal{P}_{k})$ the set of properties \eqref{eq:restartRecurA}, \eqref{eq:restartRecurB}, and \eqref{eq:restartRecurC} at rank $k$. Obviously, $(\mathcal{P}_{0})$ holds, since $m_0=0$ and $\norm{r^{(0)}}_2\leq R_\mu\tau^{2p}$. We are going to prove by induction that $(\mathcal{P}_{k})$ is true as long as $r^{(k)}$ is nonzero, if $R_\mu$ is chosen small enough. Concurrently, bounds \eqref{bound on m_k}, \eqref{bound on the extrapolation coefficients} and \eqref{r-linear convergence with respect to restarts} will be proved as consequences of $(\mathcal{P}_{k})$.

\smallskip

Let us then assume that $(\mathcal{P}_{k})$ holds for some natural integer $k$. In the case $m_k=0$ the proof of $(\mathcal{P}_{k+1})$ is very easy, so we only give detailed arguments when $m_k\geq 1$. The error vectors $r^{(k-m_k)},\dots,r^{(k)}$ then satisfy condition \eqref{lower bound on distance} in Lemma \ref{lem:induction step} for $t=\tau(1-\mu)$. They are affinely independent and bound \eqref{bound on m_k} is satisfied, as well as \eqref{bound on the extrapolation coefficients}, which is simply bound \eqref{bound on the extrapolation coefficients, general setting} in Lemma \ref{lem:induction step}.

\noindent
Moreover, Assumptions \ref{assumption 1} and \ref{assumption 2} together with properties \eqref{eq:restartRecurA} and \eqref{eq:restartRecurB} imply the bounds $\norm{x^{(k-m_{k}+i)}-x_*}_{2}\leq \sigma^{-1} R_\mu$ and $\norm{g(x^{(k-m_{k}+i)})-x_*}_{2} \leq \sigma^{-1} R_\mu K$ for any integer $i$ in $\{0,\dots,m_{k}\}$. Combining them with bound \eqref{bound on the extrapolation coefficients}, we see that, for $R_{\mu}$ small enough, $\tilde x^{(k+1)}$ belongs to $U$ as well as $g(\tilde{x}^{(k+1)})$ and $\sum_{i=0}^{m_k}\tilde{c}_i\,g(x^{(k-m_k+i)})$, so that $x^{(k+1)}$ and $r^{(k+1)}$ are well defined (with $x^{(k+1)}$ belonging to $U\cap\Sigma$), and estimates \eqref{lem:inductionRlast} and \eqref{lem:induction1KRlast} of Lemma~\ref{lem:induction step} hold.\medskip

\noindent
Now, using property \eqref{eq:restartRecurB}, bound \eqref{lem:inductionRlast} for $t=\tau(1-\mu)$ gives
\begin{align*}
\norm{r^{(k+1)}}_2&\leq K\,\mathrm{dist}_2(0,\mathcal{A}^{(k)}_{m_k})+\kappa\left(1+\frac{1}{(\tau(1-\mu))^{2m_k}}\right){\norm{r^{(k-m_k)}}_2}^2\\
&\leq \left(K\mu^{m_k}+\kappa\left(1+\frac{1}{(\tau(1-\mu))^{2p}}\right)\norm{r^{(k-m_k)}}_2\right)\norm{r^{(k-m_k)}}_2,
\end{align*}
and, to establish \eqref{r-linear convergence with respect to restarts}, we thus need the inequality
\[
\frac{K}{\mu}+\kappa\left(1+\frac{1}{(\tau(1-\mu))^{2p}}\right)\frac{R_\mu\tau^{2p}}{\mu^{m_k+1}}\leq 1.
\]
For this to be true, it suffices to choose the constant $R_\mu$ so that
\[
\frac{K}{\mu}+\kappa\left(1+\frac{1}{(1-\mu)^{2p}}\right)\frac{R_\mu}{\mu^{p+1}}\leq 1.
\]
Then, assuming that $r^{(k+1)}$ is nonzero, we distinguish two cases.
\begin{itemize}
\item If $m_{k+1}=0$ (meaning there is a restart at step $k$), then $k+1-m_{k+1}=k+1$ and, since we have previously shown that \eqref{r-linear convergence with respect to restarts} holds, it follows from property \eqref{eq:restartRecurA} that
\[
\norm{r^{(k+1-m_{k+1})}}_2\leq\mu^{m_k+1}\norm{r^{(k-m_k)}}_2\leq \mu^{m_k+1}R_\mu\tau^{2p}\leq R_\mu\tau^{2p},
\]
so the set of properties $(\mathcal{P}_{k+1})$ holds.\medskip

\item Otherwise, one has $k+1-m_{k+1}=k-m_k$. It then follows from property \eqref{eq:restartRecurA} that $\norm{r^{(k+1-m_{k+1})}}_2=\norm{r^{(k-m_k)}}_2\leq R_\mu\tau^{2p}$, so that ($a_{k+1}$) holds. Moreover, since there is no restart at step $k$, the condition
\[
\tau\norm{r^{(k+1)}-r^{(k-m_k)}}_2\leq\norm{(\id-\Pi_k)(r^{(k+1)}-r^{(k-m_k)})}_2=d^{(k)}_{m_k+1}
\]
is verified. We then have
\begin{align*}
d^{(k)}_{m_k+1}&\geq\tau\left(\norm{r^{(k-m_k)}}_2-\norm{r^{(k+1)}}_2\right)\\
&\geq\tau\left(\norm{r^{(k-m_k)}}_2-\mu^{m_k+1}\norm{r^{(k-m_k)}}_2\right)\\
&\geq\tau(1-\mu)\norm{r^{(k-m_k)}}_2.
\end{align*}
Using the same notations as in Lemma \ref{lem:induction step}, we also have $d^{(k+1)}_{i}=d^{(k)}_{i}$ for any integer $i$ in $\{1,\dots,m_k+1\}$. As a consequence, property $(c_{k+1})$ follows from property \eqref{eq:restartRecurC} and the above inequality. Finally, property \eqref{eq:restartRecurB} is equivalent to
\[
\forall j\in\{0,\dots,m_{k+1}-1\},\ \norm{r^{(k+1-m_{k+1}+j)}}_2\leq\mu^j\norm{r^{(k+1-m_{k+1})}}_2,
\]
and bound \eqref{r-linear convergence with respect to restarts} is equivalent to
\[
\norm{r^{(k+1-m_{k+1}+m_{k+1})}}_2\leq\mu^{m_{k+1}}\norm{r^{(k+1-m_{k+1})}}_2,
\]
so that property $(b_{k+1})$ holds, and so does ($\mathcal{P}_{k+1}$).
\end{itemize}
The set of properties ($\mathcal{P}_{k}$) is thus satisfied as long as $r^{(k)}$ is nonzero, and \eqref{bound on m_k}, \eqref{bound on the extrapolation coefficients}, and \eqref{r-linear convergence with respect to restarts} follow from it. The r-linear convergence property \eqref{r-linear convergence global} is next proved by complete induction.\\
For $k=0$, it is obvious. Assuming that for a given natural integer $k$, one has
\[
\forall i\in\{0,\dots,k\},\ \norm{r^{(i)}}_2\leq\mu^i\,\norm{r^{(0)}}_2
\]
and, using estimate \eqref{r-linear convergence with respect to restarts}, one finds that
\[
\norm{r^{(k+1)}}_2\leq\mu^{m_k+1}\,\norm{r^{(k-m_k)}}_2\leq\mu^{m_k+1}\mu^{k-m_k}\,\norm{r^{(0)}}_2=\mu^{k+1}\norm{r^{(0)}}_2,
\]
ending the proof of \eqref{r-linear convergence global}.

\noindent
Finally, if a restart occurs at step $k+1$, one has, employing the same notations as in Lemma \ref{lem:induction step},
\[
d^{(k)}_{m_k+1}<\tau\norm{r^{(k+1)}-r^{(k-m_k)}}_2.
\]
Inequality \eqref{estimate at restart} then follows from estimate \eqref{lem:induction1KRlast} in Lemma \ref{lem:induction step} and the triangle inequality, by setting
\[
\Gamma=\kappa\left(1+(1-\mu)^{-2p}\right).
\]

\subsubsection{Proof of Theorem~\ref{thm:linear-convergence-adaptive-depth}}
For any natural integer $k$, we consider the following properties.
\begin{multline}\label{eq:SlidingRecurA}\tag{$a_{k}$}
\text{For any $i$ in }\{0,\dots,k\},\ x^{(i)}\text{ exists and lies in }U\cap \Sigma\setminus\{x_*\}\\\text{ and, }\forall i\in\{0,\dots,k-1\},\ \norm{r^{(i+1)}}_{2}\leq\mu\norm{r^{(i)}}_{2}.
\end{multline}
\begin{equation}\label{eq:SlidingRecurB}\tag{$b_{k}$}
\text{If }m_{k}\geq 1,\text{ then, }\forall j\in\{1,\dots,m_k\},\ \forall i\in\{0,\dots,j-1\},\ \delta\norm{r^{(k-m_{k}+i)}}_{2}\leq \norm{r^{(k-m_{k}+j)}}_{2}.
\end{equation}
\begin{equation}\label{eq:SlidingRecurC}\tag{$c_{k}$}
\text{If }m_{k}\geq 1,\text{ then, }\forall\ell\in\{1,\dots,m_k\},\ d_{\ell}^{(k)}\geq\frac{1-\mu}{\mu}\norm{r^{(k-m_k+\ell)}}_{2}.
\end{equation}

Let us denote by $(\mathcal{P}_{k})$ the set of properties \eqref{eq:SlidingRecurA}, \eqref{eq:SlidingRecurB}, and \eqref{eq:SlidingRecurC} at rank $k$. Obviously, $(\mathcal{P}_{0})$ holds. Under the assumptions of Theorem \ref{thm:linear-convergence-adaptive-depth} and for $c_\mu$ sufficiently small, we will prove by induction that $(\mathcal{P}_{k})$ holds for all $k$ as long as $r^{(k)}$ is nonzero, and that this will imply the statements of the theorem. 

\noindent
Let us then assume that $(\mathcal{P}_{k})$ holds for some natural integer $k$.  In the case $m_k=0$ the proof of $(\mathcal{P}_{k+1})$ is very easy, so we only give detailed arguments when $m_k\geq 1$. From property \eqref{eq:SlidingRecurA}, one has
\begin{equation}\label{estim r}
\forall\ell\in\{0,\dots,m_k\},\ \max_{i\in\{\ell,\dots,m_k\}}\norm{r^{(k-m_k+i)}}_2=\norm{r^{(k-m_k+\ell)}}_2\leq c_\mu\,\delta^2,
\end{equation}
and property \eqref{eq:SlidingRecurC} implies that assumption~\eqref{lower bound on distance} in Lemma~\ref{lem:induction step} is satisfied with $t=\frac{1-\mu}{\mu}$. As a consequence, both \eqref{thm:sliding:enum:mk} and estimate~\eqref{bound on the extrapolation coefficients, general setting} hold, the latter taking the form
\[
\norm{c^{(k)}}_{\infty}\leq C_{m_{k}}\left(1+\left(\dfrac{\mu}{1-\mu}\right)^{m_{k}}\right),
\]
which is exactly estimate \eqref{thm:sliding:enum:ck}.

Next, using Assumptions \ref{assumption 1} and \ref{assumption 2}, the bounds
\[
\max_{i\in\{0,\dots,m_{k}\}}{\norm{x^{(k-m_{k}+i)}-x_*}_{2}}\leq \sigma^{-1} c_\mu K^2\text{ and }\max_{i\in\{0,\dots,m_{k}\}}{\norm{g(x^{(k-m_{k}+i)})-x_*}_{2}}\leq \sigma^{-1} c_\mu K^3
\]
follow from \eqref{estim r}. Combining them with \eqref{thm:sliding:enum:ck}, we see that, for $c_{\mu}$ small enough, $\tilde x^{(k+1)}$ belongs to $U$, as well as $g(\tilde{x}^{(k+1)})$ and $\sum_{i=0}^{m_k}\tilde{c}_i\,g(x^{(k-m_k+i)})$, so that $x^{(k+1)}$ and $r^{(k+1)}$ are both well defined, with $x^{(k+1)}$ belonging to $U\cap \Sigma$, and estimates \eqref{lem:inductionRlast} and \eqref{lem:induction1KRlast} of Lemma~\ref{lem:induction step} hold.

We now assume that $r^{(k+1)}$ is nonzero and impose the following additional smallness constraint on the constant $c_\mu$,
\begin{equation}\label{mu}
\kappa\left(1+\max\left\{ \left( \frac{\mu}{1-\mu} \right)^{2}, \left( \frac{\mu}{1-\mu} \right)^{2p} \right\}\right)c_{\mu} < 1-\frac{K}{\mu}\,.
\end{equation}
From property \eqref{eq:SlidingRecurB}, one has that $\norm{r^{(k-m_{k})}}_{2}\leq \delta^{-1}{\norm{r^{(k)}}_{2}}$, so that, using \eqref{estim r} (for $\ell=0$), we conclude that
\[
\max_{i\in\{0,\dots,m_{k}\}}{\norm{r^{(k-m_{k}+i)}}_{2}}^2={\norm{r^{(k-m_{k})}}_{2}}^2\leq c_\mu \delta{\norm{r^{(k)}}_{2}}\leq c_\mu K{\norm{r^{(k)}}_{2}}\,.
\]
Hence, using \eqref{lem:inductionRlast} and \eqref{mu}, we get
\begin{align*}
\norm{r^{(k+1)}}_{2}&\leq   K\,\mathrm{dist}_{2}(0,\mathcal{A}_{m_k}^{(k)})+\kappa(1+\max\{ t^{-2}, t^{-2p} \})\max_{i\in\{0,\dots,m_{k}\}}{\norm{r^{(k-m_{k}+i)}}_{2}}^{2}\\
&\leq K\left(1+\kappa(1+\max\{ t^{-2}, t^{-2p} \})c_{\mu}\right) \norm{r^{(k)}}_{2}\\
&\leq \mu  \norm{r^{(k)}}_{2}\,,
\end{align*}
which, together with property \eqref{eq:SlidingRecurA}, establishes that property $(a_{k+1})$ holds.

If $m_{k+1}=0$, there is nothing else to check, and $(\mathcal{P}_{k+1})$ holds. Otherwise, one has, by design of the algorithm,
\[
\forall i\in\{0,\dots,m_{k+1}\},\ \norm{r^{(k+1)}}_{2}\geq\delta\norm{r^{(k+1-m_{k+1}+i)}}.
\]
Using property \eqref{eq:SlidingRecurB}, property~$(b_{k+1})$ ensues. 
Now, from \eqref{eq:SlidingRecurB} and $(b_{k+1})$, one has
\[
\norm{r^{(k-m_{k})}}_{2}\leq\frac{1}{\delta}\norm{r^{(k+1-m_{k+1})}}_{2}\leq\frac{1}{\delta^{2}}\norm{r^{(k+1)}}_{2}.
\]
Combining this with \eqref{estim r} (for $\ell=0$), one finds
\[
\max_{i\in\{0,\dots,m_{k}\}}{\norm{r^{(k-m_{k}+i)}}_{2}}^2={\norm{r^{(k-m_{k})}}_{2}}^2\leq c_\mu{\norm{r^{(k+1)}}_{2}}.
\]
As a consequence, estimate~\eqref{lem:induction1KRlast} of Lemma~\ref{lem:induction step} gives 
\begin{equation}\label{ineq}
\norm{r^{(k+1)}}_{2}\leq \frac{K}{1-K}\,d^{(k)}_{m_k+1}+\frac{\kappa}{1-K}\left(1+\max\left\{ \left(\frac{\mu}{1-\mu} \right)^{2}, \left( \frac{\mu}{1-\mu} \right)^{2p} \right\}\right)c_{\mu}\,\norm{r^{(k+1)}}_{2}\,.
\end{equation}
But the constraint \eqref{mu} means exactly that
\[
\frac{K}{1-K}\frac{1-\mu}{\mu}+\frac{\kappa}{1-K}\left(1+\max\left\{ \left(\frac{\mu}{1-\mu} \right)^{2}, \left( \frac{\mu}{1-\mu} \right)^{2p} \right\}\right)c_{\mu}<1\,,
\]
so that $\norm{r^{(k+1)}}_{2}$ is non-positive if $d^{(k)}_{m_k+1}<\frac{1-\mu}{\mu}\norm{r^{(k+1)}}_{2}$, which is absurd. We have thus proved by contradiction that 
\[
d^{(k)}_{m_k+1}\geq\frac{1-\mu}{\mu}\norm{r^{(k+1)}}_{2}.
\]
Together with our assumption \eqref{eq:SlidingRecurC}, this implies the new property 
\begin{equation}\label{eq:newak}\tag{$\hat c_{k}$}
\forall\ell\in\{1,\dots,m_k+1\},\ d^{(k)}_{\ell}\geq\frac{1-\mu}{\mu}\norm{r^{(k-m_k+\ell)}}_{2}.
\end{equation}
Now, it is easily seen that $\mathcal{A}_{\ell-1}^{(k+1)}\subset\mathcal{A}^{(k)}_{\ell+m_k-m_{k+1}}$ for any integer $\ell$ in $\{1,\dots,m_{k+1}\}$, so that one has $d_{\ell}^{(k+1)}\geq d_{\ell+1+m_k-m_{k+1}}^{(k)}$. Property~\eqref{eq:newak} implying that
\[
\forall\ell\in\{1,\dots,m_{k+1}\},\ d_{\ell+1+m_k-m_{k+1}}^{(k)}\geq \frac{1-\mu}{\mu}\norm{r^{(k+1-m_{k+1}+\ell)}}_{2},
\]
property $(c_{k+1})$ holds, and so does $(\mathcal{P}_{k+1})$.

\noindent
We have thus proved that, for $c_\mu$ small enough, $(\mathcal{P}_{k})$ holds for any natural integer $k$ such that $r^{(k)}$ is nonzero, and so do statements \eqref{thm:sliding:enum:mk}, \eqref{thm:sliding:enum:ck}, and \eqref{thm:sliding:enum:rkmu}. It remains to prove statement~\eqref{thm:sliding:enum:rkdelta}.

Let $k$ be a natural integer such that $k-p\geq 1$ and $\norm{r^{(k)}}_2$ is nonzero. Since $m_{k}\leq p$, one has $k-p-1\leq k-m_k-1$, so that $\norm{r^{(k)}}_{2}\leq\delta\norm{r^{(k-m_{k}-1)}}_{2}\leq\delta\norm{r^{(k-p-1)}}_{2}$. This ends the proof.

\subsubsection{Proof of Theorem \ref{thm:superlinear restart}}
We set $\varepsilon(T_0,\zeta,\mu)= \min \left\{T_0^{-1/\zeta},(R_\mu T_0^{2p})^{\frac{1}{1-2p\zeta}},\left(\frac{1-K}{2KT_0}\right)^{1/\zeta}\right\}$ and $k_0=0$. Considering a natural integer $k_i$ such that $x^{(k_i)}$ is well-defined, $m_{k_i}=0$ and $\norm{r^{(k_i)}}_2\leq \mu^{k_i}\norm{r^{(0)}}_2$, we will prove the existence of a strictly larger natural integer $k_{i+1}$ such that $x^{(k_{i+1})}$ is well-defined, $m_{k_{i+1}}=0$, $m_{k}=k-k_i$ for any integer $k$ in $\{k_i,\dots,k_{i+1}-1\}$, and $\norm{r^{(k)}}_2\leq \mu^{k-k_i} \norm{r^{(k_i)}}_2$ for any integer $k$ in $\{k_i,\dots,k_{i+1}\}$.

We may suppose that $r^{(k_i)}$ is nonzero (otherwise, one would have $x^{(k_i)}=x^{(k_\text{stop})}$ and thus $k_{i+1}=k_i+1$). In such a case, by design of Algorithm \ref{alg:cdiis-restart-variant}', the parameter $\tau$ in the boolean test at step $k$ remains equal to $\tau^{(k_i)}=T_0\norm{r^{(k_i)}}_2^{\zeta}$ for any integer $k$ greater than or equal to $k_i$ such that $x^{(k)}$ is well-defined and $m_{k}=k-k_i$. This means that, in this range of steps, the modified algorithm is equivalent to the original one, with fixed $\tau$ equal to $\tau^{(k_i)}$ and initial point equal to $x^{(k_i)}$. Since $0<\norm{r^{(k_i)}}_2\leq \norm{r^{(k_0)}}_2< T_0^{-1/\zeta}$, one has $0<\tau^{(k_i)}<1$. Moreover, by definition of $\tau^{(k_i)}$ and since $\norm{r^{(k_0)}}_2 \leq (R_\mu T_0^{2p})^{\frac{1}{1-2p\zeta}}$, one has $\norm{r^{(k_i)}}_2 \leq R_\mu (T_0\norm{r^{(k_i)}}_2^{\zeta})^{2p} = R_\mu (\tau^{(k_i)})^{2p}$, and Theorem \ref{thm:linear-convergence-DIIS-with-restarts} applies. As a consequence, the integer $k_{i+1}$ exists and is such that $k_{i+1}\leq k_i+p$ and $\norm{r^{(k_{i+1})}}_2\leq\mu^{k_{i+1}-k_i}\norm{r^{(k_i)}}_2\leq\mu^{k_{i+1}}\norm{r^{(0)}}_2$.

\medskip

We have thus shown that the sequence $(x^{(k)})_{k\in\mathbb{N}}$ is well-defined and converges at least r-linearly to $x^*$, and that the family of steps at which the method restarts forms an infinite sequence $(k_i)_{i\in\mathbb{N}}$ such that the difference between two consecutive terms is at most equal to $p$. This allows to say more about the convergence of the sequence of approximations. Indeed, if $\norm{r^{(k_{i+1})}}_2>0$, then estimate \eqref{estimate at restart} holds with $k+1=k_{i+1}$, $m_k=k_{i+1}-k_i-1$ and $\tau=\tau^{(k_i)}$. Moreover, since it is assumed that $\norm{r^{(0)}}_2\leq\left(\frac{1-K}{2KT_0}\right)^{1/\zeta}$, it follows that $1-K(1+\tau^{(k_i)})\geq\frac{1-K}{2}$ for any natural integer $i$. Hence, using the formula for $\tau^{(k_i)}$, one gets from \eqref{estimate at restart} the estimate
\[
\frac{1-K}{2}\norm{r^{(k_{i+1})}}_2\leq \left( K\tau^{(k_i)}+\Gamma\, T_0^{-\frac{1}{\zeta}}(\tau^{(k_i)})^{\frac{1}{\zeta}-2p}\right) \norm{r^{(k_i)}}_2\leq \left( K+\Gamma\,T_0^{-\frac{1}{\zeta}}\right)(\tau^{(k_i)})^{\min\{1,\frac{1}{\zeta}-2p\}} \norm{r^{(k_i)}}_2,
\]
from which the inequality
\[
\norm{r^{(k_{i+1})}}_2\leq B\,{\norm{r^{(k_i)}}_2}^{\Theta}
\]
follows, with $B=\frac{2}{1-K}\left(K+\Gamma\,T_0^{-\frac{1}{\zeta}}\right)T_0^{\frac{\Theta-1}{\zeta}}$ and $\Theta=1+\min\{\zeta,1-2p\zeta\}$.

Next, let $i_0$ be an integer such that $B^{\frac{1}{\Theta-1}}\mu^{k_{i_0}}\norm{r^{(0)}}_2 <\frac{1}{2}$. Since $\norm{r^{(k_{i_0})}}_2\leq \mu^{k_{i_0}}\norm{r^{(0)}}_2$ and $\norm{r^{(k_{i+1})}}_2\leq B\,{\norm{r^{(k_i)}}_2}^{\Theta}$, one can prove by induction that, for any integer $i$ greater than or equal to $i_0$, $\norm{r^{(k_{i})}}_2\leq B^{-\frac{1}{\Theta-1}} 2^{-\Theta^{i-i_0}}$ and, using that $k_i-k_{i_0}\leq (i-i_0)\,p$, $\norm{r^{(k_{i})}}_2\leq B^{-\frac{1}{\Theta-1}} 2^{-\Theta^{(k_i-k_{i_0})/p}}$.

\medskip

Finally, for any integer $k$ greater than or equal to $k_{i_0}$, let the integer $i(k)$ be such that $k-m_k=k_{i(k)}$. This implies that $i(k)\geq i_0$ and $k_{i(k)}\geq k-p$, so that
\[
\norm{r^{(k)}}_2\leq\norm{r^{(k_{i(k)})}}_2\leq B^{-\frac{1}{\Theta-1}} 2^{-\Theta^{(k_{i(k)}-k_{i_0})/p}}\leq b\,\eta^{\theta^k},
\]
where $\eta=2^{-\theta^{-p-k_{i_0}}}$, $b=B^{-\frac{1}{\Theta-1}}$, and $\theta=\Theta^{1/p}=\left(1+\min\{\zeta,1-2p\zeta\}\right)^{1/p}$ is the desired r-order of superlinear convergence.

\begin{rmrk}\label{remark on growth of coefficients}
In Algorithm \ref{alg:cdiis-restart-variant}', the sequence $(\tau^{(k-m_k)})_{k\in\mathbb{N}}$ tends to $0$ and estimate \eqref{bound on the extrapolation coefficients} on the extrapolation coefficients is no longer uniform: we can only say that $\norm{c^{(k)}}_\infty = O\left({\norm{r^{(k-m_k)}}_2}^{-\zeta m_k}\right)$.
\end{rmrk}

\subsubsection{Proof of Theorem \ref{thm:superlinear adaptive}}
The proof follows an induction argument very similar to the one in the proof of Theorem \ref{thm:linear-convergence-adaptive-depth}, but with some novelties, since the value of the parameter $\delta$ is no longer fixed along the iteration in Algorithm \ref{alg:cdiis-adaptive-depth-variant}'.

The set of properties at rank $k$, denoted $(\mathcal{P}_{k})$, is almost unchanged, the only modification being the replacement of $\delta$ by $D_0\,{\norm{r^{(k-m_k+i)}}_2}^{\xi}$ in the second property:
\begin{equation}\label{eq:SlidingRecurB'}\tag{$b_{k}$}
\text{if }m_{k}\geq 1,\text{ then, }\forall j\in\{1,\dots,m_k\},\ \forall i\in\{0,\dots,j-1\},\ D_0\,{\norm{r^{(k-m_k+i)}}_2}^{1+\xi}\leq\norm{r^{(k-m_{k}+j)}}_{2}.
\end{equation}
Assume that $(\mathcal{P}_{k})$ holds for some natural integer $k$, with $m_k\geq 1$ (the case $m_k=0$ being immediate). From properties \eqref{eq:SlidingRecurA} and \eqref{eq:SlidingRecurC}, one infers that assumption~\eqref{lower bound on distance} in Lemma~\ref{lem:induction step} is satisfied with $t=\frac{1-\mu}{\mu}$, so that both \eqref{thm:sliding:enum:mk} and \eqref{thm:sliding:enum:ck} hold.

Moreover, the bounds
\[
\forall i\in\{0,\dots,m_{k}\},\ \norm{x^{(k-m_{k}+i)}-x_*}_{2}\leq\sigma^{-1} \varepsilon(D_0,\xi,\mu)\text{ and }\norm{g(x^{(k-m_{k}+i)})-x_*}_{2}\leq \sigma^{-1} K\varepsilon(D_0,\xi,\mu)
\]
follow from using property \eqref{eq:SlidingRecurA} and the use of Assumptions \ref{assumption 1} and \ref{assumption 2}. As a consequence, for $\varepsilon(D_0,\xi,\mu)$ small enough, the linear combination $\tilde x^{(k+1)}$ belongs to $U$, and so do $g(\tilde{x}^{(k+1)})$ and $\sum_{i=0}^{m_k}\tilde{c}_i\,g(x^{(k-m_k+i)})$, resulting in both $x^{(k+1)}$ and $r^{(k+1)}$ being well-defined (with $x^{(k+1)}$ belonging to $U\cap \Sigma$) and estimates \eqref{lem:inductionRlast} and \eqref{lem:induction1KRlast} of Lemma~\ref{lem:induction step} being valid.

Let us now assume that $r^{(k+1)}$ is nonzero. Property \eqref{eq:SlidingRecurB'} implies in particular that
${\norm{r^{(k-m_{k})}}_{2}}^{1+\xi}\leq {D_0}^{-1}{\norm{r^{(k)}}_{2}}$. This, together with property \eqref{eq:SlidingRecurA} and the fact that $\norm{r^{(k-m_{k})}}_{2}<\varepsilon(D_0,\xi,\mu)$ by assumption, gives
\[
\max_{i\in\{0,\dots,m_{k}\}}{\norm{r^{(k-m_{k}+i)}}_{2}}^2={\norm{r^{(k-m_{k})}}_{2}}^2\leq {D_0}^{-1}(\varepsilon(D_0,\xi,\mu))^{1-\xi}\norm{r^{(k)}}_{2}.
\]
Hence, using estimate \eqref{lem:inductionRlast}, we get
\begin{align*}
\norm{r^{(k+1)}}_{2}&\leq   K\,\mathrm{dist}_{2}(0,\mathcal{A}_{k}^{(k)})+\kappa(1+\max\{ t^{-2}, t^{-2p} \})\max_{i\in\{0,\dots,m_{k}\}}{\norm{r^{(k-m_{k}+i)}}_{2}}^{2}\\
&\leq \left(K+\kappa(1+\max\{ t^{-2}, t^{-2p} \}){D_0}^{-1}(\varepsilon(D_0,\xi,\mu))^{1-\xi}\right)\norm{r^{(k)}}_{2}.\label{pour superlin1}
\end{align*}
Since $K<\mu$, we may impose that
\begin{equation}\label{control}
\left(K+\kappa(1+\max\{ t^{-2}, t^{-2p} \}){D_0}^{-1}(\varepsilon(D_0,\xi,\mu))^{1-\xi}\right)\leq \mu,
\end{equation}
which can be achieved by taking $\varepsilon(D_0,\xi,\mu)$ small enough, since $\xi$ belongs to $(0,\sqrt{2}-1]$. Thus, there holds $\norm{r^{(k+1)}}_{2}\leq\mu\,\norm{r^{(k)}}_{2}$, which, combined with property \eqref{eq:SlidingRecurA}, establishes property $(a_{k+1})$.

If $m_{k+1}=0$, there is nothing else to check, so $(\mathcal{P}_{k+1})$ is true. If it is not the case, one has, by design of Algorithm \ref{alg:cdiis-adaptive-depth-variant}',
\[
\forall i\in\{0,\dots,m_{k+1}\},\ \norm{r^{(k+1)}}_{2}\geq D_0\,{\norm{r^{(k+1-m_{k+1}+i)}}_2}^{1+\xi}.
\]
Using property \eqref{eq:SlidingRecurB'}, property~$(b_{k+1})$ ensues, 
leading to
\[
\norm{r^{(k-m_k)}}_{2}\leq {D_0}^{-\frac{2+\xi}{(1+\xi)^2}}{\norm{r^{(k+1)}}_{2}}^{\frac{1}{(1+\xi)^2}}.
\]
As a consequence, estimate~\eqref{lem:induction1KRlast} of Lemma~\ref{lem:induction step} gives 
\begin{equation}\label{ineq'}
\norm{r^{(k+1)}}_{2}\leq \frac{K}{1-K}\,d^{(k)}_{m_k+1}+\frac{\kappa}{1-K}\left(1+\max\left\{ \left(\frac{\mu}{1-\mu} \right)^{2}, \left( \frac{\mu}{1-\mu} \right)^{2p} \right\}\right){D_0}^{-\frac{2(2+\xi)}{(1+\xi)^2}}{\norm{r^{(k+1)}}_{2}}^{\frac{2}{(1+\xi)^2}}.
\end{equation}
We now impose the condition
\[
\frac{K}{1-K}\frac{1-\mu}{\mu}+\frac{\kappa}{1-K}\left(1+\max\left\{ \left(\frac{\mu}{1-\mu} \right)^{2}, \left( \frac{\mu}{1-\mu} \right)^{2p} \right\}\right){D_0}^{-\frac{2(2+\xi)}{(1+\xi)^2}}{\norm{r^{(k+1)}}_{2}}^{\frac{2}{(1+\xi)^2}-1}<1\,,
\]
satisfied by taking $\varepsilon(D_0,\xi,\mu)$ small enough when $\xi$ belongs to $(0,\sqrt{2}-1)$, or $\Delta_\mu$ large enough when $\xi$ is equal to $\sqrt{2}-1$. Inequality \eqref{ineq'} then implies that $\norm{r^{(k+1)}}_{2}$ is non-positive if $d^{(k)}_{m_k+1}<\frac{1-\mu}{\mu}\norm{r^{(k+1)}}_{2}$, which is absurd. We have thus proved by contradiction that 
\[
d^{(k)}_{m_k+1}\geq\frac{1-\mu}{\mu}\norm{r^{(k+1)}}_{2}.
\]
The rest of the induction argument is exactly the same as in the proof of Theorem \ref{thm:linear-convergence-adaptive-depth}, with $(\mathcal{P}_{k})$ holding for any natural integer $k$ such that $r^{(k)}$ is nonzero, and implying estimates \eqref{thm:sliding:enum:mk}, \eqref{thm:sliding:enum:ck} and \eqref{thm:sliding:enum:rkmu}.

\medskip

It remains to establish that the convergence is superlinear. If $k\geq p+1$, then $k-m_k\geq k-p\geq 1$, so that, by property \eqref{eq:SlidingRecurA}, one has
\[
\norm{r^{(k)}}_{2}\leq D_0\,{\norm{r^{(k-m_{k}-1)}}_{2}}^{1+\xi}\leq D_0\,{\norm{r^{(k-p-1)}}_{2}}^{1+\xi}.
\]
It also follows from property \eqref{eq:SlidingRecurA} that there exists a natural integer $k_0$ such that
\[
\forall i\in\{0,\cdots,p\},\ {D_0}^{1/\xi}\norm{r^{(k_0+i)}}_2\leq\frac{1}{2}.
\]
As a consequence, for any integer $k$ greater or equal to $k_0+p+1$, writing $k-k_0=(p+1)q+i$ with $i$ in $\{0,\dots,p\}$ we get, reasoning by induction on $q$, that
\[
\norm{r^{(k)}}_2\leq{D_0}^{-\frac{1}{\xi}}2^{-(1+\xi)^q}\leq b\,\eta^{\theta^k},
\]
where $b=D_0^{-\frac{1}{\xi}}$, $\eta=2^{-(1+\xi)^{-\frac{k_0+p}{p+1}}}<1$ and $\theta=(1+\xi)^{\frac{1}{p+1}}$. The proof is complete. 

\section{Application to electronic ground state calculations}\label{electronic ground state calculation}
We now explain how the Anderson--Pulay acceleration analysed in the previous section can be applied to the computation of the electronic ground state of a molecular system through two of the most commonly used approximations of the quantum many-body problem in non-relativistic quantum chemistry.

\smallskip

Consider an isolated molecular system composed of $M$ atomic nuclei and $N$ electrons. Within the setting of the Born--Oppenheimer approximation, the motion of atomic nuclei and electrons can be separated and the nuclei are classical point-like particles with fixed positions. The state of the electrons is then entirely described by a wavefunction $\psi$ valued in $\mathbb{C}$, only depending on the time variable and on the respective positions and spins of the electrons. The ground state of the molecular system is determined by solving a minimisation problem, which one may try to attack ``directly'' by working in a finite basis. Unfortunately, due to the high dimension of the position space $\mathbb{R}^{3N}$, this approach is intractable for systems with more than a few electrons and one has to resort to other types of approximations.

\smallskip

On the one hand, the \textit{ab initio} \emph{Hartree--Fock method} \cite{Hartree:1928,Fock:1930} for the computation of the ground-state electronic wavefunction $\psi$ consists in restricting the functional space in the minimisation problem to the set of so-called \emph{Slater determinants}, which are antisymmetrised products of $N$ \emph{monoelectronic wavefunctions} (also called \emph{molecular (or atomic) orbitals}). While this restriction only provides with an upper bound of the exact energy, the main advantage of this approximation is that the problem to be solved remains variational. There is however a price to pay in a loss of the correlation between the positions of the electrons, as an electron will evolve independently of the way the others do in this model.

On the other hand, the \emph{density functional theory} (DFT) of Hohenberg and Kohn \cite{Hohenberg:1964} follows a completely different approach and aims at including the electronic correlation missing in the Hartree--Fock method. It is based on the fact that the ground-state properties of a many-electron system are uniquely determined by an \emph{electronic density}, which only depends on the three spatial coordinates. Indeed, following arguments by Hohenberg and Kohn, the energy functional defined in terms of the unknown wavefunction $\psi$ can be replaced by one for the unknown density function $\rho$. Since such a functional is not known explicitly for a system of $N$ interacting electrons, suitable approximations are needed to make the DFT a practical tool for computing electronic ground states. These rely on exact (or very accurate) evaluations of the density functional of a reference system ``close'' to the real one. In the semi-empirical approach of Kohn and Sham \cite{Kohn:1965}, the chosen reference system consists of $N$ \emph{non-}interacting electrons (the associated wavefunction being a single Slater determinant). The exchange-correlation part of the functional must then be approximated, using models like the local-density approximation (LDA) or the generalised gradient approximation (GGA).

\smallskip

While based on different physical principles, the discrete problems associated with these two approximations both take the form of a nonlinear generalised eigenvalue problem. Omitting the spin variables for the sake of simplicity, the monoelectronic wavefunctions are linearly expanded on a given finite basis set\footnote{Due to computational complexity considerations, the basis functions are typically Slater-type or Gaussian-type orbitals, and generally form a non-orthonormal set.} (this is the so-called \emph{LCAO approximation}, LCAO being the acronym to \emph{linear combination of atomic orbitals}), spanning a vector space of finite dimension~$d$.

In this setting, the spinless Hartree--Fock method leads to the so-called \emph{Roothaan--Hall equations} \cite{Roothaan:1951,Hall:1951},
\[
\begin{cases}
F^{\text{HF}}(D)C=SC\Lambda,\\
C^{*}SC=I_N,\\
D=CC^{*},
\end{cases}
\]
in which the matrix $F^{\text{HF}}(D)$ denotes, with a slight abuse of notation, the \emph{Fock matrix} associated with the rectangular matrix $C$ of orbital coefficients, the square matrix $D$ is the \emph{discrete density matrix}, the matrix $S$ is the so-called \emph{overlap} matrix, that is the Gram matrix associated with the basis set, the Hermitian matrix $\Lambda$, which is by convention chosen diagonal, stands for the Lagrange multipliers attached to an orthonormality constraints on the orbitals, and $I_{N}$ is the identity matrix of order $N$. A necessary condition for a matrix $D_*$ to be a ``solution'' of the above system in the submanifold of pure state density matrices of rank $N$,
\begin{equation}\label{PN}
\mathcal{P}_{N}=\left\{D\in\mathcal{M}_{d,d}(\mathbb{C}),\ D^{*}=D,\ DSD=D,\ \mathrm{tr}(SD)=N\right\},
\end{equation}
is that it commutes with the associated Fock matrix $F^{\text{HF}}(D_*)$ in the sense that
\[
[F^{\text{HF}}(D_*),D_*]=0,
\]
where the \emph{``commutator''} $[A,B]$ between two matrices $A$ and $B$ is defined by $[A,B]=ABS-SBA$.

The discretisation of other Hartree--Fock models or of the Kohn--Sham models can be dealt with in the same manner. For the spinless Kohn-Sham model, assuming differentiability for the approximated exchange-correlation functional used in practice, the Roothaan--Hall equations read
\[
\begin{cases}
F^{\text{KS}}(D)C=SC\Lambda,\\
C^{*}SC=I_N,\\
D=CC^{*}.
\end{cases}
\]

\smallskip The Roothaan--Hall equations are usually solved \emph{``self-consistently''}, that is using an iterative fixed-point procedure, the most simple and ``natural'' approach being the algorithm introduced by Roothaan \cite{Roothaan:1951}. Given the choice of an initial discrete density matrix $D^{(0)}$, it consists in generating a sequence of matrices $(D^{(k)})_{k\in\mathbb{N}}$ defined by
\[
\forall k\in\mathbb{N},\ \begin{cases}
F(D^{(k)})C^{(k+1)}=SC^{(k+1)}E^{(k+1)},\\
(C^{(k+1)})^{*}SC^{(k+1)}=I_N,\\
D^{(k+1)}=C^{(k+1)}(C^{(k+1)})^{*},
\end{cases}
\]
where $E^{(k+1)}$ is a diagonal matrix such that $(E^{(k+1)})_{ii}=\varepsilon_i^{(k+1)}$, $i=1,\cdots,N$, the scalars $\varepsilon_1^{(k+1)}\leq\varepsilon_2^{(k+1)}\leq\dots\leq\varepsilon_{N}^{(k+1)}$ being the $N$ smallest eigenvalues, counted with multiplicity, of the \emph{linear} generalised eigenproblem
\[
F(D^{(k)})V=\varepsilon\,SV,
\]
and the columns of the matrix $C^{(k+1)}$ are associated orthonormal (with respect to the scalar product induced by the matrix $S$) eigenvectors. The procedure of assembling $D^{(k+1)}$ by populating the molecular orbitals starting with those of lowest energy of the current Fock matrix is called the \emph{Aufbau principle}. 

Assuming the \emph{uniform well-posedness property} introduced in \cite{Cances:2000a}, the matrix $D^{(k+1)}$ is uniquely defined at each step of the procedure and can be characterised as the minimiser of a variational problem, that is
\[
D^{(k+1)}=\arg\inf\left\{\mathrm{tr}\left(F(D^{(k)})D\right),\ D\in\mathcal{P}_{N}\right\}=g(D^{(k)}),
\]
thus defining a fixed-point iteration process. In practice, a convergence criterion is used to end the iterations. For instance, one may compute the norm of the difference of two successive density matrices at each step and compare it to a prescribed tolerance. Another possibility is to use the norm of the commutator between the current density matrix and its associated Fock matrix.

\smallskip

Both of these choices may be employed to accelerate the convergence of the above self-consistent field (SCF) algorithm. Indeed, the first one corresponds to error vectors proposed in conjunction with the original form of the DIIS \cite{Pulay:1980}, that is $r^{(k)}=D^{(k)}-D^{(k-1)}$, while the second leads to those used in the CDIIS \cite{Pulay:1982}, that is $r^{(k)}=\left[F(D^{(k)}),D^{(k)}\right]$. The latter is widely used in quantum chemistry softwares for electronic structure calculations, notably because of its simplicity with respect to implementation and its usually rapid, but not assured, convergence. Let us describe it in more detail.

Starting from a guess density matrix $D^{(0)}$, one first sets $\widetilde{D}^{(0)}=D^{(0)}$. After $k$ steps of the method, given a set of $m_k+1$ previous density matrices $D^{(k-m_k)},\dots,D^{(k)}$, one assembles the \emph{pseudo-}density matrix\footnote{This denomination stems from the fact that such a construction does not enforce the idempotency property on $\widetilde{D}^{(k+1)}$.} $\widetilde{D}^{(k+1)}$ as
\[
\widetilde{D}^{(k+1)}=\sum_{i=0}^{m_k}c^{(k)}_iD^{(k-m_k+i)},
\]
where
\[
\vec{c}^{(k)}=\underset{\substack{(c_0,\dots,c_{m_k})\in\mathbb{R}^{m_k+1}\\\sum_{i=0}^{m_k}c_i=1}}{\arg\min}\,\left\|\sum_{i=0}^{m_k}c_i\left[F\left(D^{(k-m_k+i)}\right),D^{(k-m_k+i)}\right]\right\|_2,
\]
the matrix norm $\norm{\cdot}_2$ being the Frobenius norm. The next density matrix $D^{(k+1)}$, is then obtained by applying the Aufbau principle to $\widetilde{D}^{(k+1)}$, that is, by diagonalising the Fock matrix\footnote{For Hartree--Fock models, the pseudo-Fock matrix $\sum_{i=0}^{m_k}c^{(k)}_i\,F(D^{(k-m_k+i)})$ may as well be considered, since it holds that $F(\sum_{i=0}^{m_k}c^{(k)}_iD^{(k-m_k+i)})=\sum_{i=0}^{m_k}c^{(k)}_i\,F(D^{(k-m_k+i)})$ due to constraint \eqref{DIIS constraint} being satisfied by the coefficients $c^{(k)}_i$, $i=0,\dots,m_k$.}  associated with $\widetilde{D}^{(k+1)}$ and forming $D^{(k+1)}$ from the $N$ eigenvectors associated with its $N$ smallest eigenvalues. This process is then repeated until the numerical convergence criterion is satisfied. Modifications of the procedure have been proposed in order to make it more robust, either by replacing the constraint on the coefficients, like in the C$^2$-DIIS \cite{Sellers:1993}, or by obtaining them by minimisation of an associated energy, like in the EDIIS\footnote{The difference in this method is that the coefficients of the linear expansion are such that \[\vec{c}^{(k)}=\underset{\sum_{i=0}^{m_k}c_i=1,\,c_i\in[0,1]}{\arg\min}\,E\left(\sum_{i=0}^{m_k}c_iD^{(k-m_k+i)}\right),\] where $E$ is the energy functional of the model under consideration.} \cite{Kudin:2002}, the ADIIS\footnote{This other variant is similar to the EDIIS but uses the following augmented energy functional: \[E^{\text{A}}(D)=E(D^{(k)})+2\,\mathrm{tr}((D-D^{(k)})F(D^{(k)}))+\mathrm{tr}((D-D^{(k)})(F(D)-F(D^{(k)}))).\]} \cite{Hu:2010} or the LIST\footnote{This other variant is also similar to the EDIIS. It uses the so-called corrected Hohenberg--Kohn--Sham functional \cite{Zhang:2009}.} \cite{Wang:2011}.

\smallskip

Let us end this section by explaining how the CDIIS enters the theoretical framework set in the previous section for the Anderson--Pulay acceleration.

On the one hand, working with self-adjoint (real or complex) matrices of order $d$ ($d$ being the real or complex dimension of the vector space spanned by the basis functions used in the Galerkin approximation) with fixed trace, the integer $n$ is equal to $\frac{1}{2}d(d+1)-1$ (real case) or $d^2-1$ (complex case), the fixed-point function $g$ corresponds to the application of the Aufbau principle to such a matrix, and the submanifold $\Sigma$ is the set $\mathcal{P}_N$ of pure state density matrices of rank $N$, defined in \eqref{PN}. On the other hand, the error function $f$ is the commutator between a given density matrix and its associated Fock matrix, $f(D)=[F(D),D]$, so that the integer $p$ is equal to $\frac{1}{2}d(d-1)$ (real case) or $d^2-1$ (complex case).

\smallskip

Assumption \ref{assumption 1} then simply states that the Roothaan algorithm is locally convergent, which is a common assumption on the base fixed-point iteration method in the analysis of the DIIS or the Anderson acceleration (see \cite{Rohwedder:2011} or \cite{Toth:2015} for instance). Assumption \ref{assumption 2} amounts to a non-degeneracy assumption, which is equivalent to saying that the differential of the restriction of the error function to the submanifold is invertible at a solution $D_*$. In the present context, this function is already the gradient of the discretised Hartree--Fock (or Kohn--Sham) energy, and the assumption thus implies that the Hessian of this energy is non-degenerate at $D_*$.

\smallskip

Concerning regularity assumptions, the fact that the set $\mathcal{P}_N$ is a smooth submanifold  is a consequence of the constant rank theorem. For the functions $f$ and $g$ being of class $\mathscr{C}^2$, this is always true for the Hartree--Fock functional, which is smooth. In the case of the Kohn--Sham model, this issue is more delicate, since most of the exchange-correlation functionals used in practice present a lack of regularity at the origin, see \textit{e.g.} \cite{Anantharaman:2009}. However, it is reasonable to assume that the Kohn--Sham ground state has a non-vanishing density $\rho$, so that $f$ and $g$ are indeed regular in its neighbourhood.

\section{Numerical experiments}\label{numerical experiments}
In this section, we report on some numerical experiments with the intent of illustrating the performances of our proposed variants of the CDIIS in the context of the electronic ground state calculations considered in Section \ref{electronic ground state calculation}. All the computations presented were performed using tools provided by the PySCF package \cite{PYSCF}. The source code of our implementation of the CDIIS variants can be found in the following repository: \url{https://plmlab.math.cnrs.fr/mchupin/restarted-and-adaptive-cdiis/}.

\subsection{Implementation details}\label{implementation}
For each variant, the least-squares problem for the extrapolation coefficients is solved using the unconstrained form involving the differences of error vectors which are deemed the most convenient\footnote{Note that, for the fixed-depth CDIIS, the unconstrained formulation we employed uses successive differences.}, as seen in Algorithms \ref{alg:cdiis-restart-variant} and \ref{alg:cdiis-adaptive-depth-variant}. The solution is achieved via the QR factorisation of a matrix of \emph{tall and skinny} type (since the integer $p$ is in general very large compared to $m_{k}$). Such a factorisation can be efficiently updated from step to step by means of a dedicated routine from the SciPy linear algebra library (namely \texttt{scipy.linalg.qr\_update}), as the least-squares problem matrix is modified through the addition of a column (as in the variant with restarts) or the addition of a column and the possible removal of a set of columns (as in the adaptive-depth variant). A detailed cost analysis of the resulting factorisation algorithm is given in~\cite{Higham:2016}. 
An added benefit of using the QR decomposition is that it directly provides the matrix of the orthogonal projector $\Pi_{k}$ appearing in the restart condition~\eqref{restart condition, abstract problem}, so that testing for this condition at each step in Algorithm \ref{alg:cdiis-restart-variant} only entails a negligible cost.



Finally, we consider that numerical convergence is reached at iteration $k$ if the error vector norm $\norm{[F(D^{(k)}),D^{(k)}]}_2$ is below some prescribed tolerance.

\subsection{Test cases}
Some of the molecular systems used in our experiments are taken from benchmarks found in \cite{Garza:2012,Hu:2010} and are considered as representative of challenging convergence tests for self-consistent field algorithms. Results of some of the tests are presented here, like the cadmium(II)-imidazole complex (\ce{[Cd(Im)]^{2+}}) for instance, while others, like the acetaldehyde (\ce{C2H4O}), the acetic acid (\ce{C2H4O2}), or the silane (\ce{SiH4}), are available in the online repository given above. The glycine (\ce{C2H5NO2}) test case comes from an example given in the PySCF library and the geometries for galactonolactone (\ce{C6H10O6}) and dimethylnitramine (\ce{C2H6N2O2}) were found on the PubChem website (\url{https://pubchem.ncbi.nlm.nih.gov/}).

Both the restricted Hartree--Fock (RHF) model and the restricted Kohn--Sham (RKS) model are used in the experiments, the latter in conjunction with the B3LYP approximation of the exchange-correlation functional~\cite{Becke:1988,Lee:1988}.

\subsubsection{Global convergence behaviour}
Acceleration techniques like the CDIIS are locally convergent methods and may sometimes give poor results if a mediocre initial guess is used. In practice, in order to ensure and achieve convergence in a small number of steps, one usually employs a combination of a relaxed constraint algorithm (like the ODA \cite{Cances:2000}, the EDIIS \cite{Kudin:2002} or the ADIIS \cite{Hu:2010}), for its global convergence properties, and of the CDIIS, for its fast local convergence (see Subsection \ref{sec:local}). Nevertheless, our first experiment is meant to illustrate the convergence behaviour of the methods starting from the \emph{core Hamiltonian guess}, for which the orbital coefficients are simply obtained by diagonalising the matrix representing the kinetic energy and electron-nuclear potential energy parts of the Hamiltonian in the considered discrete basis.

Figure~\ref{fig:global} presents the convergence of the error vector norm for the (classical) fixed-depth, restarted and adaptive-depth variants of the CDIIS with different values of their respective parameters: the fixed maximum depth $m$, the restart parameter $\tau$ and the adaptive-depth parameter $\delta$. It is observed that the restarted and the adaptive-depth algorithms are efficient in most of the cases, but they may reach the convergence regime later than their fixed-depth counterpart (see Figures~\ref{fig:globalDimethyl} and~\ref{fig:globalGlycine}). However, when this regime is attained, one can observe that the rate of convergence of both the restarted and adaptive-depth CDIIS is generally better than that of the fixed-depth CDIIS.

\begin{figure}[ht]
	\centering
	\begin{subfigure}[b]{0.48\textwidth}
		\centering
		\includegraphics[width=\linewidth]{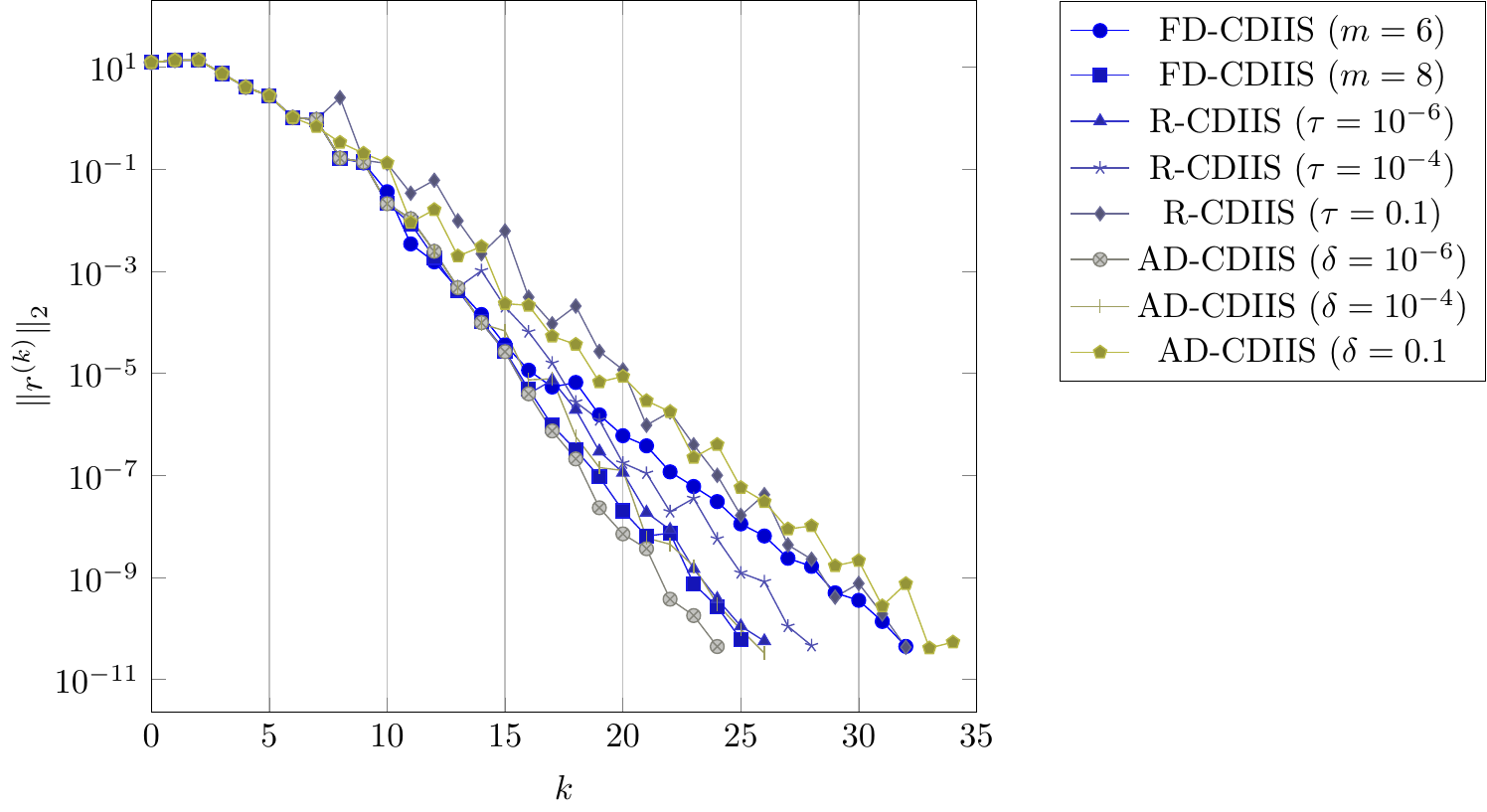}
		\caption{Cadmium-imidazole complex in the RKS/B3LYP model with basis 3-21G.}
	\end{subfigure}\quad
	\begin{subfigure}[b]{0.48\textwidth}
		\centering
		\includegraphics[width=\linewidth]{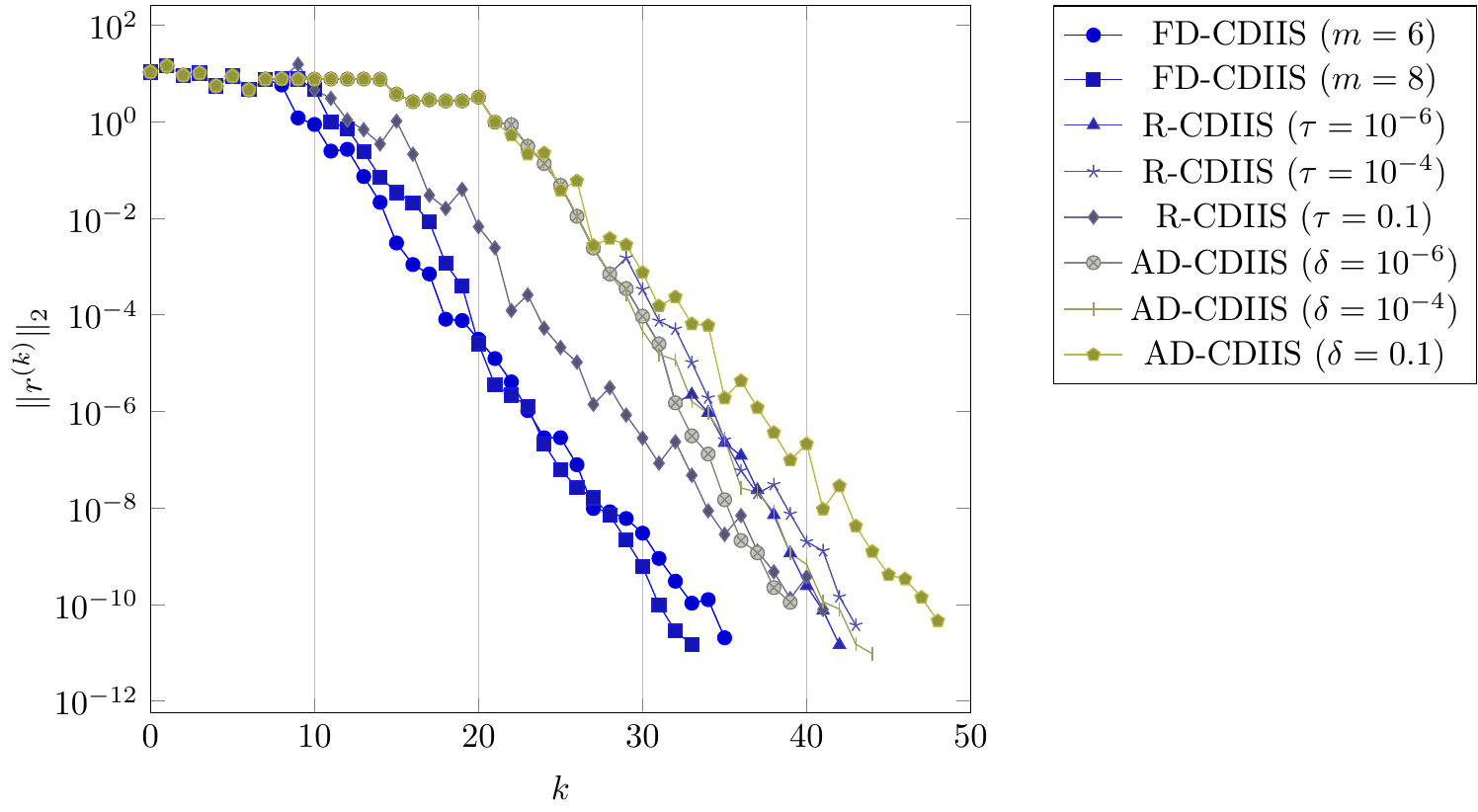}
		\caption{Glycine molecule in the RKS/B3LYP model with basis 6-31Gs.}\label{fig:globalGlycine}
	\end{subfigure}
	\begin{subfigure}[b]{0.48\textwidth}
		\centering
		\includegraphics[width=\linewidth]{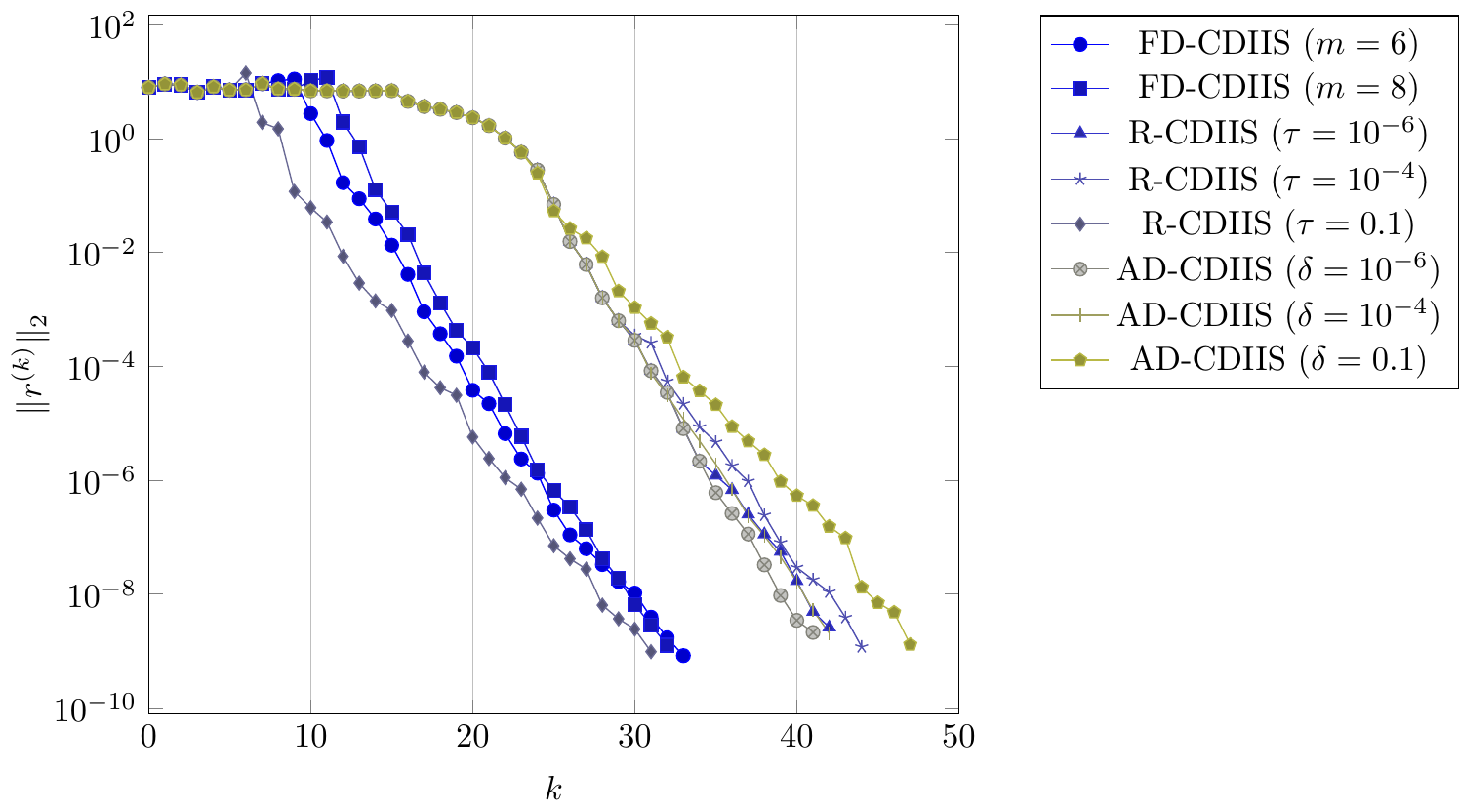} 
		\caption{Dimethylnitramine molecule in the RHF model with basis 6-31G.}\label{fig:globalDimethyl}
	\end{subfigure}\quad
	\begin{subfigure}[b]{0.48\textwidth}
		\centering
		\includegraphics[width=\linewidth]{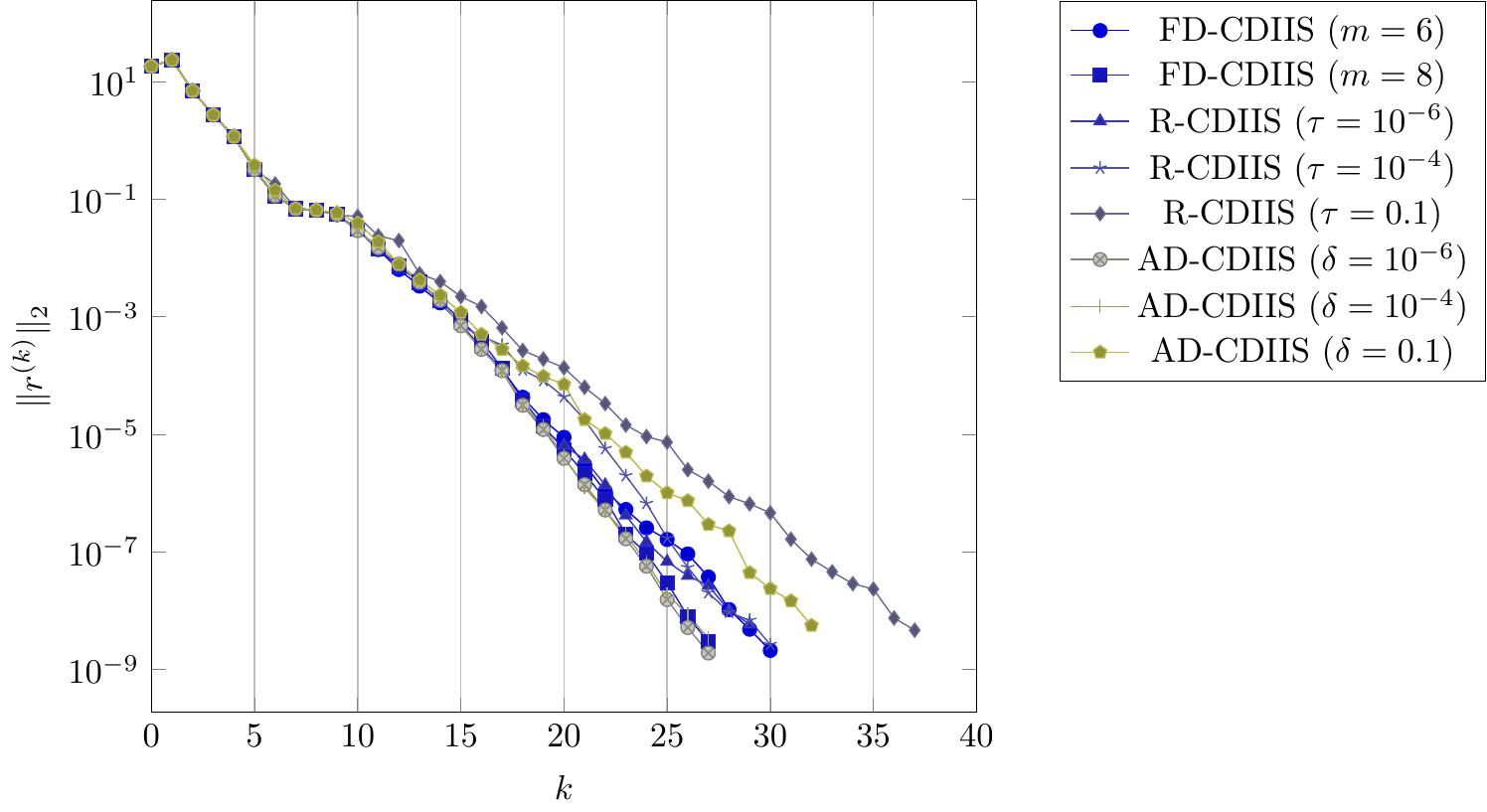}
		\caption{Galactonolactone molecule in the RHF model with basis 6-31G.}
	\end{subfigure}	
	\caption{Residual norm convergence for the fixed-depth, restarted and adaptive-depth CDIIS on different molecular systems using an initial guess obtained by diagonalising the core Hamiltonian matrix.}\label{fig:global}
\end{figure}

An effect of the use of a poor initial guess is the accumulation of many stored iterates in the early stages of the computation by the restarted and adaptive-depth variants, visible in Figure~\ref{fig:accumulation}. The number of stored iterates clearly decreases as soon as a convergence regime is reached. This behaviour can be observed in all of our numerical experiments.

\begin{figure}[ht]
	\centering
	\begin{subfigure}[b]{0.48\textwidth}
		\centering
		\includegraphics[width=\linewidth,page=1]{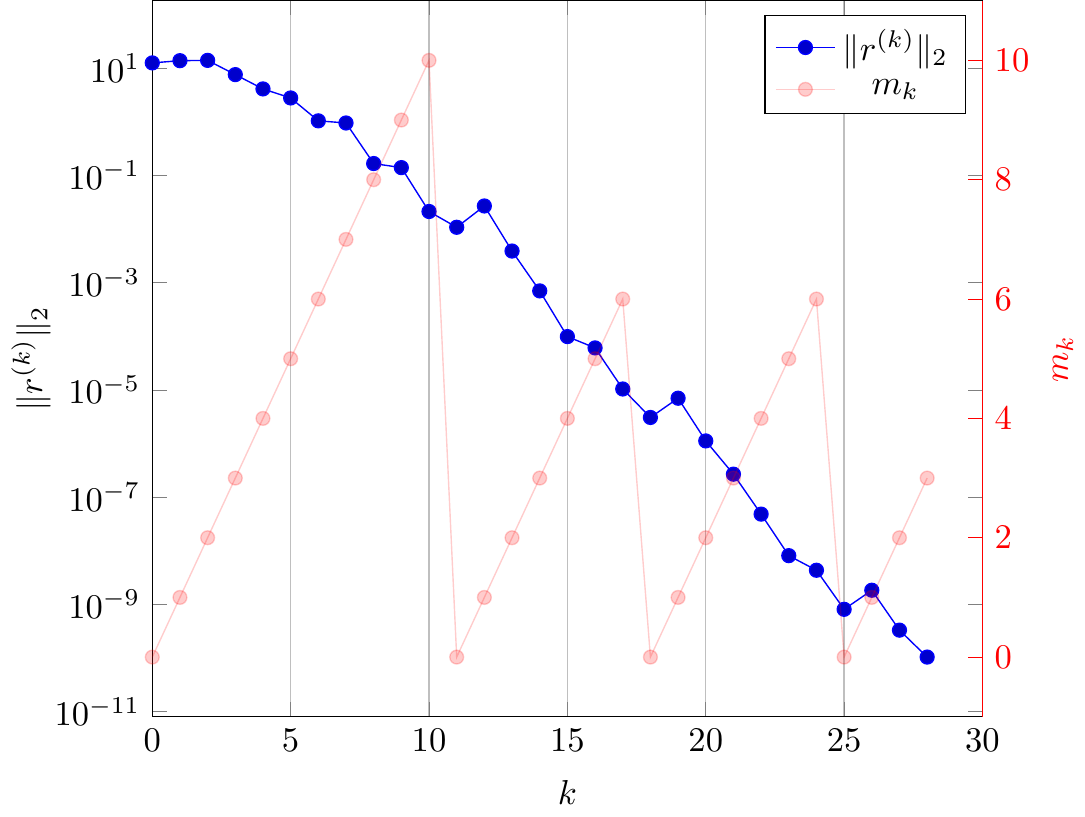}
		\caption{Restarted CDIIS with $\tau=10^{-4}$.}
	\end{subfigure}\quad
	\begin{subfigure}[b]{0.48\textwidth}
		\centering
		\includegraphics[width=\linewidth,page=2]{GlobalCdRestSlid.pdf}
		\caption{Adaptive-depth CDIIS with $\delta=10^{-4}$.}
	\end{subfigure}
	\caption{Residual norm convergence and corresponding depth value for the restarted and adaptive-depth CDIIS on the cadmium-imidazole complex in the RKS/B3LYP model with basis 3-21G, using an initial guess obtained by diagonalising the core Hamiltonian matrix.}\label{fig:accumulation}
\end{figure}

\subsubsection{Local convergence behaviour} \label{sec:local}
In order to properly assess their local convergence properties, the CDIIS variants were combined with a globally convergent method and an improved initial guess, both provided by the PySCF package. More precisely, the EDIIS with a fixed-depth equal to 8 and an initial guess generated from a superposition of atomic density matrices were used for the experiments with the RHF model, while the ADIIS with a fixed-depth equal to 8 and the same type of initial guess were used for the experiments with the RKS model. In both cases, the switch between the ``global'' method and the CDIIS variants was made when the error vector norm was below $10^{-2}$. With such an initialisation, convergence was immediately observed in every test case. Figure~\ref{fig:setGeneral} presents the obtained results for the set of molecules already considered in Figure~\ref{fig:global} and a smaller set of parameter values. In such a setting, the restarted and the adaptive-depth CDIIS are shown to be more efficient than the fixed-depth one when adequately chosen values of their respective parameters are used.

\begin{figure}[htb]
	\centering
	\begin{subfigure}[b]{0.48\textwidth}
		\centering
		\includegraphics[width=\linewidth]{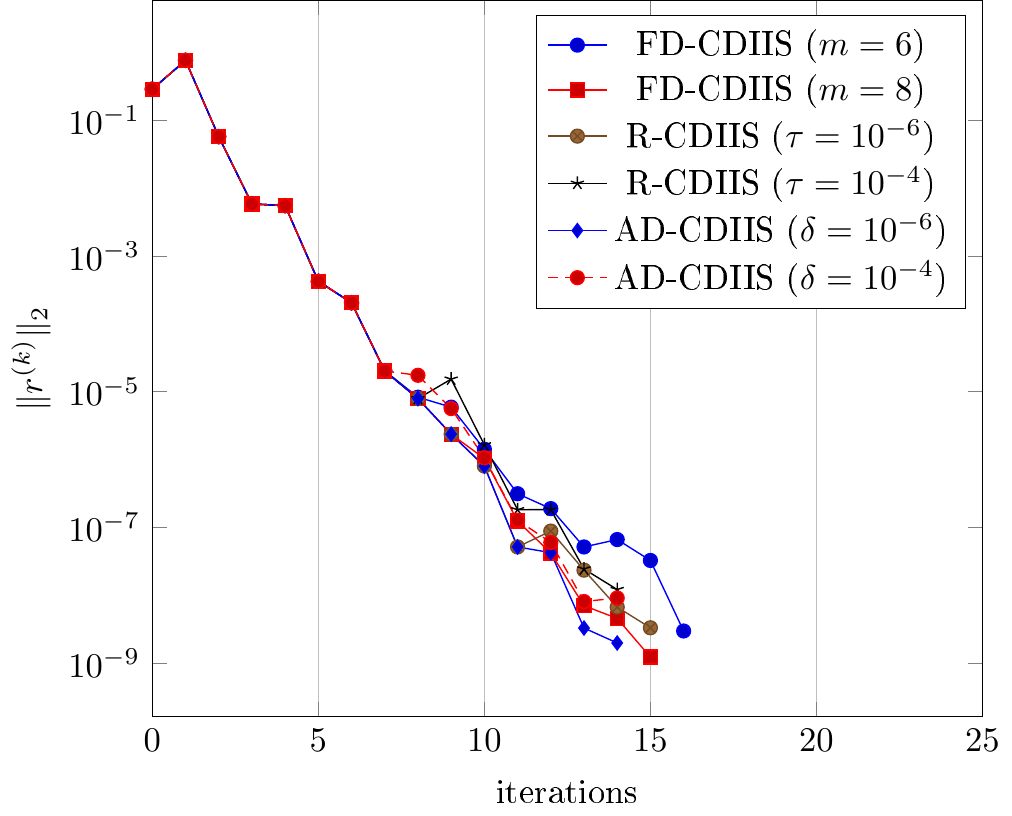}
		\caption{Cadmium-imidazole complex in the RKS/B3LYP model with basis 3-21G.}
	\end{subfigure}%
	\quad
	\begin{subfigure}[b]{0.48\textwidth}
		\centering
		\includegraphics[width=\linewidth]{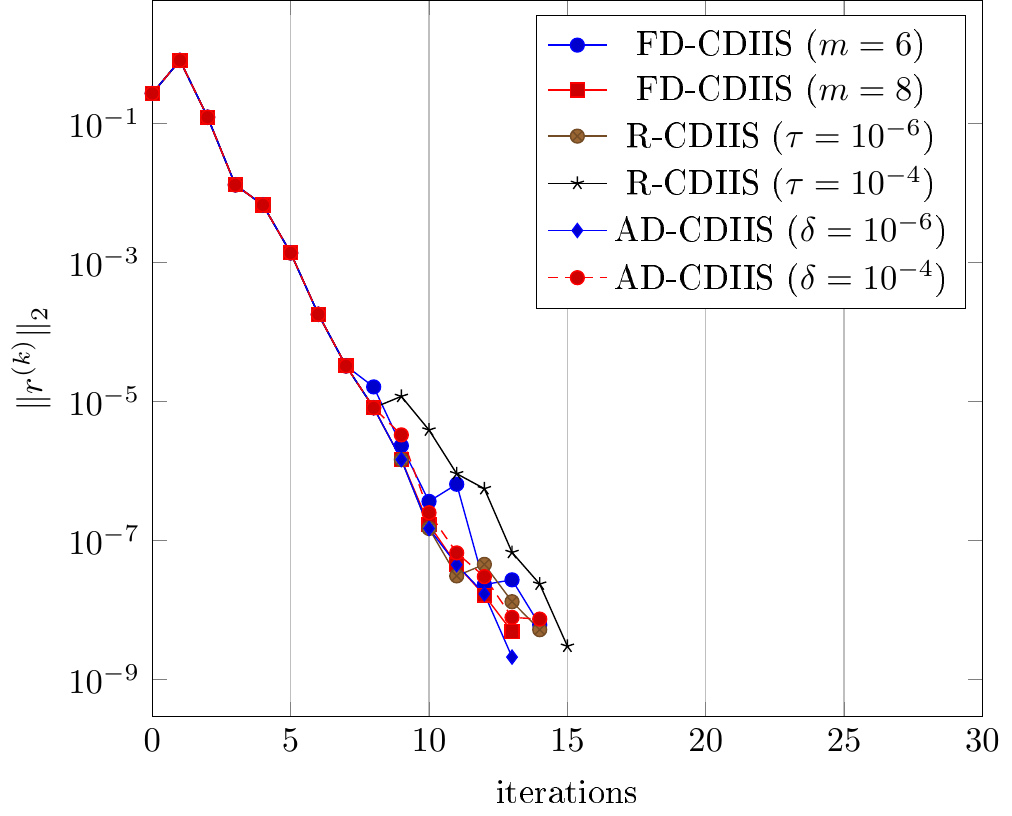}
		\caption{Glycine molecule in the RKS/B3LYP model with basis 6-31Gs.}
	\end{subfigure}
	\begin{subfigure}[b]{0.48\textwidth}
		\centering
		\includegraphics[width=\linewidth]{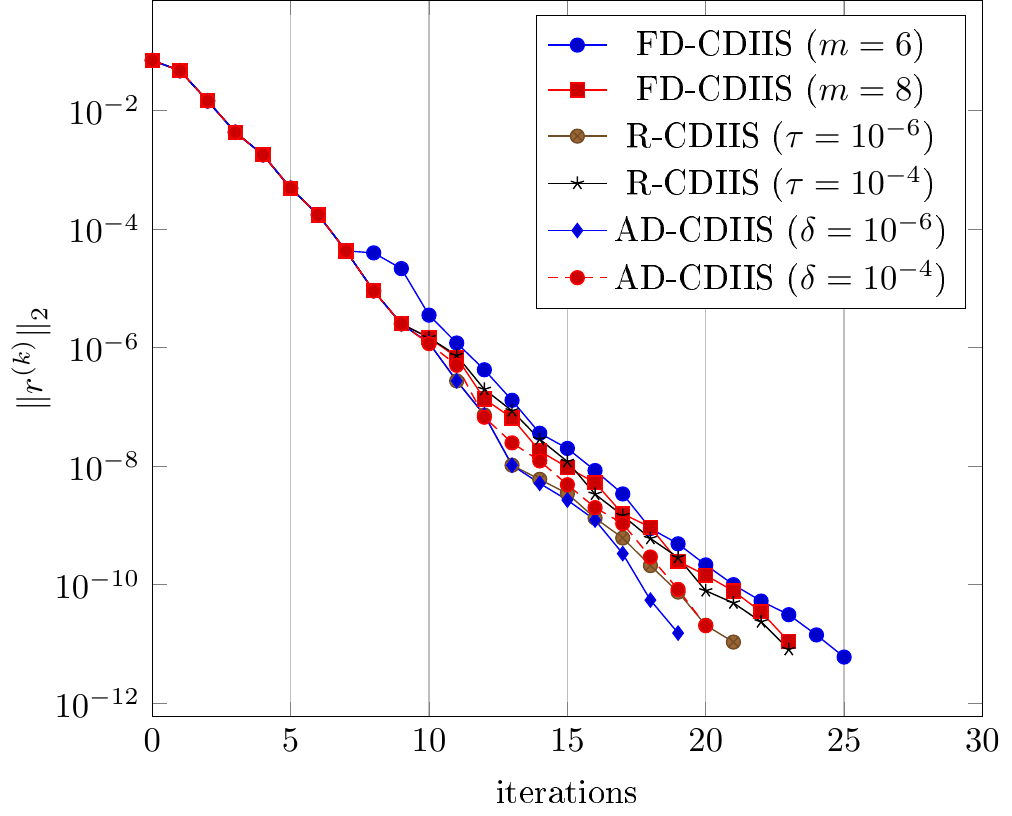}
		\caption{Dimethylnitramine molecule in the RHF model with basis 6-31G.}
	\end{subfigure}	
	\quad
	\begin{subfigure}[b]{0.48\textwidth}
		\centering
		\includegraphics[width=\linewidth]{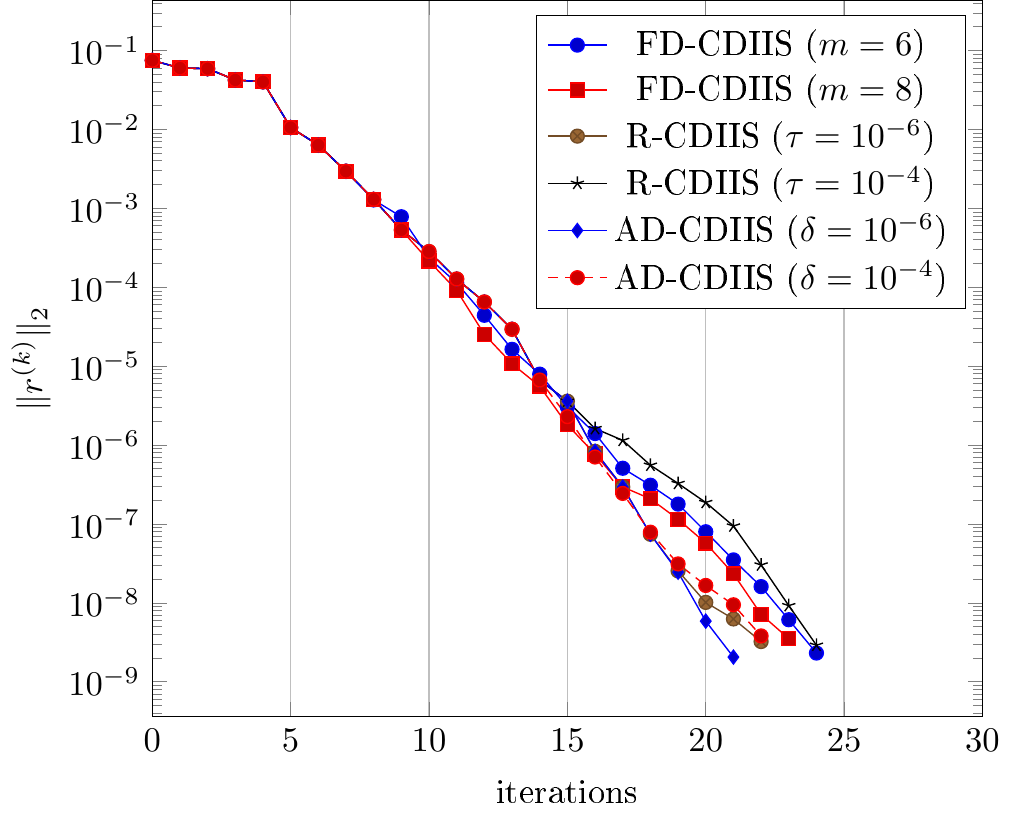}
		\caption{Galactonolactone molecule in the RHF model with basis 6-31G.}
	\end{subfigure}	
	\caption{Residual norm convergence for the fixed-depth, restarted and adaptive-depth CDIIS on different molecular systems using an initial guess provided by a globally convergent method.}\label{fig:setGeneral}
\end{figure}

Plotting the error vector norm with the corresponding depth during the course of the numerical experiments, as done in Figure~\ref{fig:convMk} for the dimethylnitramine and the glycine molecules, allows to observe that a restart occurs after a significant decrease of the error vector norm, as predicted by the theory for the restarted Anderson--Pulay acceleration. Unfortunately, one can also notice a slowdown in the convergence (or even a moderate increase of the error vector norm) just after a restart, thus motivating the introduction of an adaptive-depth mechanism which would not suffer from such a defect.

\begin{figure}[htb]
	\centering
	\begin{subfigure}[b]{0.48\textwidth}
		\centering
		\includegraphics[width=\linewidth,page=1]{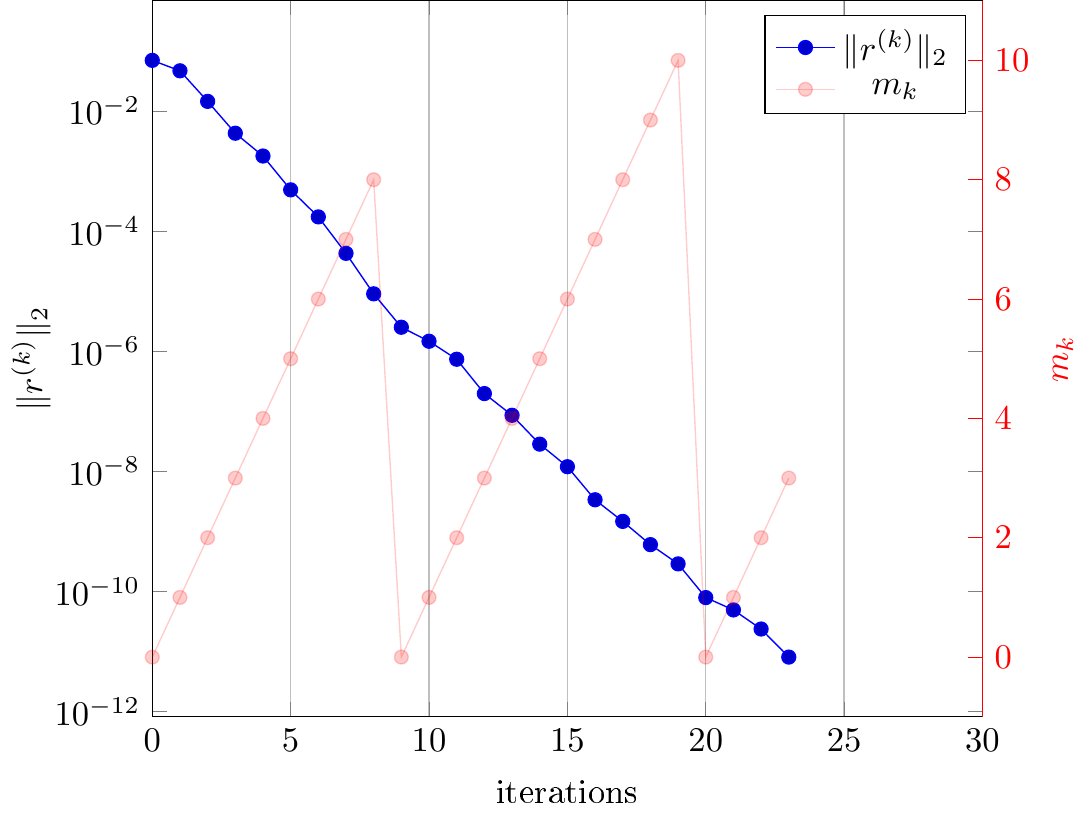}
		\caption{Restarted CDIIS with $\tau=10^{-4}$ for the dimethylnitramine molecule in the RHF model with basis 6-31G.}
	\end{subfigure}%
	\quad
	\begin{subfigure}[b]{0.48\textwidth}
		\centering
		\includegraphics[width=\linewidth,page=2]{DimethylRestartSliding.pdf}
		\caption{Adaptive-depth CDIIS with $\delta=10^{-4}$ for the dimethylnitramine molecule in the RHF model with basis 6-31G.}
	\end{subfigure}
	\begin{subfigure}[b]{0.48\textwidth}
	\centering
	\includegraphics[width=\linewidth,page=1]{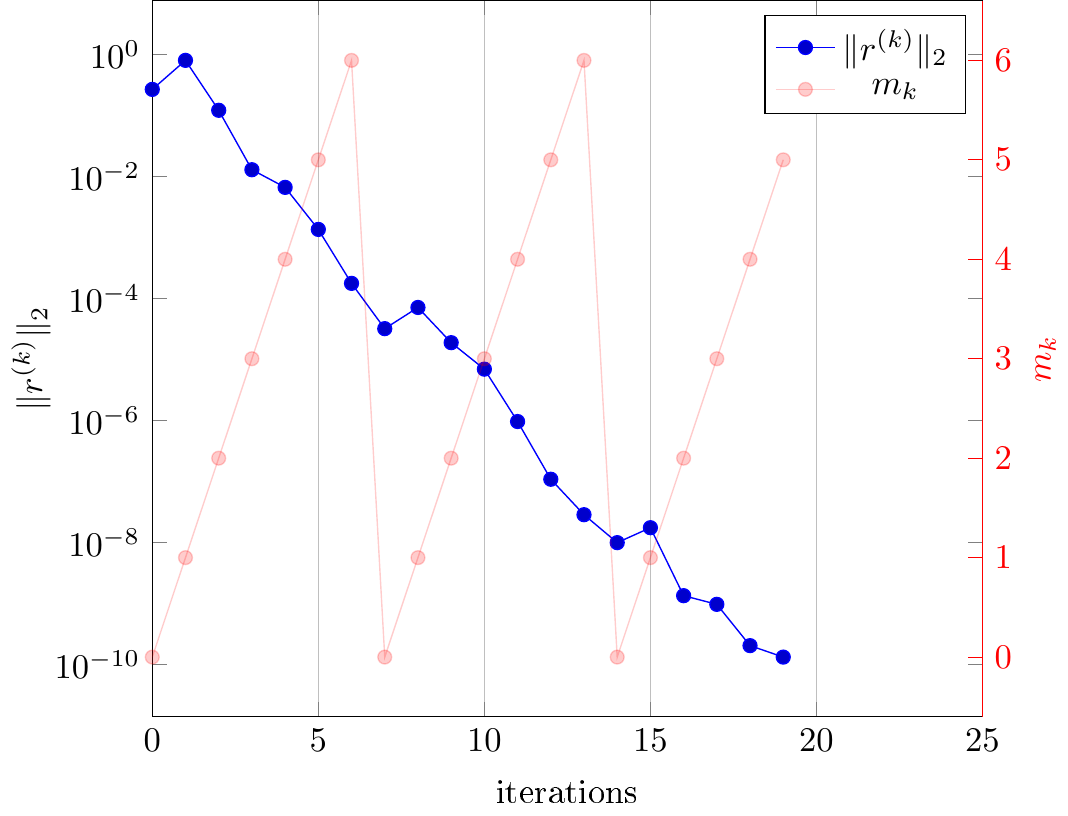}
	\caption{Restarted CDIIS with $\tau=10^{-4}$ for the glycine molecule in the RKS/B3LYP model with basis 6-31Gs.}
    \end{subfigure}%
    \quad
    \begin{subfigure}[b]{0.48\textwidth}
	\centering
	\includegraphics[width=\linewidth,page=2]{glycineRestartSliding.pdf}
	\caption{Adaptive-depth CDIIS with $\delta=10^{-4}$ for the glycine molecule in the RKS/B3LYP model with basis 6-31Gs.}
\end{subfigure}
	\caption{Residual norm convergence and corresponding depth for the restarted and adaptive-depth CDIIS on the dimethylnitramine and glycine molecules.}\label{fig:convMk}
\end{figure} 

\textbf{Mean depth.}
The cost of the CDIIS in terms of storage and computational resource at a given iteration is proportional to the value of the depth at this iteration. As a consequence, to properly compare the restarted and adaptive-depth variants with the classical fixed-depth CDIIS, we have computed the mean depth, denoted by $\bar{m}$, as the average value of $m_k$ during an experiment. Figure~\ref{fig:mk} presents the evolution of $\bar{m}$ with respect to the values of the restart parameter $\tau$ and the adaptive-depth $\delta$ for each of the molecular systems we considered. It is seen that $\bar{m}$ is a decreasing function of these parameters and that the two variants have on average lesser costs than their fixed-depth counterpart, while their performances are comparable or better.

\begin{figure}[htb]
	\centering
    \begin{subfigure}[t]{0.48\textwidth}
		\centering
		\includegraphics[width=\linewidth,page=2]{mkmean.pdf}
		\caption{Depth mean $\bar{m}$ as a function of the restart parameter~$\tau$.}
	\end{subfigure}%
	\quad
	\begin{subfigure}[t]{0.48\textwidth}
		\centering
		\includegraphics[width=\linewidth,page=1]{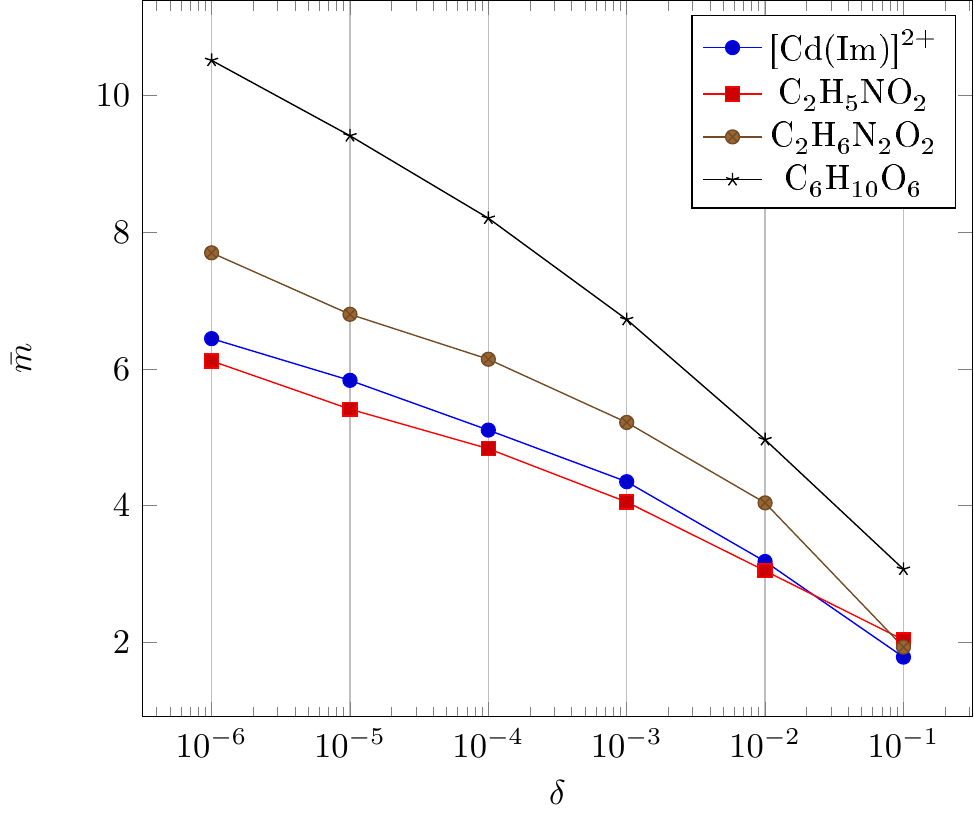}
		\caption{Depth mean $\bar{m}$ as a function of the adaptive-depth parameter~$\delta$.}
	\end{subfigure}
	\caption{Evolution of the depth mean for different molecular systems and models.}\label{fig:mk}
\end{figure}

\textbf{Rate of convergence.}
For each of the molecular systems considered in Figure~\ref{fig:setGeneral}, the practical rate of convergence of the restarted and adaptive-depth CDIIS was computed using a linear regression and plotted against the values of the parameters $\tau$ and $\delta$ in Figure~\ref{fig:slope}. As expected, it is apparent that this rate increases as the value of the parameter decreases, its evolution for the adaptive-depth variant being noticeably smoother. 

\begin{figure}[htb]
	\centering
	\begin{subfigure}[b]{0.48\textwidth}
    		\centering
    		\includegraphics[width=\linewidth,page=2]{slope.pdf}
    		\caption{Rate of convergence as a function of the restart parameter~$\tau$.}
    \end{subfigure} %
    \quad
	\begin{subfigure}[b]{0.48\textwidth}
		\centering
		\includegraphics[width=\linewidth,page=1]{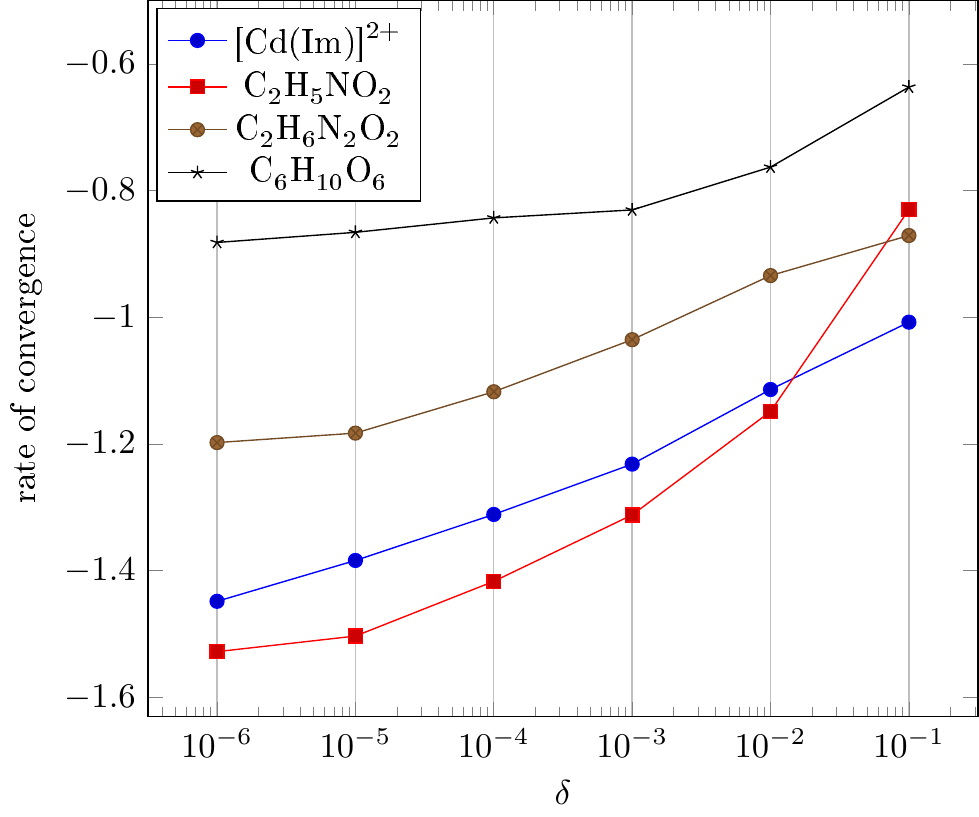}
		\caption{Rate of convergence as a function of the adaptive-depth parameter~$\delta$.}
	\end{subfigure}
	\caption{Evolution of the rate of convergence for different molecular systems.}\label{fig:slope}
\end{figure}

\textbf{Selecting values for the parameters.}
As previously mentioned, the decrease of the error vector norm becomes faster, and the average depth $\bar{m}$ increases, as the parameters $\tau$ and $\delta$ are decreased. However, a closer scrutiny of both Figures~\ref{fig:mk} and~\ref{fig:slope} reveals that the convergence rate appears to tend to an asymptotic value while the average depth grows at an almost constant rate. Thus, the gain of convergence by decreasing the parameters may not outweight the added computational cost of keeping a larger history of iterates. In this regard, a satisfying compromise was reached in our numerical tests by setting the values of both parameters at~$10^{-4}$.

\section{Conclusion}
Motivated by the DIIS and CDIIS techniques, respectively introduced by Pulay in 1980 \cite{Pulay:1980} and 1982 \cite{Pulay:1982}, and their relation with other extrapolation processes designed to accelerate fixed-point iteration methods -- one of them being the well-known Anderson acceleration -- we have considered a general and abstract class of acceleration methods and studied, theoretically and numerically, the local convergence properties of two of its instances: one allowing restarts, based on a condition initially introduced for a quasi-Newton using multiple secant equations \cite{Gay:1977,Rohwedder:2011}, and another one whose depth is continuously adapted according to a criterion that appears to be new.

Our main convergence results are obtained in a more general setting and rely on weaker assumptions than those existing in the literature for the DIIS \cite{Rohwedder:2011} or the Anderson acceleration \cite{Toth:2015}. First, we do not impose a direct relation between the function $g$ of the fixed-point iteration used to compute the solution and the error function $f$ used for the extrapolation. Second, the nondegeneracy hypothesis on the solution of the problem is weakened: it only involves the restriction of the function $f$ to a submanifold $\Sigma$ containing the range of the function $g$. Such generalisations are necessary in order to deal with the self-consistent field iterations in quantum chemistry, for which the DIIS and the CDIIS were originally introduced. As we already stated before, these results also cover the particular cases of the Anderson acceleration (for which $\Sigma=\mathbb{R}^n$ and $g=\id+f$) and of the DIIS if $\Sigma=\mathbb{R}^n$.

Another novelty of our work is the absence of assumption concerning the uniform boundedness of the extrapolation coefficients. Indeed, the proposed restart and adaptive-depth mechanisms allow us to \textit{a priori} prove such a bound. To our knowledge, the only other work where a similar estimate can be found is \cite{Zhang:2020}, in which a global linear convergence analysis for a stabilized variant of the type-I Anderson acceleration is given. As far as we know, the present article thus provides the first \emph{complete} proof of accelerated convergence for a family of extrapolation algorithms which includes instances of the DIIS, the CDIIS or the Anderson acceleration.

Finally, numerical experiments illustrate the good performances of the restarted and adaptive-depth acceleration algorithms applied to the numerical computation of the electronic ground state of various molecular systems. It has been observed that, with an adequate choice of their respective parameters, both variants exhibit a better convergence rate than their fixed-depth counterpart. In particular, the adaptive-depth variant shows good promise. It has also been noticed that acceleration occurs for values of the parameters several orders of magnitude larger than the theoretical estimates, and that the size of the set of stored iterates at each step is on average smaller than the fixed ``rule of thumb'' values generally found in implementations of the CDIIS. Understanding the reasons of this key fact will require further effort.

\section*{Acknowledgement}
The authors wish to warmly thank Antoine Levitt, for providing them with encouragements, useful references, and insightful remarks on a first draft of the manuscript, and Qiming Sun, for his help with the PySCF package. They also thank the anonymous reviewers whose comments and suggestions helped improve and clarify this manuscript.\\This project has received funding from the European Research Council (ERC) under the European Union's Horizon 2020 research and innovation programme (grant agreement MDFT No 725528 of Mathieu Lewin).

\printbibliography[heading=bibintoc]
\end{document}